\documentclass[bj,authoryear,noshowframe]{imsart}

\startlocaldefs

\theoremstyle{plain}
\newtheorem{theorem}{Theorem}[section]
\newtheorem{proposition}{Proposition}[section]
\newtheorem{lemma}[theorem]{Lemma}
\theoremstyle{remark}
\newtheorem{assumption}[theorem]{Assumption}
\newtheorem{remark}[theorem]{Remark}

\newcommand{\Var}{\textnormal{Var}} %

\newcommand{\trv}{\textnormal{TRV}}
\def\bR{\mathbb{R}}
\def\bP{\mathbb{P}}
\def\bE{\mathbb{E}}
\def\sF{\mathscr{F}}
\newcommand{\wh}{\widehat}
\newcommand{\wt}{\widetilde}

\newcommand{\toprob}{\overset{P}{\longrightarrow}}
\newcommand{\toDist}{\overset{\mathcal{D}}{\longrightarrow}}
\newcommand{\toDistSt}{\overset{st}{\longrightarrow}}

\endlocaldefs
\usepackage{booktabs}
        \usepackage[utf8]{inputenc}
         \DeclareMathAlphabet{\mathscr}{U}{BOONDOX-cal}{m}{n}
         \SetMathAlphabet{\mathscr}{bold}{U}{BOONDOX-cal}{b}{n}
         \DeclareMathAlphabet{\mathbscr} {U}{BOONDOX-cal}{b}{n}
\usepackage{color}

\begin{document}

\begin{frontmatter}
\title{Data-driven fixed-point tuning for truncated realized variations}
\runtitle{Fixed-point tuning for truncated realized variations}

\begin{aug}
\author[A]{\inits{B.~C.}\fnms{B. Cooper}~\snm{Boniece}\ead[label=e1]{cooper.boniece@drexel.edu}}
\author[B]{\inits{J.~E.}\fnms{Jos\'e E.}~\snm{Figueroa-L\'opez}\ead[label=e2]{figueroa-lopez@wustl.edu}}
\author[B]{\inits{Y.}\fnms{Yuchen}~\snm{Han}\ead[label=e3]{y.han@wustl.edu}}
\address[A]{Department of Mathematics, Drexel University\printead[presep={,\ }]{e1}}

\address[B]{Department of Statistics and Data Science, Washington University in St.~Louis\printead[presep={,\ }]{e2,e3}}
\end{aug}

\begin{abstract}
Many methods for estimating integrated volatility and related functionals of semimartingales in the presence of jumps require specification of tuning parameters for their use in practice.  In much of the available theory, tuning parameters  are assumed to be deterministic and their values are specified only up to asymptotic constraints.  However, in empirical work and in simulation studies, they are typically chosen to be random and data-dependent, with explicit choices often relying entirely on heuristics.    In this paper, we consider novel data-driven tuning procedures for the truncated realized variations of a semimartingale with jumps based on a type of random  fixed-point iteration.  Being effectively automated, our approach alleviates the need for delicate decision-making regarding tuning parameters in practice and can be implemented using information regarding sampling frequency alone.  We demonstrate our methods can lead to asymptotically efficient estimation of integrated volatility and exhibit superior finite-sample performance compared to popular alternatives in the literature.
\section{Introduction}

\end{abstract}

\begin{keyword}
\kwd{High-frequency data}
\kwd{integrated volatility estimation}
\kwd{semimartingales}

\end{keyword}

\end{frontmatter}

The continuous part of the quadratic variation of an It\^o semimartingale,  commonly known as the integrated volatility,  plays an outsize role in financial econometrics, and its estimation in various settings based on discrete observations has been a major focus in the literature at various points in the past 20+ years.   The semimartingale $X$ commonly represents the log-price of a financial asset,  and its integrated volatility serves as a measure of the overall uncertainty inherent in the continuous part of $X$ over a given time period.

Among the variety of available methods for integrated volatility estimation, the truncated realized variation (TRV),  introduced in \cite{mancini:2001}, was one of the first and remains among the most popular approaches to-date that is \textit{jump-robust}, in the sense that it can still provide reliable estimates of integrated volatility when jumps occur in the process $X$.    Other well known jump-robust methods for estimating integrated volatility include bipower variations and their extensions  \citep{barndorff-nielsen:2004,barndorff-nielsen:shephard:winkel:2006,corsi:pirino:reno:2010} or those based on empirical characteristic functions  \citep{todorov:tauchen:2012,jacod:todorov:2014,jacod:todorov:2018}, among others, giving the practitioner a wide array of choices at their disposal for estimation of integrated volatility in modeling contexts where jumps may be present.%

To choose an estimator among this array of options,  currently one must first decide between two distinct classes: either asymptotically efficient approaches, like TRV, which require selection of tuning parameters, or alternatively ``tuning-free'' estimators but at the unfortunate expense of asymptotic efficiency. 
 From the perspective of minimizing variance,  asymptotically efficient approaches are preferable, but their use in practice necessitates the critical additional step of specifying the tuning parameter values themselves.  This consequential step can significantly impact estimation performance, but current asymptotic theory does not offer direct guidelines for choosing parameters explicitly, which can be an extremely delicate matter in practice.  For instance, even in idealized asymptotic settings, appropriate choices often depend on a priori unknown properties of $X$ and can determine whether or not a given estimator retains even the basic requirement of consistency. In the absence of theoretically supported approaches for specifying explicit values of these parameters, the practical use of tuning-parameter-based methods remains entirely reliant on heuristics.  The purpose of the present work is to address this gap.

In the case of TRV,  the tuning parameter of importance is called the \textit{threshold},  denoted hereafter as $\varepsilon>0$, indicating a level above which increments are discarded from the estimation procedure.  Concretely, given a discretely observed semimartingale $X=\{X_t\}_{t\geq 0}$ at times $0=t_0<t_1<\ldots<t_n=T$, the TRV is defined as
\begin{align*}%
\trv_n(\varepsilon)=\sum_{i=1}^{n}\big(\Delta_{i}^{n}X\big)^{2}{\bf 1}_{\{|\Delta_{i}^{n}X|\leq\varepsilon\}},
\end{align*}
where  $\Delta_{i}^{n}X:=X_{t_{i}}-X_{t_{i-1}}$ is the $i^{th}$ increment of $X$, often assumed to be observed on a regular sampling grid, so that $t_i-t_{i-1}=:h_n$ for all $i$.  Statistical properties of TRV  have been extensively studied when $\varepsilon=\varepsilon(h_n)$ is a \textit{deterministic} function of the time step $h_n$ such that $\varepsilon(h_n)\to 0$ at specified rates as $h_n\to 0$.   In \cite{mancini:2009}, when either the jump component of the process $X$ is of finite activity or is a pure-jump L\'evy process with infinite jump activity, TRV was shown to be consistent whenever 
\begin{align}\label{eq:mancini_threshold_cond}
\lim_{h_n\to 0} \varepsilon(h_n) = 0, \quad \mbox{and} \quad  \lim_{h_n\to 0} \frac{\varepsilon(h_n)}{\sqrt{h_n\log \frac{1}{h_n}}} = \infty.
\end{align}
A consistency statement for TRV encompassing a broader class of semimartingales was given in \cite{jacod:2008}, but for the more restrictive case of power thresholds, namely, thresholds of the form
\begin{align}\label{e:powerthreshold}
\varepsilon(h_n)=ch_n^\beta, \qquad c>0,\quad 0<\beta<1/2.
\end{align}
Under finite jump activity, central limit theorems for TRV were established under the threshold constraint \eqref{eq:mancini_threshold_cond} in \cite{mancini:2009}; in the infinite-activity case, they were established in \cite{jacod:2008} for more general semimartingales based on thresholds satisfying \eqref{e:powerthreshold} under the additional assumption the volatility itself is a semimartingale, and also in \cite{cont:mancini:2011,mancini:2011} for general c\`adl\`ag volatility processes but for L\'evy-type jump behavior, both under additional constraints on $\varepsilon$ related to the Blumenthal-Getoor index of $X$.

While asymptotic constraints such as \eqref{eq:mancini_threshold_cond} and \eqref{e:powerthreshold}  may be informative for threshold selection,  they do not concretely indicate how one should make an explicit choice for $\varepsilon$ in a given context. Moreover, even if a particular deterministic choice for $\varepsilon$  may lead to good estimation performance under a given model, the same choice of $\varepsilon$ under a perturbed version of the same model can lead to dramatically worse estimation performance.  To illustrate this point, the left panel of Figure \ref{fig1}, below, shows histograms of the relative estimation errors for TRV using a fixed, deterministically chosen threshold value under two different parameter settings of the same model. While  TRV performs satisfactorily with this deterministic threshold value under one of the parameter settings, it performs poorly with the same threshold value under alternate parameter settings, even though the expected quadratic variation of $X$ is the same in both cases. In contrast, the right panel of Figure \ref{fig1} displays histograms of relative estimation errors for the approach developed in this paper, where satisfactory performance is maintained across both settings. %

Though Monte Carlo studies or empirical insights may help in choosing the value of $\varepsilon$ deterministically in a given setting, an arguably more natural approach is to select thresholds through some data-driven procedure, permitting the threshold itself to depend on observed data.  Indeed, random, data-driven parameter tuning is often done in numerical studies in the literature -- without theoretical support -- to illustrate finite-sample behavior of estimators and  to improve their numerical performance.%

\begin{figure}[h]
\begin{center}
\begin{minipage}{0.38\linewidth}
\begin{center}
\includegraphics[width=\linewidth]{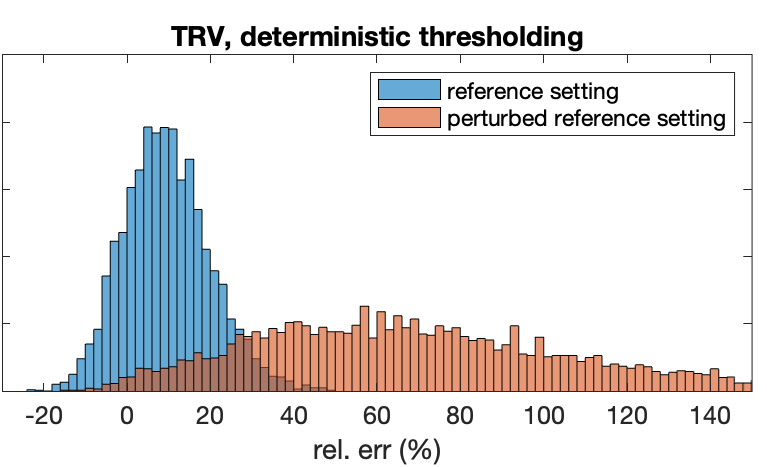}
\end{center}
\end{minipage}
\qquad
\begin{minipage}{0.38\linewidth}
\begin{center}
\includegraphics[width=\linewidth]{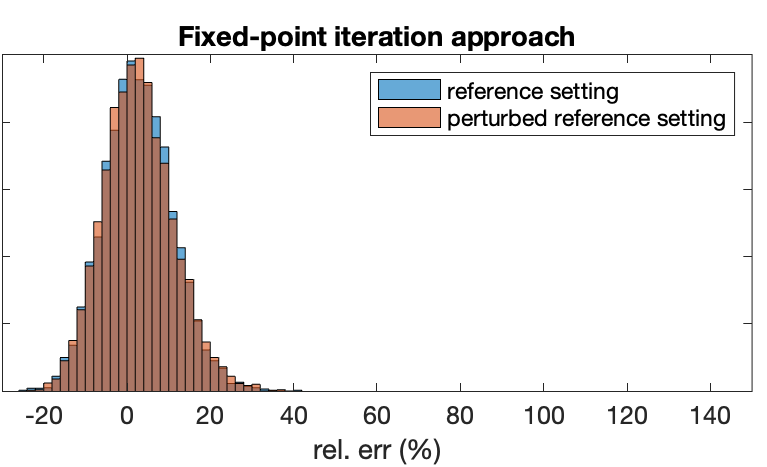}
\end{center}
\end{minipage}
\end{center}
\caption[Fig 1]{Sampling distributions of the relative estimation error for TRV with deterministic thresholding at a fixed threshold value $\varepsilon$ (left panel) versus our automated thresholding procedure\footnotemark~(right panel).   The average realized quadratic variation is the same in both models.  The reference setting is the same as Model 1 as described in Section \ref{s:montecarlo}, based on a 1-week time horizon at a sampling frequency of 5 minutes. The perturbed reference setting uses the same parameter settings as the reference setting, but with half overall (average) volatility level and roughly twice the rate of its finite jump activity component (adjusted to match the average quadratic variation of both models).\label{fig1}}
\end{figure}
\footnotetext{Specifically our approach as described in (6b) in Section \ref{s:montecarlo}, though similar behavior holds in all other cases.}
However,  by their very nature, data-driven parameter selection procedures introduce considerable statistical dependencies and associated theoretical challenges that are otherwise absent when parameters  are chosen deterministically. Consequently, despite the practicality and potential benefits of data-driven parameter selection, the literature on TRV and related methods employing data-driven tuning procedures has remained relatively scarce. For instance, in the case of finite activity jumps, it was stated without proof in a remark in \cite{mancini:reno:2011} that consistency holds for time-dependent random thresholds  (possibly different for each increment $\Delta_i^n X$) of the form $c_{t_i}\varepsilon$, where $\varepsilon=\varepsilon(h_n)$ satisfies \eqref{eq:mancini_threshold_cond} and $\{c_t\}_{t\geq 0}$ is a stochastic process that is a.s.~bounded above and bounded away from $0$. Later, in \cite{figueroa-lopez:mancini:2019}, consistency was rigorously established under finite jump activity for possibly data-dependent  time-varying thresholds of the type  $\sqrt{2(1+\eta)M_i h_n \log (1/h_n)}$, where $\eta>0$ and $M_i$ are random variables satisfying  $M_i\in[ \inf_{s\in [t_{i-1}, t_i]} \sigma_s^2,\, \sup_{s\in [0,T]} \sigma_s^2 ]$ a.s. To the authors' knowledge, these statements comprise the totality of asymptotic theory for TRV with data-driven thresholds, and there is currently no theoretical support in the literature for data-driven parameter tuning of TRV outside consistency statements in the finite activity setting.

Moreover, in spite of the considerable focus on asymptotic properties of TRV with the threshold constraints \eqref{eq:mancini_threshold_cond} and \eqref{e:powerthreshold}, recent work \citep{figueroa-lopez:mancini:2019,figueroa-lopez:nisen:2013} has demonstrated that certain \textit{optimal} choices of threshold do not satisfy these asymptotic conditions, leaving a substantive gap in the available asymptotic theory even within the scope of deterministic thresholding.  Optimal-type thresholds can lead to substantial gains in finite sample estimation performance, %
and their explicit expressions can serve as a more direct guideline for threshold selection, making them ideal choices for practitioners.  However, their direct use, even to first-order approximation, is complicated by the fact that they depend on the volatility itself. For instance, under an idealized constant volatility assumption and general finite jump activity,  the MSE-optimal threshold $\varepsilon_n^\star$ admits the approximation: %
\begin{equation}\label{e:estar}
\varepsilon^{\star}_{n} \sim \sqrt{2\sigma^2 h_n\log \tfrac{1}{h_n}}, \quad \mbox{as } h_n\to 0,
\end{equation}
where $\sigma>0$ is the volatility. Under L\'evy stable-like infinite jump activity, the MSE-optimal threshold is the same as $\varepsilon^{\star}_n$ up to an additional multiplicative constant depending on the Blumenthal-Getoor index  \citep{figueroa-lopez:gong:han:2022}. 
Though this expression cannot be used directly in practice due to its dependence on knowledge of the volatility, it lends itself naturally to fixed-point iterative procedures, as suggested in \cite{figueroa-lopez:nisen:2013} and  \cite{figueroa-lopez:mancini:2019}, whose asymptotic theory has remained unestablished, until now.%

 In this work, we consider two classes of iterative procedures for jump-robust estimation of the integrated volatility based on data-driven parameter tuning.  Our procedures are designed to turn the otherwise infeasible threshold \eqref{e:estar} into a feasible one, and will be seen to stem from solutions $\xi$ to random fixed-point equations of the type
\begin{equation}\label{e:fixedpoint}
\xi =\Phi_n\left( (\Delta_1^n X){\bf 1}_{\{|\Delta_1^n X|\leq  \sqrt{r_n  \xi }\}},\ldots ,(\Delta_n^n X) {\bf 1}_{\{|\Delta_n^n X|\leq  \sqrt{r_n\xi}\}}\right),
\end{equation}
for an appropriate function $\Phi_n$ and sequence $r_n\to0$.  Viewed as random, data-dependent thresholds, our procedures extend the asymptotic theory beyond deterministic thresholding to accommodate automatic, data-driven calibration of TRV, and further extends current  asymptotic theory beyond the general rate constraints imposed in \eqref{eq:mancini_threshold_cond} and \eqref{e:powerthreshold}, allowing for time-dependent thresholding, ultimately leading to substantial gains in finite-sample performance and more principled threshold selection procedures.  Part of our analysis is based on relating our proposed iterative estimators to oracle-like sequences of estimators;  this general approach may be of use for parameter tuning  in other jump-robust methods in the  literature.

This paper is organized as follows. Section \ref{trvSec:Model} introduces the model, estimation framework and some notation used throughout the paper. Section \ref{Sec:DebiasMthd} contains our main results, including instances of uniform and time-varying thresholding, and Section \ref{s:montecarlo} contains some numerical illustrations concerning finite-sample estimation performance. The proofs of the main results and auxiliary lemmas are given in Appendices \ref{Sec:proof} and \ref{Sec:proofs_of_lemmas}.

\section{Framework and setting}\label{trvSec:Model}

 We consider a 1-dimensional It\^o semimartingale  $X=(X_{t})_{t\geq 0}$ defined on a complete filtered probability space  $(\Omega,\sF,(\sF_{t})_{t\geq 0},\bP)$  of the form
\begin{align}\label{model:con_fija}
    dX_t= b_tdt + \sigma_tdW_t + \gamma_tdL_t + dJ_t,\quad t\geq 0.
\end{align}
Above,  $W$ is a standard Brownian motion, $b,\gamma,\sigma$ are c\'adl\'ag   adapted,  $L=\{L_t\}_{t\geq 0}$ is a pure-jump infinite-activity L\'evy process,  and  $J=\{J_t\}_{t\geq 0}$ 
is a general pure-jump process with finite jump activity. We refer to Assumption \ref{assump:Coef0} for complete conditions on all driving processes and coefficients.

We suppose that on a fixed and finite time interval $[0,T]$, $n$ observations, $X_{t_1}, X_{t_2}, \ldots, X_{t_n}$, of the continuous-time process $X$ are available at known times $0=t_0<t_1<\ldots<t_n=T$.  We assume sampling times are evenly spaced, and denote the time step between observations as $h_n:=T/n$.   Our estimation target is the integrated volatility (or integrated variance) of $X$ defined as
$$
C_T :=\int_0^T\sigma_s^2ds.
$$

We consider two classes of estimators of $C_T$. %
The first class of estimators we consider are based on an iterative scheme that proceeds as follows:
\begin{enumerate}
\item At the start, an initial guess $\widehat{C}_{n,0}$ for the integrated volatility $C_T$ is first put forward. This initial estimate should ideally be free of tuning parameters; for instance,
 the realized variance (RV) $\widehat{C}_{n,0}= \sum_{i=1}^n (\Delta_i^n X)^2$ or bipower variation $\widehat{C}_{n,0}=\frac{\pi}{2}\sum_{i=1}^{n-1} |\Delta_i^n X||\Delta_{i+1}^n X|$, among other possibilities. We refer to Theorem \ref{thm:consistency_clt_con} and Remark \ref{CmntInitCa} for further information about our required conditions on $\widehat{C}_{n,0}$. This initializes a data-dependent threshold of the form  $B_{n,0} = \big(r_n T^{-1}\widehat C_{n,0}\big)^{1/2}$, where the \textit{threshold rate} $r_n=r_n(h_n)$ is some  predetermined deterministic function of $h_n$. 
 \item Then, an iterative  sequence of thresholds $B_{n,j}$ and estimators $\widehat{C}_{n,j}$ is constructed based on the relations
\begin{align}%
\label{IMDE0}
\begin{split}
     {B}_{n,j-1} &:=\sqrt{r_n T^{-1}  \widehat C_{n,j-1}}, 
     \\
     \widehat C_{n,j} &:= \sum_{i=1}^n (\Delta_i^n X)^2 {\bf 1}_{\{|\Delta_i^n X|\leq  B_{n,j-1}\}},\quad j\geq1.%
\end{split}
\end{align}
For each fixed $n$, the sequence $\{\widehat C_{n,j}, j\geq 0\}$ will be shown to  
always ``stabilize'' %
in the sense that the index
\begin{equation}\label{e:def_jn}
	j_n:=\min\big\{j\geq{}0: \widehat C_{n,j} =\widehat C_{n,j+\ell}\text{~ for all ~}\ell\geq 0\big\} ,%
\end{equation}
always exists. With regards to \eqref{e:fixedpoint}, the above scheme can be viewed as a fixed-point iteration for the mapping $\xi \mapsto \sum_{i=1}^n (\Delta_i^n X)^2 {\bf 1}_{\{(\Delta_i^n X)^2\leq  r_nT^{-1}\xi \}}$.  
\item  Denoting  $B_n:=B_{n,j_n}$, we then define the uniform thresholding estimator
\begin{equation}\label{DfnME0}
	 \widehat C_n:=  \widehat C_{n,j_n}=  \sum_{i=1}^n (\Delta_i^nX)^2 \,{\bf 1}_{\{|\Delta_i^n X|\leq  B_n\}}.
\end{equation}
\end{enumerate}
The second class of estimators we consider are based on time-varying (or local) thresholding, namely estimators $\widehat C^*_n$ of the type
\begin{equation} \label{e:C_star}
\widehat C^{*}_{n} = \sum_{i=1}^n (\Delta_i^nX)^2 \,{\bf 1}_{\{|\Delta_i^n X|\leq  B^*_n(i)\}},
\end{equation}
for appropriate  data-driven local thresholds $B^*_n(i)$, $i=1,\ldots, n$, arising from fixed-point equations, whose precise definition is deferred to Section \ref{s:locthresh}.
The central focus of this work is to study  the classes of estimators defined by \eqref{DfnME0} and \eqref{e:C_star}.  Below we state our main assumptions relating to the model \eqref{model:con_fija}.

\begin{assumption}\label{assump:Coef0} \hfill
\begin{enumerate}
	\item[(i)] $\sigma,\gamma,b$ are c\`adl\`ag adapted;  $\inf_{0\leq t \leq T}\sigma_t>0;$

\item[(ii)] $L$ is a L\'evy process with characteristics $(0,0,\nu)$  (with respect to the truncation function ${\bf 1}_{\{|x|\leq 1\}}$, see \cite{sato:1999}) such that,  for some $\alpha\in(0,2)$ and $K_{\pm}\in[0,\infty)$ with $K_+\vee K_->0$,
$$
\lim_{x\to 0^+} x^\alpha\nu((x, \infty))=K_+, \quad \lim_{x\to 0^-} x^\alpha\nu((-\infty,x))=K_-;$$
$J$ is a general finite-activity jump process of the form $J_t=\sum_{i=1}^{N'_t} \xi_i$, where $\{\xi_i\}_{i\geq 1}$ satisfy $\bP(\xi_{i}\neq{}0)=1$, and $\{N'_t\}_{t\geq 0}$ is a non-explosive counting process; $W$ is a Brownian motion independent of $L$.  All processes are adapted.

	\item[(iii)]  There is a sequence $\tau_n$ of stopping times increasing to infinity and a positive sequence $K_n$ such that  %
	\begin{equation}\label{e:stopping_time_condition}
t\leq \tau_n \implies \begin{cases}
|b_{t}|+ |\sigma_t|  + |\gamma_t| \leq K_n,\\
\bE( |\gamma_{t+s }-\gamma_{t}|^2|\mathcal F_t)\leq K_n s^{},\quad s>0.
\end{cases}
\end{equation}
\end{enumerate}
\end{assumption}
Our assumptions, in particular, do not require $\sigma$ to be a semimartingale, which is important in rough volatility modeling. Note that the condition on $\gamma$  in \eqref{e:stopping_time_condition} is satisfied  whenever $\{\gamma_t\}_{t\geq{}0}$ is an It\^o semimartingale with locally bounded characteristics.

 We use the following standard notation throughout the paper: for two sequences $a_n,b_n>0$,
\begin{itemize}
	\item $a_n\sim b_n$ means that $a_n/b_n \to 1$;
	\item $a_n\ll b_n$ means that $a_n=o(b_n)$, i.e., $\lim_{n\to\infty}a_n/b_n=0$; $a_n \lesssim b_n$ means $a_n=O(b_n)$,  i.e., $\limsup_{n\to\infty} a_n/ b_n<\infty$;
	\item $a_n\gg b_n$ means that $b_n=o(a_n)$;  $a_n \gtrsim b_n$ means $b_n=O(a_n)$;
	\item 
$\toprob$ denotes convergence in probability;
\item 
$\toDist$ denotes convergence in law;
\item 
$\toDistSt$ denotes stable convergence in law.
\end{itemize}

\section{Main results}\label{Sec:DebiasMthd}
\subsection{Uniform thresholding}\label{s:uniformthresh}
In this section, we study the asymptotic properties of the estimator $\widehat C_n$ introduced in Section \ref{trvSec:Model}. 
As stated in the introduction, in \cite{figueroa-lopez:mancini:2019} it was shown that the first-order asymptotic behavior of the MSE-optimal threshold $\varepsilon_n^\star$ under the idealized assumption of constant volatility takes the form \eqref{e:estar}, which cannot be implemented feasibly in practice as it depends on knowledge of the volatility itself. However, exploiting this relationship is the  driving principle behind the iterative algorithm leading to the estimators $\widehat C_n$.  The proposed method can be seen as a natural mechanism to make such a threshold feasible  by taking the sequence $r_n=r(h_n)=2h_n\log(1/h_n)$  in the iterative procedure \eqref{IMDE0}. %
  As we will see, our iterative approach will enable us to asymptotically ``attain'' the infeasible threshold $\varepsilon^\star_n$, in principle rendering near-MSE-optimal behavior possible in practice.

In general,  it is a nontrivial task to establish asymptotic properties of the $\widehat C_{n}$ defined in \eqref{DfnME0}, even drawing upon results from the existing literature,  which has almost exclusively focused on deterministic uniform thresholding. 
For instance, in spite of the fact that $\widehat C_n$ satisfies the random fixed-point equation
\begin{equation*}\label{FPEFME0}
	\widehat C_n=
	\sum_{i=1}^n (\Delta_i^n X)^2 {\bf 1}_{\{(\Delta_i^n X)^2\leq  r_n T^{-1}\widehat C_n \}},
\end{equation*}
such an expression offers little insight into finding closed-form expressions for $\widehat C_n$. %

The central idea in our approach rests on relating the sequence of iterates $\widehat C_{n,j}$  to an iterative sequence of ``oracle-like'' estimators $\widetilde C_{n,j}(y_n)$ that make use of the (unknown) location of jumps of size $y_n>0$ or larger, where $y_n\to 0$ at an appropriate rate.  More concretely, for each $y\in(0,1)$, by virtue of the L\'evy-It\^o decomposition of $L$, we may reexpress
\begin{align}\label{eq:L_decompose}
b_tdt + \gamma_t dL_t& = \Big(b_t+ \gamma_t \int_{\{y<|x|\leq 1\}}x \nu(dx) \Big)dt + \gamma_t \int_{\{|x|\leq y \}} x \wt \mu(dx,dt) + \gamma_t\int_{\{|x|>y \}} x \mu(dx,dt) \notag\\
&=: {db_t(y)} + \gamma_t d M_t(y)+ \gamma_t  d H_t(y),
\end{align}
where $\mu$ is the jump measure of $L$ with intensity $\nu (dx)dt$, and $\wt \mu(dx,dt)= \mu(dx,dt)-\nu(dx)dt$ is the corresponding compensated jump measure.  Above, $H_t(y)$ is a compound Poisson process with finite jump activity satisfying $ H_t(y)= \sum_{i=1}^{N_t(y)}\zeta_i(y)$, where $N_t(y)$ is a Poisson process with rate $\lambda(y)= \int_{|x|>y} \nu(dx)$, and $\{\zeta_i(y)\}_{i\geq 1}$ are i.i.d. and supported on $(-\infty,y)\cup(y,\infty)$ with distribution
$\frac{{\bf 1}_{\{|x|\geq y\}}\nu(dx)}{\nu(|x|>y)}$.
For each $y>0$, we first define the random set
\begin{equation}\label{e:I_n(y)}
\mathcal{I}_n(y)=\{i:  \Delta_i^n N(y) = 0,\, \Delta_i^n N' = 0\},
\end{equation}
 which consists of all indices corresponding to intervals where no ``large" jumps have occurred.
For a sequence $y=y_n\to 0$, we then define an 
oracle analog of TRV that  eliminates any increments corresponding to time intervals in which ``large'' jumps of $X$ occur:
\begin{align}
    \label{e:oracle(y)}
\mathscr C_{n}(y) &:= \sum_{i=1}^n(\Delta_i^nX)^2{\bf 1}_{\{\Delta_i^n N(y) = 0,\Delta_i^n N' = 0\}}=\sum_{i\in\mathcal I_n(y)} (\Delta_i^n X)^2.
\end{align}
To connect $\mathscr C_n(y)$ with $\widehat C_n$, we then construct an iterative sequence $\{\wt C_{n,j}(y)\}_{j \geq 1}$,  analogous to \eqref{IMDE0}, by  setting 
$\wt C_{n,0}(y):=\widehat{C}_{n,0}$ (so that the oracle sequence has the same initial value as the original sequence $\widehat C_{n,j}$) and 
recursively  define, for $j\geq{}1$,
\begin{align}\label{def:C_tilde_y}
\begin{split}
    \wt{B}_{n,j-1}(y) &:=\sqrt{r_n  T^{-1}\wt C_{n,j-1}(y)},\\
    \wt C_{n,j}(y)&:= \sum_{i\in\mathcal I_n(y)} (\Delta_i^nX)^2 \,{\bf 1}_{\{|\Delta_i^n X|\leq  \wt{B}_{n,j-1}(y)\}}.
    \end{split}
\end{align}
Though the variables $\widetilde C_{n,j}(y_n)$   and $\mathscr C_{n}(y)$ are not feasible estimators themselves, their asymptotic behavior in fact completely determines that of $\widehat C_n$ provided the auxiliary sequence $y_n$ tends to 0 at an appropriate rate.  In our arguments, we demonstrate that
$$
   \widetilde C_{n,n+1}(y_n) \leq \widehat C_n\leq    \mathscr C_{n}(y_n)+ R_n,
$$
for an appropriate asymptotically negligible remainder $R_n$. The above relation allows us to analyze $\widehat C_n$ in terms of the array of oracle iterates $\{\widetilde C_{n,j}(y_n)\}_{{ j\geq 1}}$ and the oracle itself $\mathscr C_{n}(y_n)$.  We then demonstrate that $\{\widetilde C_{n,j}(y_n)\}_{ j\geq 1}$ are all asymptotically equivalent to the oracle $\mathscr C_n(y_n)$ (Proposition \ref{prop:allequal}); effectively reducing the problem to the analysis of $\mathscr C_n(y_n)$, which is considerably simpler.

We now proceed to describe the class of initial estimators  we consider in our procedure.  Apart from some mild regularity conditions, they are required only to be consistent for $C_T$ when the underlying process is continuous, allowing for a great deal of flexibility in the choice of initialization.
More specifically, for a generic process $Y$, let
\begin{equation}\label{e:def_C0}
 \widehat C_{n,0}(Y)=  \sum_{i=1}^{n-d+1} F(\Delta_i^n Y,\ldots,\Delta_{i+d-1}^n Y),%
 \end{equation}
 where $F:\bR^d\to [0,\infty)$ satisfies, for some $\delta_0 \in (0,2]$ and for all $\mathbf x,\mathbf y\in \bR^d$ with $\|\mathbf y\|\vee\|\mathbf x\|\leq 1,$
 \begin{align}
 F(\mathbf x) &\leq K \|\mathbf x\|_\infty^2\label{e:F_maxbound},\\
 \big|F(\mathbf x+\mathbf y)-F(\mathbf x) \big| &\leq  K (\|\mathbf y\|_\infty\wedge \|\mathbf x\|_\infty)^{\delta_0}(\|\mathbf y\|_\infty^{2-\delta_0}+ \|\mathbf x\|_\infty^{2-\delta_0}),\label{e:Lipbound}
 \end{align}
 for some $K<\infty$.  Above, $\|\mathbf x\|_\infty=\max_{1\leq{}i\leq{}d}|x_{i}|$. An initial estimate $\widehat C_{n,0}$ is said to belong to class $\mathcal C$ if $\widehat C_{n,0}=\widehat C_{n,0}(X)$, where $\widehat C_{n,0}(\cdot)$ is given by \eqref{e:def_C0}, and satisfies%
\begin{equation}\label{e:C0_consistent_cont_case}
\widehat C_{n,0}\big(\sigma\! \cdot\! W\big) \toprob \int_0^{T}\sigma_s^2 ds,
\end{equation}
 where $(\sigma\! \cdot\! W)_{t}:=\int_0^t \sigma_s dW_s$. 
 
\begin{remark}\label{CmntInitCa} The class $\mathcal C$ includes, for instance,  the ordinary realized variance (i.e., $F(x)=x^2)$, multipowers of the type $
F(x_1,\ldots,x_d)\propto \prod_{i=1}^{d} |x_i|^{r_i}$ with $r_i\geq 0$ and $r_1+\ldots+r_d=2$, and the nearest-neighbor truncation estimators of \cite{andersen:dobrev:schaumburg:2012}, among other possibilities.
\end{remark}
We now state our first main result.

\begin{theorem}\label{thm:consistency_clt_con} Suppose $\widehat C_{n,0}$ belongs to class $\mathcal C$, and that the sequence $r_n\to 0$ satisfies
\begin{equation}\label{e:rn_rate}
r_n\gg h_n\log (1/h_n).
\end{equation}
Then, the following assertions hold:
 \begin{enumerate}
    \item[(i)] If, further 
    $r_n\ll \left(h\log (1/h)\right)^{\frac\alpha2}$,
    then, as $n\to\infty$,
    $$ \widehat C_{n}\toprob C_T.$$%
    \item[(ii)] If $\alpha \in (0,1)$ and further
$ r_n \ll h_n^{\frac{\alpha+1}{2}} \left(\log (1/h)\right)^{\frac\alpha2},$
    then,  as $n\to\infty$,
 $$
 \frac{1}{\sqrt{h_n}}\left(\widehat C_n - C_T\right)\toDistSt \mathcal{N}\left(0,2\int_0^T\sigma^4_sds\right).%
 $$

    \item[(iii)] Suppose $\alpha \in (1,2)$ and $\gamma_t\equiv 1$.  Then,  as $n\to\infty$, for any choice of $r_n$ satisfying \eqref{e:rn_rate},
 $$
 \frac{1}{\sqrt{h_n}}\big(\widehat C_n - C_T\big)\toprob \infty.%
 $$
     \end{enumerate}
\end{theorem}
Observe that the upper bounds on $r_n$ in (i)--(ii) above depend on the jump activity index $\alpha$ and become more restrictive as $\alpha$ increases. Bearing this in mind, we make the following remarks.
\begin{remark}
A common choice in the literature {for both benchmarking estimation performance and applications} is $\sqrt{r_n}= 4 h_n^{0.49}$ (e.g., \cite{jacod:todorov:2014,li:todorov:tauchen:2017,ding:li:liu:zheng:2023}), which leads to efficient CLTs across nearly the entire range $\alpha<1$;  Section \ref{s:montecarlo} compares the performance of this threshold choice for TRV against our iterative methods.  Though the hypotheses of Theorem \ref{thm:consistency_clt_con} require  $r_n\gg h_n\log (1/h_n)$ in general, under the idealized assumption of constant volatility, our proofs demonstrate that  $r_n$ can be chosen of the form $r_{n}={2(1+\eta)h_n\log (1/h_n)}$, for any $\eta>0$ (cf. \eqref{e:estar}). In general, the common choice $\sqrt{r_n}= 4 h_n^{0.49}$ seems to also work reasonably well also for the iterative estimator $\widehat C_n$; for further discussion of selection of the threshold rate $r_n$ and initializations in practice, see Section \ref{s:montecarlo}.
\end{remark}

\begin{remark} When $\alpha\in(1,2)$, a slowly-decaying bias term  renders convergence rates of order $n^{-1/2}$ impossible for $\trv$ itself under deterministic thresholding (c.f. \cite{mancini:2011}), which is reflected for $\widehat C_n$ in case (iii) of Theorem \ref{thm:consistency_clt_con}.  For these values of $\alpha$, under deterministic thresholds for $\trv$, specialized debiasing techniques are required to achieve the optimal $n^{-1/2}$ rate (c.f. \cite{boniece:figueroa-lopez:han:2024}). %
See also \cite{jacod:todorov:2014,jacod:todorov:2018} for other efficient methods based on empirical characteristic functions (and again deterministic tuning parameters). %
We leave the extension of these debiasing techniques to data-driven parameter tuning as a topic for future research.
\end{remark}

\
\subsection{Time-varying thresholds}\label{s:locthresh}
For the best possible finite-sample performance, heuristically one should set the threshold $\varepsilon$ \textit{as small as possible} -- to remove as many jumps as possible -- but allow it to remain large enough so that a sufficient number increments remain to ultimately yield efficient estimates of $C_T$.  
From this perspective, the asymptotic lower bound on the rate $r_n$ given in the hypotheses of Theorem \ref{thm:consistency_clt_con} may appear unsatisfactory, as it precludes rates as fast as the optimal threshold $\varepsilon^\star_n\sim \sqrt{2\sigma^2 h_n\log(1/h_n)}$ in the constant volatility case. It is natural to suspect that faster rates may be possible for potential  improvement in $\widehat C_n$. However, the next result shows this is not true, in general.

\begin{proposition}\label{p:divergence_C_hat} For a given semimartingale $Y$, let $\trv_n(\varepsilon;Y)=\sum_{i=1}^{n}\big(\Delta_{i}^{n}Y\big)^{2}{\bf 1}_{\{|\Delta_{i}^{n}Y|\leq\varepsilon\}}$, and let $C_T(Y)$ denote its  predictable quadratic variation. For a given $c_0>0$, define the threshold
$$
\vartheta_n(Y):= \sqrt{  c_0C_T(Y) h_n \log(1/h_n)}.
$$
 Then, there exists a semimartingale $Y'$ such that $\textnormal{TRV}_n(\vartheta_n(Y');Y') \toprob C_T(Y')$,  but
$$
\sqrt n \big(\textnormal{TRV}_n(\vartheta_n(Y');Y') - C_T(Y')\big)\toprob -\infty.
$$

\end{proposition}
Note  that 
the estimator $\widehat C_n$ defined in (\ref{DfnME0}) is a TRV with threshold $\varepsilon_n= \big(c_0\widehat C_n h_n \log(1/h_n)\big)^{1/2}$, which is approximately equal to $\vartheta_n$ if $\widehat C_n$ remains a consistent estimator under this threshold choice.  In that case, the above result suggests that $\widehat C_n$  will not be rate-efficient in general with the threshold rate $r_n=c_0 h_n\log(1/h_n)$ and may remove too many increments even if jumps are completely absent from the process $X$. %
In particular, the proof of Proposition \ref{p:divergence_C_hat} illustrates that efficiency losses can result from volatility paths that exhibit significant jumps.  A natural way to remedy this is  to consider localized thresholds that adapt to the volatility level. In this way, thresholds corresponding to periods of high volatility are increased, and conversely, thresholds for periods of low volatility are decreased, so as to prevent efficiency losses that might otherwise occur with uniform thresholding. This is the central motivation behind our second class of estimators,  which utilize spot volatility estimates to locally tune the threshold.

To this end, for a given even integer $k_n\leq n$  and $B>0$, we define
\begin{align*}%
\widehat\sigma^2_n(i;B) &:=  \frac{ 1}{h_nk_n} \sum_{\ell=i-k_n/2+1}^{i +k_n/2} \big(\Delta_\ell^n X\big)^2{\mathbf 1}_{\{|\Delta_\ell^n X|\leq B\}},\quad \ell=1,\ldots, n,
\end{align*}
where we set $\Delta_i^n X =0 $ if $i\leq 0$ or $i >n$.  The above estimator is a  type of kernel-based estimator of the spot volatility $\sigma_{t_i}^2$,  as defined in \cite{fan:wang:2008,kristensen:2010}, with kernel function $K(x)=\frac{1}{2}{\bf 1}_{[-1,1]}$ and bandwidth $b_n=k_n h_n$ (see \cite{jacod:protter:2011} for the asymptotic theory of the estimator in the case of one-sided uniform kernels $K(x)={\bf 1}_{[0,1]}$ and  \cite{figueroa-lopez:li:2020,figueroa-lopez:wu:2022} for general kernels).
Our second thresholding scheme for the localized thresholding estimator $\widehat C_n^*$ then proceeds as follows:
\begin{enumerate}
	\item  First, we begin with  some initial local volatility estimates $\widehat c_{n,0}(i)$, $i=1,\ldots,n$, ideally free of tuning parameters. %
Natural choices include
$\widehat c_{n,0}(i) = \frac{1}{h_nk_n} \sum_{\ell=i-k_n/2+1}^{i +k_n/2}  \big(\Delta_\ell^n X\big)^2$ or the analogous local BPV estimator $\widehat c_{n,0}(i) = \frac{1}{h_nk_n} \frac{\pi}{2}\sum_{\ell=i-k_n/2+1}^{i +k_n/2}  |\Delta_\ell^n X||\Delta_{\ell+1}^n X|$ (see Theorem \ref{thm:consistency_clt_2} below for precise conditions on $\widehat c_{n,0}(i)$). 

\item 
Next, for each  $i=1,\ldots, n$ and a given deterministic rate sequence $r_n^*\to 0$, we define a local threshold for the $i$--th increment of $X$. With regards to \eqref{e:fixedpoint}, each of these local thresholds can be viewed as a fixed-point iteration for one of the $n$ maps $\xi \mapsto \widehat\sigma^2_n\big(i; \sqrt{r_n^*\xi}\big)$, $i=1,\ldots, n$.  More specifically, for $j\geq{}1$, we iteratively define:
\begin{align}\label{itrtiveOptLocEstb}
\begin{split}
B^{*}_{n,j-1}(i) &:=  \sqrt{r^*_n \widehat c_{n,j-1}(i)}, \\
\widehat c_{n,j}(i)&:=\hat \sigma^2_{n}(i,B^{*}_{n,j-1}(i)).
\end{split}
\end{align}
We then set\vspace{-1ex}
\begin{align}\label{e:stablizing_jnstar}
B^*_n(i)&:= B^{*}_{n,j^{*}_n}(i),\quad i=1,\ldots,n,
\end{align}
where $j_n^{*} :=  \min\big\{j\geq{}0: \widehat c_{n,j}(i)=\widehat c_{n,j+\ell}(i)\text{~ for all ~} 1\leq i \leq n, ~ \ell\geq 0\big\}$.  
\item Finally, we define the localized threshold estimator
\begin{equation}\label{DfnOptCLoc}
\widehat C^{*}_{n} := \sum_{i=1}^n (\Delta_i^nX)^2 \,{\bf 1}_{\{|\Delta_i^n X|\leq  B^*_n(i)\}}.
\end{equation}
\end{enumerate}

Let us now introduce the class of initial estimates $\mathcal C^\text{spot}$ for time-varying thresholds, which is essentially a localized analog of the class $\mathcal C$ of initial estimates defined in Section \ref{s:uniformthresh}. To this end, for a generic process $Y$, define
\begin{equation}\label{e:def_little_C0}
\widehat c_{n,0}(i;Y)= \frac{ 1}{h_nk_n}\sum_{\ell=i-k_n/2+1}^{i +k_n/2} F(\Delta_i^n Y,\ldots, \Delta_{i+d-1}^nY),\quad i=1,\ldots,n,
 \end{equation}
 where $F:\bR^d\to [0,\infty)$, and for convenience we set $F( \Delta_i^n Y,\ldots,\Delta_{i+d-1}^n Y)=0$ if $i\leq 0$ or $i+d-1>n$. We say the initializing threshold constants $\widehat c_{n,0}(i)$ belong to the class $\mathcal C^\text{spot}$  if $\widehat c_{n,0}(i)=\widehat c_{n,0}(i;X)$,  $i=1,\ldots n$, where $\widehat c_{n,0}(i;\,\cdot\,\,)$ are of the form \eqref{e:def_little_C0} and satisfy%
 \begin{align}\label{e:little_c0_consistent_contcase}
 \max_{1\leq i \leq n}\Bigg|  \widehat c_{n,0}(i;\sigma\!\cdot\!W) - \frac{ 1}{h_nk_n} \int_{h_n(i-k_n/2+1)}^{h_n(i+k_n/2)}\sigma_t^2dt%
 \Bigg|\toprob 0.
\end{align}
We are now in a position to state our second main result.

\begin{theorem}\label{thm:consistency_clt_2} Let  $r^*_n$ satisfy 
\begin{align}\label{e:rstar_rate}
\liminf_{n\to\infty} \frac{r_n^*}{2h_n \log (1/h_n)} > 1,%
\end{align}
 and suppose that   $n^a \ll k_n\ll n$ for some $0<a<1$.\
Asssume further that the initializing threshold constants $\widehat c_{n,0}(i)$, $i=1\ldots,n$, belong to the class $\mathcal C^\text{spot}$. %
Then, the following assertions hold:
 \begin{enumerate}
    \item[(i)] If, further, 
    $r^*_n\ll \left(h\log (1/h)\right)^{\frac\alpha2}$, 
    then, as $n\to\infty$,
    $$ \widehat C^{*}_{n}\toprob C_T.$$%
    \item[(ii)]  If $\alpha \in (0,1)$ and further
$ r_n^* \ll h_n^{\frac{\alpha+1}{2}} \left(\log (1/h)\right)^{\frac\alpha2},$ %
    then,  as $n\to\infty$,
 $$
\frac{1}{ \sqrt{h_n}}\left(\widehat C^{*}_n - C_T\right)\toDistSt \mathcal{N}\left(0,2\int_0^T\sigma^4_sds\right).
 $$
     \item[(iii)] Suppose $\alpha \in (1,2)$ and $\gamma_t\equiv 1$.  Then,  for any choice of $r_n^*$ satisfying \eqref{e:rstar_rate},  as $n\to\infty$,
 $$
 \frac{1}{\sqrt{h_n}}\big(\widehat C_n^* - C_T\big)\toprob \infty.%
 $$

     \end{enumerate}
\end{theorem}

\begin{remark}\label{CmntInitCa2}
 Similarly to the case for uniform thresholding, the assumptions on the admissible initial estimates \eqref{e:def_little_C0} are relatively mild and require consistency in a uniform sense only in the continuous case. In particular, a localized version of realized variance, namely $\widehat c_{n,0}(i;X)=\frac{1}{h_nk_n} \sum_{\ell=i-k_n/2+1}^{i +k_n/2}  \big(\Delta_\ell^n X\big)^2$, or a localized bipower variation, namely $\widehat c_{n,0}(i;X)=\frac{ 1}{h_nk_n}\frac{\pi}{2}\sum_{\ell=i-k_n/2+1}^{i +k_n/2}  |\Delta_\ell^n X||\Delta_{\ell+1}^n X|$ both satisfy these assumptions. Indeed, the validity of \eqref{e:little_c0_consistent_contcase} for the localized realized variance is shown in the proof of Lemma \ref{l:untruncated_strong_consistency} (see \eqref{e:needonlytoshow}), while for the localized bipower variation it follows along the same arguments as the proof of Proposition 3.3 in \cite{palmes:woerner:2016}.%
\end{remark}

Statistical errors for spot volatility estimation are known to be substantially larger by comparison to the $O_P(n^{-1/2})$--sized errors that occur in estimation of integrated volatility $C_T$ (for instance, optimal choices of $k_n$ in spot volatility estimation lead to errors of order $n^{-1/4}$; see, e.g., \cite{figueroa-lopez:wu:2022,jacod:protter:2011}). Interestingly enough, the estimator $\widehat C^{*}_n$ utilizes the comparatively noisier estimates of spot volatility in an auxiliary manner to lead to potentially improved estimates of $C_T$.%

\begin{remark} Recall that, to a first-order approximation, the threshold \eqref{e:estar} is MSE-optimal under the assumption of constant volatility and finite jump activity.  Though this is an idealized assumption and not expected to hold in many practical settings, it can serve as a reasonable local approximation at sufficiently high sampling frequencies. This intuition provides further support for our estimator $\widehat C_n^{*}$, which roughly incorporates the optimal-type threshold \eqref{e:estar} in a local manner.  Indeed, our simulation study in Section \ref{s:montecarlo} reflects this, showing that the resulting localized estimators exhibit superior finite-sample performance. Furthermore, we conjecture that for TRV based on time-varying thresholds $\varepsilon(i)$, $i=1\ldots,n$,  the threshold choice
 $$
 \varepsilon_n(i)=\sqrt{2\overline \sigma^2_{n,i}h_n \log(1/h_n)},
 $$
 with $\overline \sigma^2_{n,i} =\sup_{t\in[\frac{i-1}{n},\frac{i}{n})}\sigma^2_t$, 
is MSE-optimal up to first order approximation under finite jump activity.
\end{remark}
\begin{remark} Feasible CLTs (for construction of confidence intervals) are possible with either $\widehat C_n$ or $\widehat C_n^*$.   Indeed, minor extensions to our arguments show that under the rates $r_n,r_n^*$ in part (ii) of Theorems \ref{thm:consistency_clt_con} and \ref{thm:consistency_clt_2}, one has
$$
\frac{\sqrt{n}\left(\widehat C_n - C_T\right)}{\sqrt{2 \sum_{i=1}^n (\Delta_i^n X)^4 {\bf 1}_{\{|\Delta_i^n X| \leq B_n\}}}}\stackrel{d} \to \mathcal N(0,1),~~\text{and}~~
\frac{\sqrt{n}\left(\widehat C^*_n - C_T\right)}{\sqrt{2 \sum_{i=1}^n (\Delta_i^n X)^4 {\bf 1}_{\{|\Delta_i^n X| \leq B^*_n(i)\}}}}\stackrel{d} \to \mathcal N(0,1),
$$
where the thresholds $B_n$ and $B^*_n(i)$ are the same as for $\widehat C_n$ and for $\widehat C_n^*$ as described in Theorems \ref{thm:consistency_clt_con} and \ref{thm:consistency_clt_2}, respectively.
\end{remark}
\begin{remark} Strictly speaking, the estimators $\widehat C_n$ and $\widehat C_n^*$ are not ``tuning-free'' in the same way as, e.g., bipower variation is: $\widehat C_n$ depends on $r_n$ and $\widehat C_n^*$ depends on $r_n^*$ and $k_n$, which must still be chosen by the practitioner.  However, the auxiliary sequences $r_n,r_n^*$ and $k_n$ are relatively data-insensitive and can be chosen based on sampling frequency alone.  This stands in contrast with selecting the (full) parameter $\varepsilon$ itself, which is highly data-sensitive, as demonstrated in Figure \ref{fig1}.  This is further illustrated in the next section.
\end{remark}
\section{Monte Carlo study}\label{s:montecarlo}

In this section, we compare the finite-sample performance of $\widehat C_n$ and $\widehat C_n^*$ against standard tuning approaches for TRV in the literature based on simulated data from the following stochastic volatility model:
\begin{equation*}
\begin{aligned}
X_{t}&=1+\int_{0}^{t}\sigma_s\,dW_{s}+L_{t} +J_t, \\
\sigma^2_{t}&=\theta+\int_{0}^{t}\kappa\big(\theta-\sigma^2_s\big)\,ds+\xi\int_{0}^{t}\sigma_s\,dB_{s}.%
\end{aligned}%
\end{equation*}
Above,  $W$ and $B$ are two correlated standard Brownian motions with covariation $d\langle W, B\rangle_t=\rho dt$, $L$ is a CGMY L\'{e}vy process independent of $W$ and $B$, and $J$ is an inhomogeneous compound Poisson process independent of all other processes with intensity $\{\lambda(t)\}_{t\geq{}0}$ and jump distribution $\varrho(dx)$.

Based on a 6.5 hour trading day and 252 trading days per year, we consider time horizons of $T\in\{\frac{1}{252}, \frac{5}{252},\frac{1}{12}\}$, corresponding to 1 day, 1 week, and 1 month, respectively, at the 5-minute ($h_n=(\frac{1}{252})(\frac{1}{6.5})(\frac{5}{60})$)  sampling frequency.  For illustration, we examine five separate scenarios we now describe.
  Unless otherwise stated, for ease of comparison the parameters for  $\{\sigma_t^2\}_{t\geq{}0}$  are set as:
\begin{align*}
\kappa=5,\quad\xi=0.3,\quad\theta=(0.2)^2, \quad {\rho = -0.5}.
\end{align*}
With these parameter choices, the annualized expected integrated variance is $(1/T)\bE C_T = (0.2)^2$, and, in all settings, parameters are chosen so that the expected annualized realized volatility is approximately $\sqrt{(1/T)\bE (\text{RV}_n)}\approx 0.275$ (Models 1,2,4,5, below) or  $0.3$ (Model 3), which are realistic for financial data.
\begin{itemize}
\item Model 1 (homogeneous jumps):
for the infinite activity component $\{L_t\}_{t\geq{}0}$, we choose
\begin{align*}%
C_-=0.148,\quad C_+ = 0.033, \quad G=3.295,\quad M=4.685, \quad Y=0.917.
\end{align*}
The parameters $C_-,C_+,G,M$ are taken from estimates in \cite{kawai:2010} for a 1-year interval based on calibration from index options; here $Y=0.917$ corresponds to the average of the reported estimates of $Y^+$ and $Y^-$ in their model, namely  $Y=\frac{Y^++Y^-}{2}$. For  $\{J_t\}_{t\geq{}0}$, we choose%
$$\lambda(t)\equiv   252~\text{(1 jump per day)},\quad  \varrho(dx)\sim \mathcal N(-0.005, 0.01^2).$$

\item Model 2 (switching jump intensity): 
all settings are the same as in Model 1, except  $\{J_t\}_{t\geq{}0}$ has time-varying intensity: 
$$\lambda(t)=  (252)\vartheta(t),\quad \vartheta(t)=\begin{cases}
2,&\sigma^2_t\geq\theta,\\
0,& \sigma^2_t<\theta.
\end{cases}$$
\item Model 3 (higher jump intensity): all settings are the same as Model 1, except $\{J_t\}_{t\geq{}0}$ has a higher jump intensity:
$$\lambda(t)\equiv   (1.5)(252)~\text{(1.5 jumps per day)}.$$
\item Model 4 (finite jump activity): all settings are the same as Model 1, except we take  $L_t\equiv 0$, and adjust $\lambda(t)\equiv   (1.15)252$ to match the expected realized variance of Model 1.\\
\item Model 5 (no jumps): All settings are the same as Model 1, except we take $J_t\equiv 0$, $L_t\equiv 0$, and adjust $\theta=(0.275)^2$ to match the expected realized variance of Model 1.
\end{itemize}
We compare 6 types of estimators based on TRV: two instances of standard approaches, and two instances each of the iterative estimator $\widehat C_n$ and of the localized iterative estimator $\widehat C^*_n$ with different types of initializations.  Specifically, for $\widehat C_n$ we use the following estimators as initializations: %
$$
\text{RV}_n = \sum_{i=1}^n(\Delta_i^nX)^2,\qquad \text{BV}_n= \frac{\pi}{2}\sum_{i=2}^n |\Delta_{i-1}^n X| |\Delta_{i}^n X|.
$$
For $\widehat C_n^*$, we use their localized counterparts, denoted by
$$
\hat \sigma_n^2(\ell) = \frac{1}{h_nk_n}\sum_{i=\ell-k_n/2+1}^{\ell+k_n/2}(\Delta_i^n X)^2,\quad \text{BV}^{\text{spot}}_n(\ell)= \frac{1}{h_nk_n}\frac{\pi}{2}\sum_{i=\ell-k_n/2+1}^{\ell+k_n/2} |\Delta_{i-1}^n X| |\Delta_{i}^n X|.
$$
We consider the following estimation procedures:
\begin{enumerate}%
\item[(1)] $\text{TRV}_n(\varepsilon_{0,n})$, where $\varepsilon_{0,n}= h^{0.49}_n$;
\item[(2)] $\text{TRV}_n(\varepsilon_{1,n})$, where $\varepsilon_{1,n}= \sqrt{\text{BV}_n r_n}$, with $\sqrt{r_n}=4 h^{0.49}_n$ as used in \cite{jacod:todorov:2014,li:todorov:tauchen:2017,ding:li:liu:zheng:2023}.
\item[(3)] $\widehat C_n$, with initialization   $\widehat{C}_{n,0}=\text{RV}_n$ and $r_n$ as in (2);
\item[(4)]  $\widehat C_n$, with initialization $\widehat{C}_{n,0}=\text{BV}_n$ and $r_n$ as in item (2);
\item[(5a,b)] $\widehat C_n^{*}$, with  $r_n^*=2h_n(\log (1/h_n) -\log\log (1/h_n))$, $k_n=h^{-0.5}_n$ (5a) or $k_n=h^{-0.6}_n$ (5b), and initialization $\widehat c_{n,0}(i)=\hat \sigma^2_n(i)$;
\item[(6a,b)] $\widehat C_n^{*}$, with $r_n^*=2h_n(\log (1/h_n) -\log\log (1/h_n))$, $k_n=h^{-0.5}_n$ (6a) or $k_n=h^{-0.6}_n$ (6b),  and initialization $\widehat c_{n,0}(i)=\text{BV}^{\text{spot}}_n(i)$.
\end{enumerate}

We  simulate  $m=5000$ paths for each Model 1-5.  Denoting by $\widehat{\mathcal C}$ one of the estimators in (1)-(6), on the $j$--th realization we compute the estimator value, $\widehat{\mathcal C}_j$, the corresponding true integrated volatility, $C_{T,j}$, and report  
\begin{itemize}
\item The mean relative error  (in \%): $100(\frac{1}{m}\sum_{j=1}^m e_j$), where $e_j=\frac{\widehat{\mathcal C}_j-C_{T,j}}{C_{T,j}}$;%
\item The standard deviation of the relative error (in \%): $100\sqrt{\frac{1}{m}\sum_{j=1}^m (e_j-\overline e)^2}$;
\item $\sqrt{\text{MSE}} = \sqrt{\frac{1}{m}\sum_{j=1}^m (\widehat{\mathcal C}_j - C_{T,j})^2}$.
\end{itemize}
The results are displayed in Tables \ref{t:mod123}-\ref{t:mod4}; the smallest bias and MSE for each time horizon are shown in bold.
\begin{table}[h!]
\centering
\scalebox{0.89}{
\begin{footnotesize}
\begin{tabular}{c||ccc|ccc|ccc}
\midrule
\multicolumn{10}{c}{Model 1 (homogeneous jumps)}\\\toprule
  & \multicolumn{3}{c|}{$T=1/252$ (1 day)}   & \multicolumn{3}{c|}{$T=5/252$ (1 wk.)} & \multicolumn{3}{c}{$T=1/12$ (1 mo.)} \\\cline{2-10}
&rel.\,err (\%)&sd(rel.\,err)&\scalebox{0.8}{$\sqrt{\text{MSE}}$}\scalebox{0.7}{$\times 10^4$}&rel.\,err (\%)&sd(rel.\,err)&\scalebox{0.8}{$\sqrt{\text{MSE}}$}\scalebox{0.7}{$\times 10^4$}&rel.\,err (\%)&sd(rel.\,err)&\scalebox{0.8}{$\sqrt{\text{MSE}}$}\scalebox{0.7}{$\times 10^4$}\\\midrule
(1) TRV&90.4783&149.1879&2.7355&90.2462&67.2049&8.7914&92.9712&39.0118&31.4622\\
(2) TRV&5.3009&20.8770&0.3421&4.8734&8.7488&0.8040&4.7463&4.4099&2.2017\\
(3) \scalebox{0.8}{$\widehat C_n$}&5.8617&22.3127&0.3665&3.9965&8.4482&0.7505&3.6931&4.2304&1.9256\\
(4) \scalebox{0.8}{$\widehat C_n$}&4.7353&20.1997&0.3301&3.9688&8.4253&0.7478&3.6931&4.2304&1.9256\\
(5a) \scalebox{0.8}{$\widehat C^*_n$}&3.3682&19.0994&0.3087&2.9690&8.1442&0.6956&2.8222&4.1195&1.7050\\
(5b) \scalebox{0.8}{$\widehat C_n^*$}&3.3297&19.0585&0.3079&2.5464&8.0069&0.6740&2.3890&4.0001&1.5881\\
(6a) \scalebox{0.8}{$\widehat C^*_n$}&2.8296&18.5914&0.2994&2.7570&8.0723&0.6843&2.6449&4.0711&1.6546\\
(6b) \scalebox{0.8}{$\widehat C_n^*$}&\textbf{2.8294}&18.5915&\textbf{0.2989}&\textbf{2.4985}&8.1560&\textbf{0.6714}&\textbf{2.3621}&4.2415&\textbf{1.5815}\\\toprule
\end{tabular}
\end{footnotesize}
}\\
\scalebox{0.89}{
\begin{footnotesize}
\begin{tabular}{c||ccc|ccc|ccc}
\midrule
\multicolumn{10}{c}{Model 2 (switching jump intensity)}\\\toprule
  & \multicolumn{3}{c|}{$T=1/252$ (1 day)}   & \multicolumn{3}{c|}{$T=5/252$ (1 wk.)} & \multicolumn{3}{c}{$T=1/12$ (1 mo.)} \\\cline{2-10}
&rel.\,err (\%)&sd(rel.\,err)&\scalebox{0.8}{$\sqrt{\text{MSE}}$}\scalebox{0.7}{$\times 10^4$}&rel.\,err (\%)&sd(rel.\,err)&\scalebox{0.8}{$\sqrt{\text{MSE}}$}\scalebox{0.7}{$\times 10^4$}&rel.\,err (\%)&sd(rel.\,err)&\scalebox{0.8}{$\sqrt{\text{MSE}}$}\scalebox{0.7}{$\times 10^4$}\\\midrule
(1) TRV&98.4132&190.3505&3.5307&89.9458&95.9290&11.3156&83.7461&62.7170&40.2647\\
(2) TRV&6.4172&23.0589&0.3911&5.1908&10.1623&0.9698&4.8893&6.0500&3.0746\\
(3) \scalebox{0.8}{$\widehat C_n$}&6.6465&29.6733&0.5108&3.7809&9.1587&0.8318&3.4781&5.0420&2.3673\\
(4) \scalebox{0.8}{$\widehat C_n$}&5.1206&20.8151&0.3475&3.7568&9.1281&0.8278&3.4772&5.0429&2.3673\\
(5a) \scalebox{0.8}{$\widehat C^*_n$}&3.6255&19.4813&0.3193&2.8709&8.6546&0.7561&3.0375&4.7020&2.1106\\
(5b) \scalebox{0.8}{$\widehat C_n^*$}&3.6256&19.5659&0.3207&2.2625&8.3286&0.7137&2.4621&4.4255&1.8864\\
(6a) \scalebox{0.8}{$\widehat C^*_n$}&3.1309&18.7966&\textbf{0.3062}&2.6498&8.5198&0.7385&2.8564&4.6173&2.0377\\
(6b) \scalebox{0.8}{$\widehat C_n^*$}&\textbf{3.1190}&18.7985&0.3063&\textbf{2.2198}&8.6419&\textbf{0.7114}&\textbf{2.4332}&4.8536&\textbf{1.8801}\\\toprule
\end{tabular}
\end{footnotesize}
}\\
\scalebox{0.89}{
\begin{footnotesize}
\begin{tabular}{l||ccc|ccc|ccc}
\midrule
\multicolumn{10}{c}{Model 3 (higher jump intensity)}\\\toprule
  & \multicolumn{3}{c|}{$T=1/252$ (1 day)}   & \multicolumn{3}{c|}{$T=5/252$ (1 wk.)} & \multicolumn{3}{c}{$T=1/12$ (1 mo.)} \\\cline{2-10}
&rel.\,err (\%)&sd(rel.\,err)&\scalebox{0.8}{$\sqrt{\text{MSE}}$}\scalebox{0.7}{$\times 10^4$}&rel.\,err (\%)&sd(rel.\,err)&\scalebox{0.8}{$\sqrt{\text{MSE}}$}\scalebox{0.7}{$\times 10^4$}&rel.\,err (\%)&sd(rel.\,err)&\scalebox{0.8}{$\sqrt{\text{MSE}}$}\scalebox{0.7}{$\times 10^4$}\\\midrule
(1) TRV&121.7856&167.8171&3.2809&133.4801&81.8533&12.1529&134.3111&49.2704&44.6016\\
(2) TRV&7.8602&24.2170&0.4053&7.3329&10.1078&0.9921&7.1037&4.8398&2.9336\\
(3) \scalebox{0.8}{$\widehat C_n$}&9.1251&36.1741&0.5982&5.2370&9.2686&0.8526&4.8782&4.5093&2.3053\\
(4) \scalebox{0.8}{$\widehat C_n$}&6.3337&22.0729&0.3656&5.2050&9.2380&0.8490&4.8782&4.5093&2.3053\\
(5a) \scalebox{0.8}{$\widehat C^*_n$}&4.0298&20.3463&0.3307&3.7660&8.7422&0.7627&3.6815&4.2953&1.9582\\
(5b) \scalebox{0.8}{$\widehat C_n^*$}&4.0138&20.2726&0.3295&3.0898&8.4432&0.7199&3.0808&4.1801&1.7936\\
(6a) \scalebox{0.8}{$\widehat C^*_n$}&\textbf{3.3382}&19.4716&0.3145&3.5166&8.6317&0.7458&3.4983&4.2390&1.8997\\
(6b) \scalebox{0.8}{$\widehat C_n^*$}&3.5848&19.5251&\textbf{0.3140}&\textbf{3.0404}&8.6777&\textbf{0.7159}&\textbf{3.0138}&4.5086&\textbf{1.7595}\\\toprule
\end{tabular}
\end{footnotesize}}%
\caption{ \label{t:mod123}\footnotesize Estimation performance of $\widehat C_n$, $\widehat C_n^*$, and standard tuning approaches for $\text{TRV}$ in Models 1-3; reported values are based on $m=5000$ realizations in each model at the 5-minute sampling frequency.}
\end{table}

In general, we see that when jumps are present (Models 1-4), both the iterative estimator $\widehat C_n$ and localized iterative estimator $\widehat C_n^*$ can outperform the standard tuning choice (2) for $\text{TRV}$ both in terms of relative error and MSE by significant margins, with reductions in bias often by 50\% or more by comparison to (2)  and reductions in $\sqrt{\text{MSE}}$ as high as 40\%. As anticipated, deterministic tuning (1) performs rather poorly by comparison to approaches (2)-(6) on all time horizons, and although the (non-iterative) bipower-tuned TRV in (2) leads to a substantial improvement over (1), it is uniformly outperformed by (4)-(6) on all time horizons considered and also outperformed by (3) except on daily time horizons.

 In general, the localized estimators (6a,6b) have the largest relative performance gains compared to standard procedures (1)-(2) over longer time horizons, which is somewhat expected, ranging from 13\%-23\% reduction in $\sqrt{\text{MSE}}$ at daily horizons to 28\%-40\% reduction in $\sqrt{\text{MSE}}$ at monthly horizons compared to (2).  Also, iterative approaches with jump-robust initializations (4,6a,6b) generally have improved performance compared to those  without jump-robust initializations (3,5a,5b).  Furthermore, for the localized estimators, the choice $k_n=h_n^{-0.6}$ (5b,6b) tends to lead to improvement relative to the choice $k_n=h_n^{-0.5}$ (5a,6a)  over longer time horizons. The best performance in terms of both relative error and MSE is typically achieved by (6b).

 \begin{table}[h!]
\centering
\scalebox{0.89}{
\begin{footnotesize}
\begin{tabular}{l||ccc|ccc|ccc}
\midrule
\multicolumn{10}{c}{Model 4 (finite jump activity)}\\\toprule
  & \multicolumn{3}{c|}{$T=1/252$ (1 day)}   & \multicolumn{3}{c|}{$T=5/252$ (1 wk.)} & \multicolumn{3}{c}{$T=1/12$ (1 mo.)} \\\cline{2-10}
&rel.\,err (\%)&sd(rel.\,err)&\scalebox{0.8}{$\sqrt{\text{MSE}}$}\scalebox{0.7}{$\times 10^4$}&rel.\,err (\%)&sd(rel.\,err)&\scalebox{0.8}{$\sqrt{\text{MSE}}$}\scalebox{0.7}{$\times 10^4$}&rel.\,err (\%)&sd(rel.\,err)&\scalebox{0.8}{$\sqrt{\text{MSE}}$}\scalebox{0.7}{$\times 10^4$}\\\midrule
(1) TRV&95.5145&159.3735&2.9314&89.8228&66.1216&8.6889&93.9742&40.3833&31.7663\\
(2) TRV&4.3921&21.5406&0.3515&3.5605&8.7916&0.7641&3.5693&4.2456&1.9309\\
(3) \scalebox{0.8}{$\widehat C_n$}&4.9183&23.2150&0.3783&2.6648&8.4874&0.7164&2.5184&4.1242&1.6911\\
(4) \scalebox{0.8}{$\widehat C_n$}&3.5558&20.1041&0.3255&2.6388&8.4716&0.7145&2.5179&4.1244&1.6911\\
(5a) \scalebox{0.8}{$\widehat C^*_n$}&1.9505&18.3226&0.2933&1.5714&8.0859&0.6617&1.7152&3.9780&1.5053\\
(5b) \scalebox{0.8}{$\widehat C_n^*$}&1.9454&18.2932&0.2928&1.1940&7.9727&0.6475&1.2999&3.8990&1.4207\\
(6a) \scalebox{0.8}{$\widehat C^*_n$}&1.5691&18.0673&0.2886&1.3933&8.0169&0.6527&1.5442&3.9375&1.4680\\
(6b) \scalebox{0.8}{$\widehat C_n^*$}&\textbf{1.5623}&18.0679&\textbf{0.2880}&\textbf{1.1613}&8.0538&\textbf{0.6452}&\textbf{1.2761}&4.0324&\textbf{1.4166}\\\toprule
\end{tabular}
\end{footnotesize}
}
\scalebox{0.89}{
\begin{footnotesize}
\begin{tabular}{l||ccc|ccc|ccc}
\midrule
\multicolumn{10}{c}{Model 5 (no jumps)}\\\toprule
  & \multicolumn{3}{c|}{$T=1/252$ (1 day)}   & \multicolumn{3}{c|}{$T=5/252$ (1 wk.)} & \multicolumn{3}{c}{$T=1/12$ (1 mo.)} \\\cline{2-10}
&rel.\,err (\%)&sd(rel.\,err)&\scalebox{0.8}{$\sqrt{\text{MSE}}$}\scalebox{0.7}{$\times 10^4$}&rel.\,err (\%)&sd(rel.\,err)&\scalebox{0.8}{$\sqrt{\text{MSE}}$}\scalebox{0.7}{$\times 10^4$}&rel.\,err (\%)&sd(rel.\,err)&\scalebox{0.8}{$\sqrt{\text{MSE}}$}\scalebox{0.7}{$\times 10^4$}\\\midrule
(1) TRV&0.0917&15.8450&0.5671&\textbf{-0.1081}&7.0844&\textbf{1.2695}&\textbf{-0.0536}&3.4524&\textbf{2.6256}\\
(2) TRV&0.0504&15.8488&0.5671&-0.1308&7.0946&1.2717&-0.0936&3.4619&2.6346\\
(3) \scalebox{0.8}{$\widehat C_n$}&0.0857&15.8440&\textbf{0.5670}&-0.1274&7.0909&1.2708&-0.0930&3.4617&2.6343\\
(4) \scalebox{0.8}{$\widehat C_n$}&0.0632&15.8549&0.5673&-0.1311&7.0941&1.2715&-0.0940&3.4616&2.6344\\
(5a) \scalebox{0.8}{$\widehat C^*_n$}&\textbf{0.0306}&15.8651&0.5677&-0.1920&7.1099&1.2744&-0.1748&3.4660&2.6394\\
(5b) \scalebox{0.8}{$\widehat C_n^*$}&0.0344&15.8649&0.5677&-0.2343&7.1112&1.2750&-0.2063&3.4639&2.6388\\
(6a) \scalebox{0.8}{$\widehat C^*_n$}&-0.0636&15.8999&0.5689&-0.2669&7.1294&1.2782&-0.2492&3.4651&2.6424\\
(6b) \scalebox{0.8}{$\widehat C_n^*$}&-0.0668&15.8999&0.5691&-0.2696&7.1293&1.2766&-0.2432&3.4656&2.6425\\\toprule
\end{tabular}
\end{footnotesize}
}
\caption{\label{t:mod4} \footnotesize Estimation performance of $\widehat C_n$, $\widehat C_n^*$, and standard tuning approaches for $\text{TRV}$ in the finite jump activity setting of Model 4 and the jump-free setting of Model 5; reported values are based on $m=5000$ realizations at the 5-minute sampling frequency.}
\end{table}

 Comparing performance across Models 1-4, we see that all  iterative approaches (3)-(6) are generally more robust against increased levels of jump activity as well as time-varying jump behavior compared to (1)-(2).  In both Models 2 and 3,  on longer time horizons, the performance advantage of the localized estimators over uniform approaches is typically larger by comparison to the performance advantage they have over uniform approaches in Model 1. Though all estimators (1)-(6) have better overall performance under finite jump activity (Model 4) relative to settings with infinite activity (Models 1-3), the iterative approaches still retain performance advantages over standard choices (1)-(2) even without an infinite activity component in the model.

Turning to the jump-free case (Model 5), we note that all estimators perform very similarly in terms of both bias and MSE and are typically slightly negatively biased.  Over weekly and monthly time horizons, the localized estimators (5-6) incur a very slight increase in bias (appx. 0.15\%)  compared to uniform thresholding (2), and the deterministic TRV has marginally smaller $\sqrt{\text{MSE}}$ compared to (2)-(6).

We note that although a slight increase in bias occurs in the localized estimators (5)-(6) in the absence of jumps, it is relatively small relative to the potential performance gain one may attain if jumps are present.  Since jumps are generically expected in many types of data, for use in practice we recommend the localized estimator with jump-robust initialization and settings of (6b).  However, if a simpler implementation is desired, or one wants to avoid the potential marginal additional bias when jumps are absent,   method (4) is a reasonable alternative.  We remark that in any case, these choices (4,6b) in the presence of jumps can significantly outperform the common choice in the literature (2).   %

\begin{table}[h!]
\begin{footnotesize}
\begin{tabular}{l | c c c c c c c  }
\midrule
\multicolumn{8}{c}{Number of iterations until stabilization $(T=1/12)$ }\\\toprule
& 1 & 2 & 3 & 4 &5 & 6 &  $\geq 7$\\
\midrule
 $j_n$ (\scalebox{0.9}{$\widehat C_n$}) &1.13\% & 27.21\% &   54.13\% &  15.69\% & 1.65\% &  0.17\% &  0.02\% \\
 $j_n^*$ (\scalebox{0.9}{$\widehat C^*_n$}) & 1.06\% & 27.15\% & 55.78\% &  14.45\% &  1.45\% &  0.10\% & 0.01\%\\\bottomrule
\end{tabular} \vspace{2ex}
\end{footnotesize}
\caption{\label{t:counts} \footnotesize Empirical distribution of $j_n$ and $j_n^*$ at the $T=1/12$ (1 month) time horizon.  Reported values reflect the empirical percentages of the aggregated counts of iterations until stabilization across all computed values of each estimator (for $\widehat C_n$, across both settings $(3)$ and $(4)$; for $\widehat C_n^*$,  across both (5ab) and (6ab)) and across Models 1-4.  }
\end{table}

Regarding computational considerations, in Table \ref{t:counts} we report the empirical distribution of the number of iterations required for stabilization for both the localized thresholding and uniform thresholding approaches (i.e., $j_n$, as in \eqref{e:def_jn}, and $j_n^*$ as in \eqref{e:stablizing_jnstar})  across all models with jumps (Models 1-4) on 1-month time horizons. Both $\widehat C_n$ and $\widehat C_n^*$  stabilize rather quickly, with roughly 98\% of all estimates stabilizing in 4 or fewer iterations, and the global and local thresholding approaches take roughly the same number of iterations.  Though not included in Table \ref{t:counts}, in the jump-free case (Model 5), all observed instances of estimators stabilized in 3 or fewer iterations, with the vast majority taking 1 or 2; also, shorter time horizons typically required fewer iterations to stabilize in all settings.
 
	   Unreported simulation studies suggest  localized estimators can have further performance gains relative to uniform thresholding approaches when the time horizon is extended or when additional inhomogeneities are incorporated into the model such as volatility jumps.  Generally performance improvement of $\widehat C_n$ and $\widehat C_n^*$ relative to standard-type TRV tuning (1) and (2) becomes more dramatic as the overall proportion of jump variation increases relative to the quadratic variation of $X$, or when the activity of either jump component ($L$ or $J$) is increased, and substantive performance gains are typically observed provided at least one of these components is present. We also remark that at daily horizons, with relatively small sample size ($n=78$)  there is little difference between uniform thresholding (3)-(4) and the localized thresholding (5)-(6), except for the rates $r_n$ and $r_n^*$; not included in this study is a detailed examination of the optimal choice of $k_n$, which could be of future interest, though $k_n = h_n^{-0.6}$ seems to reasonably well in most scenarios.

\appendix
\section{Proofs}\label{Sec:proof}

Throughout the proofs, we often omit the subscript $n$ in $h_n$ and $y_n$, and $K$ denotes a generic constant that may change from line to line.   Without loss of generality, we assume that $T=1$ (giving $h_n=\frac{1}{n}$), and for notational simplicity when dealing with boundary terms we set $\sigma_t:=0$ for $t\notin[0,1]$.  Also, for any process $(Y_t)_{t\in[0,1]}$, we set $\Delta_i^n Y=0$ whenever $i\leq 0$ or $i>n$.

By a standard localization argument, we may assume without loss of generality that $b,\sigma, \sigma^{-1},$ and $\gamma$ are bounded above by a nonrandom constant, and that
$$
\bE( |\gamma_{t+s }-\gamma_{t}|^2|\mathcal F_t)\leq K s,\quad  s,t\in (0,T]. %
$$
 We also collect some useful estimates below that are used throughout the appendix. Based on the decomposition of  the process $L$ in \eqref{eq:L_decompose}, we have the following as  $n\to\infty$ (and thus as $y_n\to 0$, $h_n\to 0$):
 \begin{equation}\label{e:estimates0}
\begin{gathered}
\lambda_n(y)\sim Ky^{-\alpha}, \quad \bP( \Delta_i^n N(y) \neq0) \sim  K h y^{-\alpha},\\
|\Delta_i^nb(y)| \leq K h { (y^{1-\alpha}+1)}.%
\end{gathered}
\end{equation}
In our arguments, we also need oracle  analogs of the estimator $\widehat C^{*}_n$ in (\ref{DfnOptCLoc}).  To construct them, we first define
\begin{align*}%
\widetilde \sigma^2_n(\ell;B,y) &:= \frac{n}{k_n} \sum_{i=\ell-k_n/2+1}^{\ell +k_n/2}  \big(\Delta_i^n X\big)^2{\mathbf 1}_{\{|\Delta_i^n X|\leq B,~\Delta_i^n N(y) = 0,~\Delta_i^n N' = 0\}},\quad \ell=1,\ldots n.
\end{align*}
We then define an auxiliary  sequence $\widetilde c_{n,j}(i;y)$ through a fixed-point iteration for the function $\xi \mapsto \widetilde \sigma^2_n(\ell;\sqrt{r_n^* \xi},y)$ as follows: we set
$$
\widetilde c_{n,1}(i;y):= 
\widetilde \sigma^2_n (i;B^*_{n,0}(i),y),\quad i=1,\ldots,n,%
$$ 
 where $B^*_{n,0}(i)$ is defined as \eqref{itrtiveOptLocEstb} with the initializing threshold constants $\widehat c_{n,0}(i),$ $i=1,\ldots,n$,  being \textit{the same} as those for $\widehat C^*_n$.  Next, we define oracle counterparts of the iterative localized estimates $\widehat c_{n,j}(i)$ in \eqref{itrtiveOptLocEstb}, for every $1 \leq i \leq n$ and $j\geq{}2$ by setting
\begin{align}
\wt B^{*}_{n,j-1}(i;y) &:= \sqrt{r^*_n \widetilde c_{n,j-1}(i;y), \label{e:Bj(i;y)}}\\
\widetilde  c_{n,j}(i;y)&:=\widetilde \sigma^2_{n}\big(i;\wt B^{*}_{n,j-1}(i;y),y\big).\notag 
\end{align}
 We finally define oracle counterparts of the time-varying threshold iterates $\{ \widetilde C^{*}_{n,j}, j\geq 1\}$  by setting %
\begin{equation}\label{e:def_spot_oracle}
 \widetilde C^{*}_{n,j}(y) = \sum_{i\in \mathcal I_n(y)} (\Delta_i^nX)^2 \,{\bf 1}_{\{|\Delta_i^n X|\leq  \wt B^{*}_{n,j}(i;y)\}},\quad j \geq 1,
\end{equation}
 where $\mathcal I_n(y)$ is given in  \eqref{e:I_n(y)}. The above oracle sequence plays an analogous role to the oracle sequence $\widetilde C_{n,j}(y)$  of \eqref{def:C_tilde_y}, but for the case of time-varying thresholds.

We first establish the following key intermediate result, which shows that the iterative sequences $\wt C_{n,j}(y)$, $\wt C^*_{n,j}(y)$ (defined in \eqref{def:C_tilde_y}  and  \eqref{e:def_spot_oracle}, respectively) are asymptotically equivalent to the oracle version of TRV, $\mathscr C_n(y)$, defined in \eqref{e:oracle(y)}, provided $y\to 0$ at an appropriate rate. Below, we use the notation $\overline \sigma^2=\sup_{0\leq s \leq 1}\sigma_s^2$.
\begin{proposition} \label{prop:allequal} 
Assume the initial estimates $\widehat C_{n,0}$ and $\hat c_{n,0}(i)$ belong to the classes $\mathcal C$ and $\mathcal C^\text{spot}$, respectively. Let $y\to 0$ so that, for some  $\delta\in\big(0,(\frac{1}{\alpha} \wedge 1)-\frac12 \big)$,
\begin{gather}\label{cond:ylu2}
    h^{(\frac{1}{\alpha} \wedge 1)-\delta}  \ll y  \ll (h\log n)^{\frac12}.%
\end{gather}
Further suppose that  $n^a\ll k_n \ll n$ for some $0<a<1$, and
\begin{gather}\label{cond:ylu3}
 \liminf_{n\to\infty}  \frac{r_n}{2 h \log n}>\frac{\overline \sigma^2}{\int_0^1\sigma_s^2ds} \quad \textnormal{and} \quad  \liminf_{n\to\infty}\frac{r^*_n}{2 h \log n}>1.%
\end{gather}
Then, with probability tending to 1,
\begin{equation*}%
\mathscr C_n(y) = \wt C_{n,j}(y) =\wt C^*_{n,j}(y),\quad\textnormal{for all}~ j \geq 1.%
\end{equation*}
\end{proposition}
\begin{proof}
 Recall the notation $\mathcal{I}_n(y)$ as in \eqref{e:I_n(y)}, and let 
 $$\widetilde{\mathbf c}_{n,j}(y)=(\widetilde c_{n,j}(1;y),\ldots,\widetilde c_{n,j}(n;y)).$$
 Also, by analogy to $\mathscr C_n(y)$, define
\begin{align}\label{AnlgDfnc}
\mathscr{c}_n(i;y) = \frac{n}{k_n} \sum_{\ell= i-k_n/2}^{i+k_n/2}\big( \Delta_\ell^n X\big)^2{\bf 1}_{\{\Delta_\ell^n N(y) = 0,\Delta_\ell^n N' = 0\}},\quad i=1,\ldots,n.%
\end{align}
First we claim it suffices to show that, with probability tending to 1,
\begin{align}\label{Eq1a}
&|\Delta_i^n X|^2\leq  r_n \widehat C_{n,0},\text{ for all } i\in \mathcal{I}_n(y),\\
\label{Eq1b}
&|\Delta_i^n X|^2\leq  r_n\mathscr{C}_n(y),\text{ for all } i\in \mathcal{I}_n(y),\\
\label{Eq1c}
&|\Delta_i^n X|^2 \leq r_n^* \widehat c_{n,0}(i),\text{ for all } i\in \mathcal{I}_n(y),\\
\label{Eq1d}
&|\Delta_i^n X|^2\leq  r^*_n \mathscr{c}_n(i;y),\text{ for all } i\in \mathcal{I}_n(y).
\end{align}
Indeed, recalling $B_{n,0}= \sqrt{r_n \widehat C_{n,0}}$ and the definitions of  $\mathscr C_{n}(y)$ and $\widetilde C_{n,1}(y)$ as in \eqref{e:oracle(y)} and \eqref{def:C_tilde_y}, respectively, expression \eqref{Eq1a} implies that $\wt C_{n,1}(y)=\mathscr C_{n}(y)$, while \eqref{Eq1b} implies that $\wt C_{n,j}(y)=\wt C_{n,1}(y)$, for any $j\geq{}2$.
 In a similar fashion, \eqref{Eq1c} implies that $ \widetilde c_{n,1}(i;y)=\mathscr c_{n}(i;y)$ for all $i$. From \eqref{e:Bj(i;y)} this immediately gives $\wt B_{n,1}^{*}(i;y)=\sqrt{r^*_n \mathscr c_{n}(i;y)}$  for all $i$, which, from \eqref{e:def_spot_oracle} and using \eqref{Eq1d}, implies that $\widetilde C^{*}_{n,1}(y)=\mathscr C_n(y)$. Continuing, this gives $\widetilde c_{n,2}(i;y)=\mathscr c_{n}(i;y)$ for all $i$, and thus, $\wt B^*_{n,2}(i;y)=\sqrt{r^*_n \mathscr c_n(i;y)}$ for all $i$; proceeding by induction, we conclude that $\widetilde C^{*}_{n,j}(y)=\mathscr C_n(y)$ and $\widetilde{\mathbf c}_{n,j}(y)=(\mathscr c_{n}(1;y),\ldots,\mathscr c_{n}(n;y))$ for all $j\geq{}1$.

We first establish  \eqref{Eq1a} and \eqref{Eq1b}.  It suffices to show that for some small (nonrandom) $\eta>0$, with probability tending to 1,
\begin{equation}\label{e:P(Omeg_0)to1}
\begin{aligned}
 \frac{\max_{i\in\mathcal I_n (y)}(\Delta_i^n X)^2}{r_n \widehat C_{n,0}}  \leq 1-\eta,\quad \text{and}\quad  \frac{\max_{i\in\mathcal I_n (y)}(\Delta_i^n X)^2}{r_n \mathscr C_n(y)}\leq 1-\eta.
 \end{aligned}
\end{equation}
Note that, for each $i\in \mathcal I_n(y)$, $\Delta_i^n J =0$, and for $H_t(y)$ as in the decomposition \eqref{eq:L_decompose},   we have  $  \int_{t_{i-1}}^{t_i} \gamma_t H_t(y)=0$.  So, denoting $(\sigma\!\cdot\! W)_t = \int_0^t \sigma_s dW_s$, if we define%
\begin{equation}\label{e:delta_i^nX_over_In(y)}
 \Delta_i^n \chi(y):=\int_{t_{i-1}}^{t_i} \gamma_t dM_t(y) + \Delta_i^n b_t(y),
\end{equation}
then we have
$\Delta_i^n X -  \Delta_i^n (\sigma \!\cdot\! W) = \Delta_i^n \chi(y)$ for all $i\in \mathcal I_{n}(y)$.
To show \eqref{e:P(Omeg_0)to1}, note that
\begin{align*}
\max_{i\in \mathcal I_n (y) }(\Delta_i^n X)^2   & \geq  
 \max_{i\in \mathcal I_n (y) }\Big(( \Delta_i^n  (\sigma \!\cdot\! W))^2 -\big|\Delta_i^n \chi(y)\big| ( 2 \big|\Delta_i^n (\sigma \!\cdot\! W)\big | +  \big| \Delta_i^n \chi(y) \big|  ) \Big)\\
&\geq  \max_{i\in \mathcal I_n (y) } ( \Delta_i^n (\sigma \!\cdot\! W))^2 -\max_{i\in \mathcal I_n (y) }\Big\{\big|\Delta_i^n \chi(y)\big|\big( 2 |\Delta_i^n (\sigma \!\cdot\! W) | +  | \Delta_i^n \chi(y) |  \big) \Big\},
\end{align*}
and clearly, 
$$\max_{i\in \mathcal I_n (y) }(\Delta_i^n X)^2  \leq \max_{i\in \mathcal I_n (y) } ( \Delta_i^n (\sigma \!\cdot\! W))^2 +\max_{i\in \mathcal I_n (y) }\Big\{\big|\Delta_i^n \chi(y)\big|\big( 2 |\Delta_i^n (\sigma \!\cdot\! W) | +  | \Delta_i^n \chi(y) |  \big) \Big\}.$$
We first bound $\max_{i\in \mathcal I_n (y) }\big\{\big|\Delta_i^n \chi(y)\big|\big( 2 |\Delta_i^n (\sigma \!\cdot\! W) | +  | \Delta_i^n \chi(y) |  \big) \big\}$.  %
Thus,
\begin{align}
 &h^{-1}\Big|\max_{i\in \mathcal I_n (y) }(\Delta_i^n X)^2 -  \max_{i\in \mathcal I_n (y) }(\Delta_i^n (\sigma\! \cdot\! W))^2\Big|\nonumber\\
 &\leq  h^{-1}\max_{i\in \mathcal I_n (y) }\Big\{\big|\Delta_i^n \chi(y)\big|\Big( 2 |\Delta_i^n (\sigma\! \cdot\! W) | +  | \Delta_i^n \chi(y) |  \Big)\Big\}\nonumber\\
&\quad \leq h^{-1} \max_{1\leq i \leq n} \big|\Delta_i^n \chi(y)\big|\Big( 2 \max_{1\leq i \leq n}|\Delta_i^n (\sigma\! \cdot\! W) | + \max_{1\leq i \leq n} | \Delta_i^n \chi(y) |  \Big) \nonumber\\
& \quad = h^{-1} O_P(y)\left(O\Big( \sqrt{ h\log (1/h)}\Big) +O_P(y) \right)\nonumber\\
& \quad= O_P \Big(y h^{-1/2}(\log n)^{1/2}\Big ) + O_p(h^{-1} y^2 ) \nonumber\\
& \quad=o_P(\log n),\label{eq:maxX_maxW}
 \end{align}
where $\max_{1\leq i \leq n}|\Delta_i^n (\sigma\! \cdot\! W)| =O\big(\sqrt {h\log(1/h)}\big)$ holds as a consequence of Lemma \ref{l:mod_of_cont}, and the last line holds since $y\ll (h\log n)^{1/2}$.%

 We are now in position to show \eqref{Eq1a}--\eqref{Eq1b}.  Let $\widetilde y_n\geq 0$ be a sequence with $h^{(\frac{1}{\alpha} \wedge 1)-\delta}\ll \widetilde y_n\ll h^{\frac12}(\log n)^{\frac12 - \frac1{\delta_0}}$, where $\delta_0$ is as in \eqref{e:Lipbound}.  With  $\underline C_{n}=\underline C_n(X,\widetilde y_n)$ as defined in Proposition \ref{p:prop3_IV}, we then have $\widehat{C}_{n,0}\geq \underline{C}_{n}$ and $\underline{C}_n\stackrel{P}{\longrightarrow}\int_0^{1}\sigma_s^2 ds$.  Furthermore, condition \eqref{cond:ylu3} implies, for some small $\delta'>0$, for large enough $n$, $\frac{r_n}{2h \log n}\geq (1+\delta') \Big(\frac{\overline \sigma^2}{\int_0^1\sigma_s^2ds}\Big)$.  Thus,
\begin{align}
&\frac{\max_{i\in \mathcal I_n (y) }(\Delta_i^n X)^2}{  r_n \wh C_{n,0}}\notag\\
& \leq \bigg(\frac{2}{2(1+\delta')}\bigg)\bigg(\frac{\int_0^1 \sigma^2_s ds}{ \underline{C}_{n}}\bigg) \frac{\max_{i\in \mathcal I_n (y) }(\Delta_i^n X)^2}{ 2\overline \sigma^2 h \log n}\notag\\
& \leq  \bigg(\frac{1}{1+\delta'}\bigg)\bigg(\frac{\int_0^1 \sigma^2_s ds}{ \underline{C}_{n}}\bigg)\Bigg(\frac{\max_{i\in \mathcal I_n (y) } ( \Delta_i^n (\sigma \!\cdot\! W))^2 }{2\overline \sigma^2 h \log n} + \frac{\Big|\max_{i\in \mathcal I_n (y) }(\Delta_i^n X)^2 - \max_{i\in \mathcal I_n (y) } ( \Delta_i^n (\sigma \!\cdot\! W))^2\Big| }{2\overline \sigma^2 h\log n}\Bigg)\notag\\
& = \bigg(\frac{1}{1+\delta'}\bigg)\bigg(1 + o_P(1)\bigg)\bigg( \frac{\max_{i\in \mathcal I_n (y) } ( \Delta_i^n (\sigma \!\cdot\! W))^2 }{2\overline \sigma^2 h \log n}+ o_P(1)\bigg).\label{Eq1TSIn}
\end{align}
Applying Lemma \ref{l:mod_of_cont}, we have, for every $\eta'>0$, $\frac{\max_{i\in \mathcal I_n (y) } ( \Delta_i^n (\sigma \!\cdot\! W))^2 }{2\overline \sigma^2 h \log n}<1+\eta'$ with probability tending to 1. Thus, taking $\eta'>0$ small enough, the right-hand side in \eqref{Eq1TSIn} is strictly less than $1-\eta$  with probability tending to one for some small $\eta>0$, and we obtain the first statement in \eqref{e:P(Omeg_0)to1};  the  statement \eqref{Eq1a} then follows. The second statement in \eqref{e:P(Omeg_0)to1} is proved along the same lines as in \eqref{Eq1TSIn}, replacing both $\wh C_{n,0}$ and $\underline{C}_{n}$ with $\mathscr C_n(y)$ and applying Proposition \ref{p:prop2}-(i), giving \eqref{Eq1b}.

We now show \eqref{Eq1c}-\eqref{Eq1d}. Condition \eqref{cond:ylu3} implies, for some small $\delta'>0$, for large enough $n$, $\frac{r^*_n}{2h \log n}\geq (1+\delta')$.
With $\underline c_{n}(i)=\underline c_{n}(i;X,\widetilde y_n)$ as in Proposition \ref{p:prop3}, we have $\widehat c_{n,0}(i)\geq \underline c_{n}(i)$ and thus
\begin{align}\label{AgnWeUs}
&\max_{i\in \mathcal I_n (y) }\frac{(\Delta_i^n X)^2}{  r_n^*  \widehat c_{n,0}(i)}\\
& \leq\bigg(\frac{1}{1+\delta'}\bigg)\max_{1\leq i \leq n} \bigg(\frac{ \sup_{t\in [\frac{i-k_n/2}{n}, \frac{i+k_n/2}n)} \sigma_t^2 }{ \underline{c}_{n}(i)}\bigg) \max_{i\in \mathcal I_n (y) }\frac{(\Delta_i^n X)^2}{ 2\big(\sup_{t\in [\frac{i-k_n/2}{n}, \frac{i+k_n/2}n)} \sigma_t^2\big) h \log n}\notag\\ %
\nonumber
& \leq \bigg(\frac{1}{1+\delta'}\bigg)  \bigg(1+o_P(1)\bigg)\bigg(\max_{i\in \mathcal I_n (y) } \frac{ ( \Delta_i^n (\sigma \!\cdot\! W))^2 }{2\big(\sup_{t\in[\frac{i-k_n/2}{n}, \frac{i+k_n/2}n)} \sigma_t^2\big) h \log n}+ o_P(1)\bigg),%
\end{align}
where on the last line we applied Proposition \ref{p:prop3}.  Together with Lemma \ref{l:mod_of_cont}, this shows, for small enough $\eta>0$, with probability tending to 1,
$$
\max_{i\in\mathcal I_n (y)}\frac{(\Delta_i^n X)^2}{ r^*_n \widehat c_{n,0}(i)}  \leq 1-\eta,
$$
which in turn implies \eqref{Eq1c} holds with probability tending to 1. The statement \eqref{Eq1d} is shown along the same lines of \eqref{AgnWeUs}, replacing $\widehat c_{n,0}(i)$ with $\mathscr c_{n}(i;y)$ in \eqref{AgnWeUs},  and $\underline{c}_{n}(i)$  with  $\mathscr c_{n}(i;\widetilde y_n\wedge y)$ (giving $\mathscr c_{n}(i;y)\geq \mathscr c_{n}(i;\widetilde y_n\wedge y)$) and applying Proposition \ref{p:prop3} to $\mathscr c_{n}(i;\widetilde y_n\wedge y)$.
\end{proof}

\begin{proof}[Proof of Theorem \ref{thm:consistency_clt_con}]
We begin by laying out some arguments used across all cases (i)--(iii). First note that for each $n$, the sequence $\widehat C_{n,j}$ (and hence, $B_{n,j}$) is either nonincreasing or nondecreasing in $j$. Indeed, suppose that $\widehat C_{n,0}\leq \widehat C_{n,1}$. Then, $B_{n,0}\leq B_{n,1}$, giving
\begin{align*}
	\widehat C_{n,2} &= \sum_{i=1}^n (\Delta_i^n X)^2 {\bf 1}_{\{|\Delta_i^n X|\leq  B_{n,1}\}}\geq\sum_{i=1}^n (\Delta_i^n X)^2 {\bf 1}_{\{|\Delta_i^n X|\leq  B_{n,0}\}}=\widehat C_{n,1}. 
\end{align*}
Then, $B_{n,1}\leq B_{n,2}$ and we can proceed by induction to conclude that $\widehat C_{n,j}$ is nondecreasing in $j$. If $\widehat C_{n,0}\leq \widehat C_{n,1}$, we can follow the same argument to show that $\widehat C_{n,j}$ is nonincreasing in $j$.
Also, since for each $n$, the function $B\mapsto \sum_{i=1}^n(\Delta_i^nX)^2 \,{\bf 1}_{\{|\Delta_i^n X|\leq  B \}}$ takes at most $n+1$ possible values, it holds that 
\begin{align*}%
    \widehat C_{n}=\widehat C_{n,n+1},\quad  \textnormal{a.s.}
\end{align*}

For any $y>0$, we will  now show that{,} for each $j\geq  1$, we have $\widehat C_{n,j} \geq { \widetilde{C}_{n,j}(y)}$.  First note the inequality  $\widehat C_{n,1} \geq \widetilde{C}_{n,1}(y)$ is straightforward by definitions \eqref{IMDE0}  and \eqref{def:C_tilde_y}, which implies $B_{n,1} \geq \wt B_{n,1}(y)$. 
Proceeding by induction, suppose that for some  $j\geq 1$ we have $\widehat C_{n,j} \geq  \widetilde C_{n,j}(y)$. Then, by definition, $B_{n,j} \geq \wt B_{n,j}(y)$, implying
\begin{align}
\widehat C_{n,j+1} &= \sum_{i=1}^n (\Delta_i^nX)^2 \,{\bf 1}_{\{|\Delta_i^n X|\leq  B_{n,j}\}} \notag\\
&\geq \sum_{i=1}^n (\Delta_i^nX)^2 \,{\bf 1}_{\{|\Delta_i^n X|\leq  \wt{B}_{n,j}(y)\}}\notag \\
&\geq \sum_{ i\in\mathcal I_n(y)} (\Delta_i^nX)^2 \,{\bf 1}_{\{|\Delta_i^n X|\leq  \wt{B}_{n,j}(y)\}} = \widetilde C_{n,j+1}(y).\label{e:inductive_Cjn}
\end{align}
Therefore, for all $j\geq 1$, $\widehat C_{n,j} \geq \widetilde C_{n,j}(y)$; in particular, $\widehat C_{n}=\widehat C_{n,n+1} \geq \widetilde C_{n,n+1}(y)$. %
Next, we decompose $\widehat C_n$ as
\begin{align*}
    \widehat C_n &= \sum_{i=1}^n (\Delta_i^nX)^2 \,{\bf 1}_{\{|\Delta_i^n X|\leq  B_{n,n}\}} \\
    &= \sum_{i=1}^n (\Delta_i^nX)^2 \Big({\bf 1}_{\{\Delta_i^n N(y)=0,\, \Delta_i^n N' = 0\}} + {\bf 1}_{\{\Delta_i^n N(y)=0,\, \Delta_i^n N' \neq 0,\,|\Delta_i^n X|\leq  B_{n,n}\}} \\
    &\qquad \qquad \qquad \qquad - {\bf 1}_{\{\Delta_i^n N(y)=0, \,|\Delta_i^n X| > B_{n,n},\, \Delta_i^n N' = 0\}} + {\bf 1}_{\{\Delta_i^n N(y)\neq 0,\,|\Delta_i^n X|\leq  B_{n,n}\}}\Big)\\
    &\leq  \sum_{i=1}^n (\Delta_i^nX)^2 \Big({\bf 1}_{\{\Delta_i^n N(y)=0,\, \Delta_i^n N' = 0\}} + {\bf 1}_{\{\Delta_i^n N(y)=0,\, \Delta_i^n N' \neq 0,\,|\Delta_i^n X|\leq  B_{n,n}\}} \\
    &\qquad \qquad \qquad \qquad + {\bf 1}_{\{\Delta_i^n N(y)\neq 0,\,|\Delta_i^n X|\leq  B_{n,n}\}}\Big)\\
    &= \mathscr C_n(y) + R_n,%
\end{align*}
where $R_n = \sum_{i=1}^n (\Delta_i^nX)^2\big({\bf 1}_{\{\Delta_i^n N(y)\neq 0,\,|\Delta_i^n X|\leq  B_{n,n}\}} + {\bf 1}_{\{\Delta_i^n N(y)=0,\, \Delta_i^n N' \neq 0,\,|\Delta_i^n X|\leq  B_{n,n}\}} \big) $.  Thus,
\begin{align}\label{eq:squeeze}
    \wt C_{n,n+1}(y)\leq \widehat C_n\leq \mathscr C_n(y) + R_n.
\end{align}
Further, with the sequence of events $\Omega_n:=\{\wt C_{n,n+1}(y)=\mathscr C_n(y)\}$, we may write expression \eqref{eq:squeeze} as 
\begin{align}\label{eq:squeeze2a}
  \mathscr C_n(y)+{\bf 1}_{\Omega_n^c}\big(\wt C_{n,n+1}(y)-\mathscr C_n(y)  \big)\leq \widehat C_n \leq \mathscr C_n(y) + R_n.
\end{align}
We now turn to the statement (i), in which case we recall $r_n$ is assumed to satisfy
\begin{align}\label{eq:rn_toprob}
  h \log (1/h)\ll r_n \ll \left(h\log (1/h)\right)^{\frac\alpha2}.%
\end{align}
Under this assumption, we may choose a sequence $y_n\to0$ such that, for some small $\delta>0$,
\begin{equation}\label{e:y_in_main_proof}
r_n^{\frac1\alpha} \vee h^{(\frac{1}{\alpha}\wedge 1)-\delta} \ll y_n \ll (h\log n)^{\frac12}, %
\end{equation}
under which the hypotheses of Proposition \ref{p:prop2}(i) are satisfied, giving $\mathscr C_{n}(y) \stackrel{P}\to C_T$.

Now, since $r_n \gg h\log n $,  clearly $r_n\geq 2\Big(\frac{\overline \sigma^2}{\int_0^1\sigma_s^2ds} \Big)h \log n$ a.s. for all large $n$, showing the hypotheses of Proposition \ref{prop:allequal} are satisfied, giving $\bP(\Omega_n)\to 1$. Further, by the assumption \eqref{e:y_in_main_proof}, $y^\alpha\gg r_n$, giving
$$\bE\left(N_1(y)   r_n  \right) = { O\left( y^{-\alpha}   r_n  \right) }= o(1). $$
 Therefore,
\begin{align*}
    0\leq R_n &\leq N_1(y)B_{n,n}^2 + N'_1 B_{n,n}^2=N_1(y)  r_n  \widehat C_{n,n-1}+ N'_1  r_n \widehat C_{n,n-1} \toprob 0,
\end{align*}
where above we used that $\widehat C_{n,n-1}=
	\sum_{i=1}^n (\Delta_i^n X)^2 {\bf 1}_{\{|\Delta_i^n X|\leq  B_{n,n-1} \}} \leq\sum_{i=1}^n (\Delta_i^nX)^2=O_{P}(1)$.
Since ${\bf 1}_{\Omega_n^c}\big(\mathscr C_n(y)-\wt C_{n,n+1}(y)\big)=o_P(1)$, expression  \eqref{eq:squeeze2a} and $R_n=o_P(1)$ give (i).

\medskip
We now establish (ii), in which case,  we recall $r_n$ is assumed to satisfy
\begin{align}\label{eq:rn_todist}
 h \log (1/h) \ll r_n \ll  h^{\frac{\alpha+1}{2}} (\log (1/h))^{\frac\alpha2}.
\end{align}
Under these constraints, we may again select a sequence $y\to 0$ satisfying \eqref{e:y_in_main_proof}.
 By \eqref{eq:squeeze2a}, we have
 \begin{align}
&\sqrt{n}\left(\mathscr C_n(y)  - C_T\right)+\sqrt n {\bf 1}_{\Omega_n^c}\big(\wt C_{n,n+1}(y)-\mathscr C_n(y)  \big)\notag\\
&\quad \leq \sqrt{n}\left(\widehat C_n   - C_T\right)\notag\\
&\quad\leq \sqrt{n}\left(\mathscr C_n(y)  - C_T\right)  +\sqrt n R_n.\label{eq:squeeze_clt}
\end{align}
Note that, for the right side of \eqref{eq:squeeze_clt}, the condition \eqref{eq:rn_todist} gives $r_n\ll n^{-1/2} y^\alpha $, which yields%
\begin{align*}%
    \sqrt nR_n &=\sqrt{n}\sum_{i=1}^n (\Delta_i^nX)^2\left({\bf 1}_{\{\Delta_i^n N(y)\neq 0,\,|\Delta_i^n X|\leq  B_{n,n}\}} + {\bf 1}_{\{\Delta_i^n N(y)=0,\, \Delta_i^n N' \neq 0,\,|\Delta_i^n X|\leq  B_{n,n}\}} \right)\\
    &\leq \sqrt{n} N_1(y)  r_n  \widehat C_{n, n} + \sqrt{n} N'_1  r_n \widehat C_{n, n} \\
    & = O_P( n^{1/2}r_ny^{-\alpha}) + O_P(n^{1/2}r_n)\toprob 0.%
\end{align*}
Moreover, since $\alpha<1$, we have $y\ll (h\log n)^{1/2}\ll h^{\frac{1}{2(2-\alpha)}}$, implying the hypotheses of Proposition \ref{p:prop2}(ii) are satisfied, and the hypotheses of  Proposition \ref{prop:allequal} are clearly also satisfied. Statement (ii) then follows from Proposition \ref{p:prop2}(ii), Proposition \ref{prop:allequal}, and the string of inequalities in \eqref{eq:squeeze_clt}, since $\bP(\Omega_n^c)\to 0$, implies
\begin{equation}\label{e:sqrt_oracle_op(1)}
\sqrt n {\bf 1}_{\Omega_n^c}\big(\wt C_{n,n+1}(y)-\mathscr C_n(y)  \big) = o_P(1).
\end{equation}
For statement (iii), we now take $y\to 0$ such that
$$
h^{\frac{1}{2}}\ll y \ll (h\log n)^{\frac{1}{2}}.
$$
In particular, under such choice of $y$, from  Proposition \ref{p:prop2}(iii) we have $\sqrt{n}\left(\mathscr C_n(y)  - C_T\right)\toprob \infty$. Again from \eqref{eq:squeeze2a}, we have the string of inequalities \eqref{eq:squeeze_clt}, and since we still have $r_n \gg h \log(1/h)$, we may apply Proposition \ref{prop:allequal} to again conclude \eqref{e:sqrt_oracle_op(1)}, and thus the leftmost inequality in \eqref{eq:squeeze_clt} yields the result.%
\end{proof}

\begin{proof}[Proof of Proposition \ref{p:divergence_C_hat}] For simplicity we write $Y$ instead of $Y'$ throughout the proof.  Let $Y_t = \int_0^t \sigma_s dW_s$, where
$$
\sigma_t= a{\bf 1}_{\{t<\theta\}} + b{\bf 1}_{\{t\geq \theta\}},
$$
so that $C_T(Y)= a^2\theta + b^2(1-\theta)$.  Above, the quantities $\theta \in(0,1)$, $0<a<b$ are nonrandom constants chosen so that
 $$\delta :=  \frac{c_0}{2}\cdot \frac{C_T(Y)}{ b^2}%
 = \frac{c_0}{2}\Big(\frac{a^2}{b^2 }\theta + (1-\theta)\Big)  \in\Big(0,\frac{1}{2}\Big).$$
  We have
\begin{align}
\hspace{-2ex}\textnormal{TRV}_n(Y;\vartheta_n) -C_T(Y) &= \Big(\sum_{i=1}^n(\Delta_i^n Y)^2 -C_T(Y)\Big) -  \sum_{i=1}^n(\Delta_i^n Y)^2{\bf 1}_{\{|\Delta_i^n Y|> \vartheta_n\}}\label{e:truncated_bias}\\
 &= O_P(n^{-1/2}) -  \bigg(\sum_{i=1}^{\lfloor n\theta\rfloor }+ \sum_{i=\lfloor n\theta \rfloor+2}^n\bigg)(\Delta_i^n Y)^2{\bf 1}_{\{|\Delta_i^n Y|> \vartheta_n\}}.\notag %
\end{align}
For simplicity let $\mathcal Y_i^n = (\Delta_i^n Y)^2{\bf 1}_{\{|\Delta_i^n Y|> \vartheta_n\}}$.
 For $Z\sim \mathcal N(0,1)$, $\phi(x)={(2\pi)}^{-1/2}e^{-x^2/2}$, an integration by parts shows %
\begin{equation}\label{e:z_2kmoments}
\bE Z^{2k}{\bf 1}_{\{|Z|> x\}} \sim 2 x^{2k-1}\phi(x),\quad x\to \infty.
\end{equation}
Hence, for $i>\lfloor n \theta \rfloor+1$, $\Delta_i^n Y\stackrel{d}=b\sqrt{h}Z$, and since $c_0 C_T(Y)/b^2=2\delta$, 
\begin{align}
 \bE\mathcal Y_i^n&=  \bE (\Delta_i^n Y)^2{\bf 1}_{\{| b \sqrt hZ|> \sqrt{c_0 C_T(Y) h \log (1/h)}\}}\notag\\
 &= b^2 h\bE Z^2{\bf 1}_{\{|Z|> \sqrt{2\delta \log n}\}}\notag\\
 &\sim 2 b^2h\sqrt{2\delta \log n}\, \phi\Big(\sqrt{2\delta \log n}\Big),\label{e:calY_firstmoment}
\end{align}
giving %
\begin{align*}
 \sum_{i=\lfloor n\theta \rfloor+2}^n\bE \mathcal Y_i^n%
& \sim   2 b^2(1-\theta) \sqrt{2\delta \log n} \cdot (2\pi)^{-1/2} n^{-\delta} =:K_0n^{-\delta} \sqrt{\log n}.
\end{align*}
Analogously, for $i\leq \lfloor n\theta\rfloor$,  $\Delta_i^n Y\stackrel{d}=a\sqrt{h}Z$,  and since $c_0 C_T(Y)/a^2=2(b^2/a^2)\delta>2\delta$,  we have $\bE \sum_{i=1}^{\lfloor n\theta\rfloor }\mathcal Y_i^n= O\big(n^{-\delta(b^2/a^2)} \sqrt{\log n}\big)=o\big(n^{-\delta} \sqrt{\log n}\big).$ This implies
\begin{equation}\label{e:calY_firstmoments}
\frac{n^{\delta}}{\sqrt{\log n}}\sum_{i=1}^n\bE\mathcal Y_i^n \to K_0.%
\end{equation}
On the other hand, using \eqref{e:z_2kmoments}, we have,  for $i>\lfloor n\theta\rfloor +1$,
$$
 \bE(\mathcal Y_i^n)^2 = b^4 h^2 \bE Z^4{\bf 1}_{\{|Z|> \sqrt{2\delta \log n}\}} = O\big(n^{-2-\delta}({\log n})^{3/2}\big).
$$
Since \eqref{e:calY_firstmoment} gives $(\bE\mathcal Y_i^n)^2=O(n^{-2-2\delta} \log n)$, we obtain $\Var(\mathcal Y_i^n) = O\big(n^{-2-\delta}({\log n})^{3/2}\big)$ for all $i>\lfloor n\theta\rfloor+1$. 
 Arguing analogously as for \eqref{e:calY_firstmoments}, for $i\leq \lfloor n\theta\rfloor$ we have $ \Var(\mathcal Y_i^n) \ll n^{-2-\delta}({\log n})^{3/2}$.  Thus, 
$$
\Var \bigg(\sum_{i=1}^n \mathcal Y_i^n\bigg)  = O( n^{-1-\delta} (\log n)^{3/2}),
$$
giving
$$
\frac{n^{\delta}}{\sqrt{\log n}}\sum_{i=1}^n( \mathcal Y_i^n-\bE\mathcal Y_i^n) = o_P(n^{(\delta-1)/2}(\log n)^{1/4})=o_P(1).
$$
Thus, from \eqref{e:truncated_bias}, we obtain  $\textnormal{TRV}_n(Y;\vartheta_n)\toprob C_T$ and
\begin{align*}
n^{1/2}\Big(\textnormal{TRV}_n(Y;\vartheta_n) -C_T(Y) \Big)& =O_P(1) - n^{1/2}\bigg(\sum_{i=1}^n \bE\mathcal Y_i^n+ \sum_{i=1}^n( \mathcal Y_i^n-\bE\mathcal Y_i^n)\bigg)\notag\\
&= O_P(1) - \big(n^{\frac{1}{2}-\delta}\sqrt {\log n}\big) \bigg(\frac{n^{\delta}}{\sqrt{\log n}}\sum_{i=1}^n\bE\mathcal Y_i^n + o_P(1)\bigg)\notag\\
&= O_P(1) - \big(n^{\frac{1}{2}-\delta}\sqrt {\log n}\big) \Big(K_0 + o_P(1)\Big)\\
& \toprob -\infty.
\end{align*}
\end{proof}

\begin{proof}[Proof of Theorem \ref{thm:consistency_clt_2}] %
The proof is similar to that of Theorem \ref{thm:consistency_clt_con}; we provide details here where there are substantive differences. We first note that for each $n$ and $i$, $\widehat c_{n,j}(i)$ (and, hence, $B^*_{n,j}(i)$) is either nonincreasing or nondecreasing in $j$. Indeed, suppose that $\widehat c_{n,0}(i)\geq \widehat c_{n,1}(i)$. Then, $B^*_{n,0}(i)\geq B^*_{n,1}(i)$ and thus, 
\begin{align*}
	\widehat c_{n,2}(i)&= \frac{1}{k_n} \sum_{\ell=i-k_n/2+1}^{i +k_n/2} \big(\Delta_\ell^n X\big)^2{\mathbf 1}_{\{|\Delta_\ell^n X|\leq B^*_{n,1}(i)\}}\\
	&\leq \frac{1}{k_n} \sum_{\ell=i-k_n/2+1}^{i +k_n/2} \big(\Delta_\ell^n X\big)^2{\mathbf 1}_{\{|\Delta_\ell^n X|\leq B^*_{n,0}(i)\}}=\widehat c_{n,1}(i). 
\end{align*}
Then, $B^*_{n,1}(i)\geq B^*_{n,2}(i)$ and we can proceed by induction to conclude that $\widehat c_{n,j}(i)$ is nonincreasing in $j$. Analogously, in the case $\widehat c_{n,0}(i)\leq \widehat c_{n,1}(i)$, both sequences $\widehat c_{n,j}(i)$ and $B^*_{n,j}(i)$ are nondecreasing in $j$. As a consequence, we will have that $B^*_n(i)=B^{*}_{n,n+1}(i)$ and 
\[
	\widehat C^{*}_{n} = \sum_{i=1}^n (\Delta_i^nX)^2 \,{\bf 1}_{\{|\Delta_i^n X|\leq  B^*_{n,n+1}(i)\}}.
\]
Next, for any $y>0$, recalling $\widetilde C^*_{n,j}(y)$ as in \eqref{e:def_spot_oracle}, and $\widetilde c_{n,j}(i;y),$ $\widetilde{B}^*_{n,j}(i;y)$ as in \eqref{e:Bj(i;y)}, clearly  $\widetilde c_{n,1}(i;y) \leq \widehat c_{n,1}(i)$, $i=1,\ldots, n$, giving $\widetilde B^*_{n,1}(i;y) \leq B^*_{n,1}(i)$, $i=1,\ldots,n$.  Arguing inductively  in an analogous manner to \eqref{e:inductive_Cjn}, we then obtain
$$
\widetilde B^*_{n,j}(i;y) \leq B^*_{n,j}(i),\quad i=1,\ldots, n, \quad j\geq 1.
$$
Arguing the same fashion as in \eqref{eq:squeeze2a}, we obtain
\begin{align}\label{eq:squeeze3}
   \mathscr C_n(y)+ {\bf 1}_{(\Omega_n^*)^c}\big(\wt C^*_{n,n+1}(y)-\mathscr C_n(y)  \big) \leq \widehat C^{*}_n\leq  \mathscr C_n(y) + R^{*}_n,
\end{align}
where $\Omega^*_n=\{\wt C^*_{n,n+1}(y)=\mathscr C_n(y)\}$ and
$$
R^{*}_n = \sum_{i=1}^n (\Delta_i^nX)^2\big({\bf 1}_{\{\Delta_i^n N(y)\neq 0,\,|\Delta_i^n X|\leq  {B}^{*}_{n,n}(i)\}} + {\bf 1}_{\{\Delta_i^n N(y)=0,\, \Delta_i^n N' \neq 0,\,|\Delta_i^n X|\leq  {B}^{*}_{n,n}(i)\}} \big).
$$

Turning now to (i), choose $y\to 0$ in such a way that \eqref{e:y_in_main_proof} holds with $r_n^*$ replacing $r_n$.  Applying Proposition \ref{prop:allequal},  we obtain  $\bP(\Omega^*_n)\to 1$.  The statement then follows from Proposition \ref{p:prop2}(i), the convergence ${\bf 1}_{\Omega_n^c}\big(\wt C^*_{n,n+1}(y)-\mathscr C_n(y)  \big) =o_P(1)$,  expression \eqref{eq:squeeze3},  and
\begin{align*}
    0\leq R^{*}_n &\leq N_1(y)\Big(\max_{1\leq i \leq n}B^{*}_{n}(i)\Big)^2 + N'_1 \Big(\max_{1\leq i \leq n}B^{*}_{n}(i)\Big)^2\\
    &\leq   O_P(y^{-\alpha}r_n^*) +    O_P(r_n^*) \toprob 0,
\end{align*}
since  $\max_{1\leq i \leq n}(B^{*}_{n}(i))^2/r_{n}^*\leq \max_{1\leq i \leq n}\widehat{c}_{n,n+1}(i)\leq  \max_{1\leq i \leq n}\frac{n}{k_n}\sum_{\ell=i-k_n/2+1}^{i+k_n/2}(\Delta_\ell^n X)^2=O_P(1)$ (see \cite{figueroa-lopez:wu:2022}), and $r^*_n\ll y^\alpha$ by assumption.
Similarly, for (ii), we again take $y\to 0$ such that \eqref{e:y_in_main_proof} is satisfied (with $r_n^*$ replacing $r_n$). Then, from expression \eqref{eq:squeeze3}, we may arrive at \eqref{eq:squeeze_clt} with $\widetilde C^*_{n,j}(y)$ replacing  $\widetilde C_{n,j}(y)$ and $\Omega_n^*$ replacing $\Omega_n$. 
The statement then follows in view of the convergence $\sqrt n{\bf 1}_{\Omega_n^c}\big(\wt C^*_{n,n+1}(y)-\mathscr C_n(y)  \big) =o_P(1),$ Proposition \ref{p:prop2}(ii), and the estimates%
\begin{align*}
    0\leq \sqrt n R^{*}_n &\leq N_1(y)\Big(\max_{1\leq i \leq n}B^{*}_{n}(i)\Big)^2 + N'_1 \Big(\max_{1\leq i \leq n}B^{*}_{n}(i)\Big)^2\\
       &\leq   O_P(\sqrt n y^{-\alpha}r_n^*) +    O_P(\sqrt n r_n^* ) \toprob 0,
\end{align*}
since $r^*_n\ll n^{-1/2} y^\alpha$ by assumption. 
Statement (iii) is proved analogously to Theorem \ref{thm:consistency_clt_con}(iii).
\end{proof}

\section{Auxiliary results}\label{Sec:proofs_of_lemmas}
Throughout this section, for notational simplicity, we again often omit the subscript $n$ in $h_n$ and $y_n$.
\begin{lemma}\label{l:maxM(y)}
Let $m=m_n\geq 0$ be any sequence with $m\gg h y^{1-\alpha}$.  Then for for large enough $n$, for $M_t(y)$ as in \eqref{eq:L_decompose},  we have
$$ \bP\left( |\Delta_i^n M(y)|>m\right)\leq 2\left(\frac{eh\Sigma_y}{my}\right)^{\frac{m}{2y}},
$$
where $\Sigma_y := \int_{|z|\leq y} z^2 \nu(dz)\sim K y^{2-\alpha}$.
\end{lemma}

\begin{proof}
Note that for each fixed $y$, $\Delta_i^n M(y)$ is infinitely divisible with triplet $(0,0,h{\bf 1}_{\{|x|\leq y\}}\nu(dx))$, and in particular $\bE e^{u \Delta_i^n M(y)} <\infty$ for all $u\in \bR$.  Write
$$
s_h(x) = hs(x)=h\int_{|z|\leq y} z\left(e^{xz}-1\right) \nu(dz). 
$$
Note $s(0)=0$ and $s(x)$ is increasing. Let $s_+=\lim_{x\to\infty}s(x)$, and first suppose $s_+=\infty$.  By Lemma 26.4 in \cite{sato:1999}, the inverses $\theta_h=s_h^{-1}$ and $\theta=s^{-1}$ exist on $[0,\infty)$, and %
\begin{align*}
    \bP\left(\Delta_i^n M(y) >k\right) \leq \exp\left\{-\int_{0}^{k} \theta_h(s)ds\right\}=\exp\left\{-h\int_{0}^{k/h} \theta(z)dz\right\}.
\end{align*}
Observe
\begin{align*}
    s(x) \leq \int_{|z|\leq y} xz^2 e^{xz} \nu(dz) \leq xe^{xy}\Sigma_y \leq \frac{e^{2xy}-1}{y}\Sigma_y.
\end{align*}
This implies 
$$
\theta(z) \geq \frac{1}{2y}\log \Big(1+\frac{y}{\Sigma_y}z\Big).
$$
So, we have
\begin{align}
    \bP\left(\Delta_i^n M(y) >m\right) &\leq \exp\left\{-\frac{h}{2y}\int_0^{m/h}\log \Big(1+\frac{y}{\Sigma_y}z\Big)dz \right\}\notag\\
    &= \exp\left\{-\frac{h\Sigma_y}{2y^2}\int_0^{\frac{my}{h\Sigma_y}}\log (1+s)ds \right\}\notag\\
    &\leq \exp\left\{-\frac{m}{2y} \log \frac{my}{eh\Sigma_y} \right\}\notag\\
    &=  \left(\frac{eh\Sigma_y}{my}\right)^{\frac{m}{2y}}.\label{e:Deltam_righttail}
\end{align}
If instead $s_+<\infty$, then necessarily $\nu(0,y]=0$ and $\int_{[-y,0)}|z|\nu(dz)<\infty$.  Thus, \cite[Theorem 24.7]{sato:1999} implies $\Delta_i^n M(y)$ is supported on $(-\infty, a_y]$, with $a_y=h\int_{(-y,0)}|z|\nu(dz)=O(h y^{1-\alpha})=o(m)$, giving $\bP\left(\Delta_i^n M(y) >m\right)=0$ for all large $n$.

By repeating the argument for the cases $s_+=\infty$ and $s_+<\infty$ to $-\Delta_i^n M(y)$, we obtain the bound \eqref{e:Deltam_righttail} for  $\bP\left(-\Delta_i^n M(y) >m\right)$, completing the proof.
\end{proof}

\begin{lemma}\label{l:maxbound_chi}Suppose $y\to 0$ with $y\gg h^{(\frac{1}{\alpha} \wedge 1)-\delta}$ for some $\delta>0$.  Then, with $\Delta_i^n\chi(y)$ as defined in \eqref{e:delta_i^nX_over_In(y)}, 
\begin{equation}\label{e:maxbound_chi}
\max_{1\leq i \leq n} |\Delta_i^n \chi(y)|=o_P(y).
\end{equation}
\end{lemma}
\begin{proof}

First observe %
$$
\max_{1\leq i \leq n}\frac{|\Delta_i^n b_t(y)|}{y} \leq K\left(\frac{h(y^{1-\alpha}+1)}{y}\right)\to 0,
$$
because $h/y\ll h^{\delta}\to{}0$, and $h y^{-\alpha}\to 0$. %
Moreover, Lemma 2.1.5 in \cite{jacod:protter:2011} gives 
\begin{align*}
\bE_{i-1} \bigg|\int_{t_{i-1}}^{t_i}(\gamma_t-\gamma_{t_{i-1}}) dM_t(y)\bigg|^2 & \leq K y^{2-\alpha}\bE_{i-1} \int_{t_{i-1}}^{t_i}|\gamma_t-\gamma_{t_{i-1}}|^2 dt%
\leq K y^{2-\alpha} h^{2}.%
\end{align*}
Since $y \gg hy^{1-\alpha}$, and $|\gamma_t|\leq C$ for some $C>0$, we may then apply Lemma \ref{l:maxM(y)} to obtain, for large enough $\kappa>0$, and all large $n$ %
\begin{align*}
\bP\left( \bigg|\int_{t_{i-1}}^{t_i} \gamma_t dM_t(y)\bigg|>\kappa y\right)
&\leq \bP\left( \bigg|\int_{t_{i-1}}^{t_i}(\gamma_t-\gamma_{t_{i-1}}) dM_t(y)\bigg|>\kappa y/2\right) + \bP\big( |\gamma_{t_i}||\Delta_i^nM(y)| >\kappa y/2\big)\\
&\leq K\bigg( y^{-2} \bE  \bigg|\int_{t_{i-1}}^{t_i}(\gamma_t-\gamma_{t_{i-1}}) dM_t(y)\bigg|^2  +   \left(\frac{hy^{2-\alpha}}{\kappa y^2}\right)^{\frac\kappa{2C}}  \bigg)\\
& \leq K \bigg(  h^2 y^{-\alpha} +  \left(hy^{-\alpha}\right)^{\frac \kappa {2C}}\bigg)\\
&\leq K\bigg(  h^{1+\delta\alpha}  + \left(h^{\delta\alpha}\right)^{\frac\kappa{2C}} \bigg)=o(h).%
\end{align*}
This gives, for all large  $n$, and large enough $\kappa>0$, 
\begin{align*}
    \bP\left( \max_{1\leq i \leq n} |\Delta_i^n \chi(y)|\geq \kappa y \right) &\leq \sum_{i=1}^n\bP\left(|\Delta_i^n \chi(y)|\geq \kappa y\right) \nonumber\\
    &\leq \sum_{i=1}^n\Bigg(\bP\left(\bigg|\int_{t_{i-1}}^{t_i} \gamma_t dM_t(y)\bigg|\geq \kappa y/2\right) + \bP\left( {\max_{1\leq i \leq n}}\frac{|\Delta_i^n b_t(y)|}{y}>\kappa/2\right) \Bigg)\nonumber\\
    &= \sum_{i=1}^n\bP\left(\bigg|\int_{t_{i-1}}^{t_i} \gamma_t dM_t(y)\bigg|\geq \kappa y/2\right)=o(1),%
\end{align*}
which gives \eqref{e:maxbound_chi}.
\end{proof}

For statements ahead, recall
$
\mathscr C_{n}(y)= \sum_{i\in\mathcal I_n(y)} (\Delta_i^nX)^2,%
$
 where $\mathcal I_n(y) =\{i:  \Delta_i^n N(y) = 0,\, \Delta_i^n N' = 0\}$.

\begin{proposition}\label{p:prop2} The following statements hold.
\begin{enumerate}
\item[(i)] Suppose  $y=y_n\to 0$  satisfies $(h\log n)^{\frac{1}{\alpha}}  \vee  h^{1-\delta} \ll y $ for some $\delta\in(0,1)$. Then,%
\begin{equation}\label{e:basecase}
\mathscr C_{n}(y) \stackrel{P}\to C_T.%
\end{equation}
\item[(ii)] Suppose  $\alpha\in (0,1)$, and $y\to 0$ such that $(h^{1/2}\log n)^{\frac{1}{\alpha}}  \vee  h^{1-\delta} \ll y  \lesssim h^{\frac{1}{2(2-\alpha)}}$ for some $\delta\in(0,1)$. Then, %
\begin{equation*}%
 \sqrt n \bigg(\mathscr C_{n}(y) - C_T\bigg)\stackrel{st}\to \mathcal N\bigg(0,2 \int_0^1 \sigma^4_sds\bigg){.}%
\end{equation*}
\item[(iii)] Suppose  $\alpha\in (1,2)$, and $y\to 0$. Then, %
\begin{equation}\label{e:cases_for_y}
\sqrt n \bigg(\mathscr C_{n}(y) - C_T\bigg) \toprob\begin{cases}
 -\infty, &  h^{\frac{1}{2(2-\alpha)}} \vee h^{\frac{1}{\alpha}} \ll y\ll h^{\frac12},\\
  \infty, & y\gg h^{\frac12}.\\
\end{cases} 
\end{equation}
\end{enumerate} 
\end{proposition}
\begin{proof}
 Set
\begin{equation*}%
\tau^2_n(y) = \sum_{i\in\mathcal I_n(y)} \big( \Delta_i^n (\sigma\!\cdot\! W)\big)^2.%
\end{equation*}
Observe that,  by definition of $N(y)$,
$$
	\bE N_1(y)\leq K\int_{|x|>y}|x|^{-\alpha-1}dx=Ky^{-\alpha}.
$$

Since Lemma \ref{l:mod_of_cont} gives $\limsup_{n\to\infty}  \max_{i=1,\dots,n}\frac{(\Delta_i^n (\sigma \cdot W))^2}{h\log n} < \infty$, we obtain%
\begin{align*}%
\begin{split}
 \sum_{i=1}^n ( \Delta_i^n (\sigma\!\cdot\! W))^2{\bf 1}_{\{\Delta_i^n N(y) \neq 0\}} &\leq \max ( \Delta_i^n (\sigma\!\cdot\! W))^2N_1(y)\\
 & =   O_P(h(\log n)y^{-\alpha}) = o_P(1). %
 \end{split}
\end{align*}
Similarly,
$$
\sum_{i=1}^n ( \Delta_i^n (\sigma\!\cdot\! W))^2{\bf 1}_{\{\Delta_i^n N' \neq 0\}} \leq N'_1 \max ( \Delta_i^n (\sigma\!\cdot\! W))^2 \stackrel P\to 0.
$$
Thus,
\begin{align*}
 0\leq \sum_{i=1}^n \big( \Delta_i^n (\sigma\!\cdot\! W)\big)^2 -\tau_n^2(y) &= \sum_{i=1}^n \big( \Delta_i^n (\sigma\!\cdot\! W)\big)^2(1-{\bf 1}_{\{\Delta_i^n N(y) = 0,\, \Delta_i^n N' = 0\}})\\
 &\leq \sum_{i=1}^n  \big( \Delta_i^n (\sigma\!\cdot\! W)\big)^2\big({\bf 1}_{\{\Delta_i^n N(y) \neq 0\}} + {\bf 1}_{\{\Delta_i^n N' \neq 0\}}\big)\stackrel P\to 0,%
\end{align*}
and we obtain
\begin{equation} \label{e:tau2_consistent0} 
\tau^2_n(y) =  \sum_{i=1}^n ( \Delta_i^n (\sigma\!\cdot\! W))^2+ o_P(1)\;  \stackrel{P}\longrightarrow\; \int_0^1 \sigma_s^2ds.%
\end{equation}
On the other hand, recalling the notation \eqref{e:delta_i^nX_over_In(y)}, we have
\begin{align}
\bE \big( \Delta_i^n \chi(y)\big)^2 &\leq K\bigg[ \bE \bigg(\int_{t_{i-1}}^{t_i} \gamma_t dM_t(y) \bigg)^2 + \bE\big(\Delta_i^n b(y))^2 \bigg]\notag\\
& \leq K\big( h y^{2-\alpha} + h^2(y^{1-\alpha}+1)^2\big) = O(h y^{2-\alpha}) + O(h^2).\label{e:chi_momentbound}
\end{align}
Thus,
\begin{align*}
|\mathscr C_n(y) - \tau^2_n(y)| & = \bigg| \sum_{i\in \mathcal I_n (y) }\Big[\big(\Delta_i^n (\sigma \!\cdot\! W)+ \Delta_i^n \chi(y)\big)^2 - (\Delta_i^n (\sigma \!\cdot\! W) )^2\Big]\bigg| \\
& \leq \sum_{i=1}^n\Big(\big(\Delta_i^n \chi(y)\big)^2 + 2 |\Delta_i^n \chi(y)|\Delta_i^n(\sigma \!\cdot\! W)|\Big)\\
& \leq \sum_{i=1}^n \Big(O_P(h y^{2-\alpha}) +O(h^2) + O_P(h y^{\frac{1-\alpha}{2}})\Big)\\%n \max_{i\in \mathcal I_n (y) }\Big\{\big|\Delta_i^n \chi(y)\big|\big( 2 |\Delta_i^n (\sigma \!\cdot\! W) | +  | \Delta_i^n \chi(y) |  \big) \Big\} \\%+ N'_1\big(O(h\log n) + o(k_n^2)\big) \\
& =o_P(1),%
\end{align*}
where the second-to-last line follows from an application of Cauchy-Schwarz. %
 Thus,
$$
\mathscr C_n(y) - \tau^2_n(y) = o_P(1),
$$
 which, by \eqref{e:tau2_consistent0}, establishes \eqref{e:basecase}. 
 
 \medskip
For (ii), note that
  \begin{align}
    &\sqrt{n}\left(\tau^2_n(y)- C_T\right) \notag\\
    &= \sqrt{n}\left(\sum_{i=1}^n  \big( \Delta_i^n (\sigma\!\cdot\! W)\big)^2-C_T\right) - \sqrt{n}\sum_{i=1}^n  \big( \Delta_i^n (\sigma\!\cdot\! W)\big)^2{\bf 1}_{\{\Delta_i^n N(y) \neq 0\}\cup \{\Delta_i^n N' \neq 0\}}\notag\\
    &\toDistSt  \mathcal N\bigg(0,2 \int_0^1 \sigma^4_sds\bigg), \label{eq:tau12_sto}
\end{align}
where the first term of \eqref{eq:tau12_sto} converges to  $ \mathcal N\left(0,2 \int_0^1 \sigma^4_sds\right)$ stably in law by the usual CLT for realized variance and the second term of \eqref{eq:tau12_sto} converges to 0 in probability because

$$
\sqrt{n} \sum_{\Delta_i^n N' \neq 0} \big( \Delta_i^n (\sigma\!\cdot\! W)\big)^2\leq  \sqrt{n} N'_1 \max_{1\leq i\leq n}  \big( \Delta_i^n (\sigma\!\cdot\! W)\big)^2  = \sqrt{n}  \cdot O(h \log n) \to 0,
$$
and
$$
\sqrt{n} \sum_{\Delta_i^n N(y) \neq 0}  \big( \Delta_i^n (\sigma\!\cdot\! W)\big)^2\leq \sqrt{n} N_1(y) \max_{1\leq i\leq n}  \big( \Delta_i^n (\sigma\!\cdot\! W)\big)^2= \sqrt{n}\cdot O_p\left(y^{-\alpha}\right) \cdot  O(h \log n)
=o_P(1),$$
where we used that $h^{\frac{1}{2\alpha}} (\log n)^{\frac{1}{\alpha}}\ll y$. On the other hand, recalling the notation \eqref{e:delta_i^nX_over_In(y)}, we have
\begin{align*}
\sqrt{n}\big|\mathscr C_n(y) - \tau^2_n(y)\big|  & = \bigg| \sum_{\mathcal I_n(y)}\sqrt{n}\Big[\big(\Delta_i^n (\sigma \!\cdot\! W)+ \Delta_i^n \chi(y)\big)^2 - \big(\Delta_i^n (\sigma \!\cdot\! W)\big)^2\Big] \bigg|\\
& \leq 2\sqrt{n}\bigg|\sum_{\mathcal I_n(y)} \Delta_i^n (\sigma \!\cdot\! W) \Delta_i^n \chi(y)\bigg|+  \sum_{\mathcal I_n(y)} \sqrt{n}\big( \Delta_i^n \chi(y)\big)^2\\
& =: T_1 + T_2.
\end{align*}
Now, by \eqref{e:chi_momentbound},  we get
$$
\bE T_2 \leq  K n^{3/2}\big( h y^{2-\alpha} + h^2(y^{1-\alpha}+1)^2\big) \to 0,
$$
where we used that  $y \ll h^{\frac{1}{2(2-\alpha)}}$.
For $T_1$, first note
\begin{align*}
T_1  &\leq 2\sqrt{n}\bigg|\sum_{i=1}^n \Delta_i^n (\sigma \!\cdot\! W) \Delta_i^n \chi(y)\bigg| + 2\sqrt{n}\bigg|\sum_{i=1}^n \Delta_i^n (\sigma \!\cdot\! W) \Delta_i^n \chi(y)\big({\bf 1}_{\{\Delta_i^n N(y) \neq 0\}} + {\bf 1}_{\{\Delta_i^n N' \neq 0\}}\big)\bigg|\\
& = 2\sqrt{n}\bigg|\sum_{i=1}^n \Delta_i^n (\sigma \!\cdot\! W) \Delta_i^n \chi(y)\bigg|  + o_P(1),
\end{align*}
since, using \eqref{e:maxbound_chi}
\begin{align*}
\sqrt n \bigg|\sum_{\Delta_i^n N(y) \neq 0} \Delta_i^n (\sigma \!\cdot\! W) \Delta_i^n \chi(y)\bigg| &\leq \sqrt n N_1(y)\max_{i}|\Delta_i^n (\sigma \!\cdot\! W)| \max_{i}|\Delta_i^n \chi(y)| \\
& = \sqrt n ~\! O_P(y^{-\alpha}) \cdot O(\sqrt{h\log n} ) \cdot O_P(y)\\
& = o_P(y^{1-\alpha} \sqrt{\log n}) = o_P(1),
\end{align*}
and similarly
\begin{align*}
\sqrt n \bigg|\sum_{\Delta_i^n N'\neq 0} \Delta_i^n (\sigma \!\cdot\! W) \Delta_i^n \chi(y)\bigg|  \leq \sqrt n ~\!N'_1 \cdot O(\sqrt{ h \log n}) \cdot O_P(y) = O_P( y\sqrt{\log n}) \toprob 0.
\end{align*}
 Next, we show that
 \begin{equation}\label{e:T1_hardterm}
 \sqrt{n}\bigg|\sum_{i=1}^n \Delta_i^n (\sigma \!\cdot\! W) \Delta_i^n \chi(y)\bigg| = o_P(1),
\end{equation}
which will imply $T_1=o_P(1)$, and consequently that $\sqrt{n}\big(\mathscr C_n(y) - \tau^2_n(y)\big)  = o_P(1)$, and in view of \eqref{eq:tau12_sto} will establish the desired convergence. Write
\begin{align*}
\sqrt{n}\bigg|\sum_{i=1}^n \Delta_i^n (\sigma \!\cdot\! W) \Delta_i^n \chi(y)\bigg|  &\leq \sqrt{n} \bigg|\sum_{i=1}^n\Delta_i^n (\sigma \!\cdot\! W)  \int_{t_{i-1}}^{t_i} \gamma_t dM_t(y) \bigg|  + \sqrt{n}\bigg|\sum_{i=1}^n \Delta_i^n (\sigma \!\cdot\! W) \Delta_i^n b(y)\bigg|  \\
& =: T_{1,1} + T_{1,2}.
\end{align*}
Applying Lemma 2.1.5 in \cite{jacod:protter:2011}, since $y \gg h^{1/\alpha}$, we have, for every $p\geq1$
$$
\bE \bigg|   \int_{t_{i-1}}^{t_i} \gamma_t dM_t(y)\bigg|^{2p} \leq K \big(h y ^{2p-\alpha} + h^p y^{p(2-\alpha)}\big) \leq K hy ^{2p-\alpha}.
$$
This gives, for every $p,q>1$ with $p^{-1} + q^{-1}=1$,
\begin{align}
\bE T_{1,1} \leq \big(\bE T_{1,1}^2\big)^{1/2}&=\sqrt n \Bigg( \sum_{i=1}^n \bE  \bigg[(\sigma \!\cdot\! W)  \int_{t_{i-1}}^{t_i} \gamma_t dM_t(y) \bigg]^2\Bigg)^{1/2}\notag\\
&\leq \sqrt n \Bigg( \sum_{i=1}^n \big(\bE (\sigma \!\cdot\! W)^{2q}\big)^{1/q}  \bigg(\bE \bigg|\int_{t_{i-1}}^{t_i} \gamma_t dM_t(y) \bigg|^{2p}\bigg)^{1/p}\Bigg)^{1/2}\notag\\
&\leq K \sqrt n \bigg( \sum_{i=1}^n h \cdot  (h^{1/p} y^{2-\alpha/p})\bigg)^{1/2}\notag\\%&\leq \sqrt n \Bigg( \sum_{i=1}^n \bE  \bigg[(\sigma \!\cdot\! W)  \int_{t_{i-1}}^{t_i} \gamma_t dM_t(y) \bigg]^2\\ %
 &\leq K  (h y^{-\alpha})^{\frac1{2p}} h^{-\frac12}y\label{e:T11_argument_penultimate}\\
 & \leq K h^{\frac{1}{4p}-\frac12}y,%
 \label{e:T11_argument}
\end{align}
where  for the last inequality we used that $y\gg h^{\frac1{2\alpha}}$.  Since $\frac{1}{2(2-\alpha)}>\frac{1}{4}$,  by  taking $p>1$ close enough to $1$, we have $\frac{1}{2}-\frac{1}{4p}=\frac{1}{4} + (\frac{1}{4}-\frac{1}{4p})<\frac1{2(2-\alpha)}$ giving $y\ll h^{\frac{1}{2(2-\alpha)}} \ll h^{\frac12-\frac1{4p}}$, and thus expression \eqref{e:T11_argument} tends to 0. %
For $T_{1,2}$, with $b_0(y) = \int_{|x|>y}x\nu(dx)$, and
\begin{align*}
\Delta_i^n b(y) &= h\big(b_{t_{i-1}} + \gamma_{t_{i-1}} b_0(y)\big)+ \int_{t_{i-1}}^{t_i} \big(  b_s-b_{t_{i-1}}+ b_0(y)(\gamma_s -\gamma_{t_{i-1}})\big)ds\\
&=: h b_{t_{i-1}}(y) + \widetilde b_i(y),
\end{align*}
we have
$$
T_{1,2} \leq \sqrt{n}\bigg|\sum_{i=1}^n \Delta_i^n (\sigma \!\cdot\! W) \big(h b_{t_{i-1}}(y))\bigg|  + \sqrt{n}\bigg|\sum_{i=1}^n \Delta_i^n (\sigma \!\cdot\! W) \Delta_i^n \widetilde b_i(y)\bigg|.
$$
Since
\begin{align*}
\bE \bigg(\sqrt{n}\sum_{i=1}^n \Delta_i^n (\sigma \!\cdot\! W)( h b_{t_{i-1}}(y) )\bigg)^2 = n \sum_{i=1}^n \bE( \Delta_i^n (\sigma \!\cdot\! W))^2 h^2 \bE b_{i-1}^2(y)\leq K h
\end{align*}
and
\begin{align*}
\sqrt{n}\bE \bigg|\sum_{i=1}^n \Delta_i^n (\sigma \!\cdot\! W) \Delta_i^n \widetilde b_i(y)\bigg| &\leq  \sqrt n \sum_{i=1}^n  \sqrt{\bE \big(\Delta_i^n (\sigma \!\cdot\! W)\big)^2 \bE \big( \widetilde b_i(y)\big)^2}
= o(1),
\end{align*}
we get $T_{1,1}+T_{1,2}=o_P(1)$, which implies \eqref{e:T1_hardterm} and completes the proof of (ii). 

\medskip
We now turn to (iii).  Write%
\begin{align}
\sqrt{n}\left(\mathscr C_n(y)- C_T\right) & = \sqrt{n}\left(\sum_{\mathcal I_n(y)}  \big( \Delta_i^n (\sigma\!\cdot\! W)\big)^2-C_T\right) + \sqrt{n} \sum_{\mathcal I_n(y)} \left(\Delta_i^n M(y)\right)^2+ \sqrt{n} \sum_{\mathcal I_n(y)}\left(\Delta_i^n b (y)\right)^2\notag\\
&\quad + \sqrt{n}\sum_{\mathcal I_n(y)} 2\big[\Delta_i^n M(y)\left( \Delta_i^n (\sigma\!\cdot\! W)+ \Delta_i^n b (y)\right) + \Delta_i^n b(y) \Delta_i^n (\sigma\!\cdot\! W)\big]\notag \\
&=: \sum_{i=0}^3 I_i. \label{e:fourterms}
\end{align}

For the term $I_0$,  the argument for  \eqref{eq:tau12_sto} shows 
$$I_0= \sqrt{n}\left(\sum_{\mathcal I_n(y)}\left(\Delta_i^n (\sigma\!\cdot\! W)\right)^2 -C_T\right)=O_P(1) - \sqrt{n} \sum_{i=1}^n  \big( \Delta_i^n (\sigma\!\cdot\! W)\big)^2{\bf 1}_{\{\Delta_i^n N(y) \neq 0\}}.$$
Let $\beta_i^n(y) = \left(\Delta_i^n (\sigma\!\cdot\! W)\right)^2{\bf 1}_{\{\Delta_i^n N(y) \neq 0\}}-\bE_{i-1}\left(\Delta_i^n (\sigma\!\cdot\! W)\right)^2{\bf 1}_{\{\Delta_i^n N(y) \neq 0\}}$ (with $\bE_{i-1}(\cdot)=\bE(\cdot | \mathcal F_{(i-1)h})$), and observe, for every $p\geq 1$, using \eqref{e:estimates0},
\begin{align*}
\bE_{i-1}\left(\Delta_i^n (\sigma\!\cdot\! W)\right)^{2p}{\bf 1}_{\{\Delta_i^n N(y) \neq 0\}} &= \bE_{i-1} \bE_{i-1}\Big[\left(\Delta_i^n (\sigma\!\cdot\! W)\right)^{2p}{\bf 1}_{\{\Delta_i^n N(y) \neq 0\}} \Big| \Delta_i^n N(y) \Big]\\
&\leq K \bE_{i-1} \bigg({\bf 1}_{\{\Delta_i^n N(y) \neq 0\}} \bE_{i-1}\Big[\Big(\int_{t_{i-1}}^{t_i} \sigma_s^2 ds \Big)^{p}  \Big| \Delta_i^n N(y) \Big]\bigg)\\
& \leq K h^{p+1} y^{-\alpha}.
\end{align*}
 Using that $hy^{-\alpha}\ll 1$, we obtain $\bE_{i-1}|\beta_i^n(y)|^p \leq K h^{p+1}y^{-\alpha}$, giving 
 \begin{equation}\label{e:betabound}
 \bE \bigg(\sqrt n \sum_{i=1}^n \beta_{i}^n(y)\bigg)^2 = n\sum_{i=1}^n \bE \left(\beta_{i}^n(y)\right)^2 = O(hy^{-\alpha}).
 \end{equation}
  Since $\sigma^2$ is bounded away from zero almost surely, we have
\begin{align*}
\bE_{i-1}\left(\Delta_i^n (\sigma\!\cdot\! W)\right)^2{\bf 1}_{\{\Delta_i^n N(y) \neq 0\}} &=\bE_{i-1} \bigg({\bf 1}_{\{\Delta_i^n N(y) \neq 0\}} \bE_{i-1}\Big[ \int_{t_{i-1}}^{t_i} \sigma_s^2 ds\Big  | \Delta_i^n N(y) \Big]\bigg)\\
& \geq Kh \bP(\Delta_i^n N(y) \neq 0)\sim Kh^2 y^{-\alpha}.
\end{align*}
Analogously, we have $\bE_{i-1}\left(\Delta_i^n (\sigma\!\cdot\! W)\right)^2{\bf 1}_{\{\Delta_i^n N(y) \neq 0\}}\leq K'h^2 y^{-\alpha}$, i.e., for suitable $K_0,K_1>0$,$$
 K_0 h^{1/2} y^{-\alpha}\leq \sqrt n \sum_{i=1}^n\bE_{i-1}\left(\Delta_i^n (\sigma\!\cdot\! W)\right)^2{\bf 1}_{\{\Delta_i^n N(y) \neq 0\}}\leq  K_1 h^{1/2} y^{-\alpha}.
$$
Thus, for all large $n$, from \eqref{e:betabound}, 
\begin{align}
-K_1h^{1/2} y^{-\alpha}\Big(1 + o_P\left(y^{\alpha/2}\right)\Big)+ O_P(1) \leq I_0 \leq O_P(1) - K_0h^{1/2} y^{-\alpha}\Big(1 + o_P\left(y^{\alpha/2}\right)\Big),\label{e:I0_bound}%
\end{align}
which implies $I_0 = O_P( 1 \vee h^{1/2} y^{-\alpha})$.
We now consider $I_1.$ Let
$$\xi_{j,y} = y^{-(j-\alpha)}\int_{|x|\leq y} x^j \nu(dx) =y^{-(j-\alpha)}h^{-1}\bE \left(\Delta_i^nM(y)\right)^j,\quad j=2,4,$$
(e.g., \cite[Example 25.12]{sato:1999}) which satisfy $\liminf_{n\to \infty} \xi_{j,y} >0$ under Assumption \ref{assump:Coef0}(ii).  
Observe $\Var(\sum _{i=1}^n\left(\Delta_i^nM(y)\right)^2 ) = y^{4-\alpha}\xi_{4,y} - h y^{4-2\alpha}\xi^2_{2,y} = O(y^{4-\alpha})$, and 
\begin{align*}
\sum_{i=1}^n\bE|\big(\Delta_i^nM(y)\big)^2 -\bE\big(\Delta_i^nM(y)\big)^2|^4 &\leq nK\Big(\bE\big(\Delta_i^nM(y)\big)^8  + \big[\bE\big(\Delta_i^nM(y)\big)^2\big]^4\Big)\\
&\leq K  y^{8-\alpha}  + K' h^3y^{8-4\alpha}\\
& = o((y^{4-\alpha})^{2}).
\end{align*}
Therefore, the Lyapunov CLT implies that
\begin{align*}
\frac{\sum _{i=1}^n\left(\Delta_i^nM(y)\right)^2 - n \bE\left(\Delta_i^nM(y)\right)^2 }{\sqrt {n\Var\big( (\Delta_i^nM(y)^2\big)} } &= \frac{\sum _{i=1}^n\left(\Delta_i^nM(y)\right)^2 - y^{2-\alpha}\xi_{2,y} }{\sqrt{ y^{4-\alpha}\xi_{4,y} - h y^{4-2\alpha}\xi^2_{2,y} }}\notag\\
& = \frac{\sum _{i=1}^n\left(\Delta_i^nM(y)\right)^2 - y^{2-\alpha}\xi_{2,y} }{y^{2-\frac{\alpha}{2}} \xi^{1/2}_{4,y} \sqrt{1 + o(1)} }
\toDist \mathcal{N}(0,1).%
\end{align*}
This gives
\begin{align}
    I_1& = 
    \sqrt{n}\sum_{i=1}^n\left(\Delta_i^nM(y)\right)^2{\bf 1}_{\{\Delta_i^n N(y)=0 \}} \nonumber\\
    &=  \sqrt{n} \sum_{i=1}^n \left(\Delta_i^n M(y)\right)^2 - \sqrt{n} \sum_{i=1}^n \left(\Delta_i^n M(y)\right)^2 \, {\bf 1}_{\{\Delta_i^n N(y)\neq 0\}} \label{eq:by_M_clt_0}\\
    &= h^{-1/2} \left(y^{2-\alpha}\xi_{2,y}  + y^{2-\frac{\alpha}{2}} \xi^{1/2}_{4,y} \sqrt{1 + o(1)} \cdot Y_n \right)  + o_P\left(h^{-1/2}y^{2-\alpha}\right)\label{eq:by_M_clt},
\end{align}
where $Y_n\toDist \mathcal{N}(0,1)$.   The order of the second term in \eqref{eq:by_M_clt_0} is a consequence of
\begin{align*}
 \bE\left(\frac{\sqrt{n} \sum_{i=1}^n \left(\Delta_i^n M(y)\right)^2 {\bf 1}_{\{\Delta_i^n N(y)\neq 0\}}}{h^{-1/2}y^{2-\alpha}}\right) \leq y^{\alpha-2}\, \bE\left( N_1(y)\right)\bE\left( \left(\Delta_i^n M(y)\right)^2 \right) =O\left( h y^{-\alpha} \right) = o(1),
\end{align*}
since $h^{1/\alpha} \ll y$.  Moreover,  \eqref{eq:by_M_clt} implies $I_1/(h^{-1/2} y^{2-\alpha}\xi_{2,y})\toprob 1$. %
Next, for $I_2$, observe
\begin{gather*}
     \sqrt{n}\sum_{\mathcal I_n(y)} \left( \Delta_i^n b(y)\right)^2 \leq \sqrt{n}\sum_{i=1}^n \left( \Delta_i^n b(y)\right)^2 = Kn^{3/2} h^2 y^{2-2\alpha} = o(h^{1/2}y^{-\alpha}),%
\end{gather*}
showing $I_2=o_P(I_0)$.  We now show  $I_3  =o_P(|I_0|+I_1)$. For the first term in $I_3$, \eqref{e:T11_argument_penultimate} shows, for any $p>1$,%

\begin{align*}
   \sqrt{n}\sum_{\Delta_i^n N(y) = 0} \left|\left(\Delta_i^n (\sigma\!\cdot\! W)\right)\left( \Delta_i^n M(y)\right)\right| =O_P((hy^{-\alpha})^{\frac{1}{2p}}h^{-\frac12}y),%
\end{align*}
which is smaller than $|I_0|$ when $y< h^{1/2}$ since $(hy^{-\alpha})^{\frac{1}{2p}}h^{-\frac12}y<(hy^{-\alpha})^{\frac{1}{2p}}\ll h^{1/2}y^{-\alpha}$ for $p$ close enough to 1, and smaller than $I_1$ when $y\geq h^{1/2}$, since $(hy^{-\alpha})^{\frac{1}{2p}}h^{-\frac12}y\leq (h^{1-\alpha/2})^{\frac{1}{2p}}h^{-\frac12}y\ll h^{-1/2}y^{2-\alpha} \iff (h^{1-\alpha/2})^{\frac1{2p}}\ll y^{1-\alpha}$ which holds for all $p>1$ since $\alpha>1$. 
Also, noting in this case $b_t(y)$ is deterministic, we have
\begin{align*}
    \bE\left(\sqrt{n}\sum_{\Delta_i^n N(y) = 0} \left|\left(\Delta_i^n b(y)\right) \left( \Delta_i^n M(y)\right)\right|\right)
    &\leq \sqrt{n}\sqrt{\sum_{i=1}^n \left( \Delta_i^n b(y)\right)^2 } \sqrt{\sum_{i=1}^n \bE\left( \Delta_i^n M(y)\right)^2 }\\
    &= O\left(y^{2-3\alpha/2}\right)= o_P(I_1),
\end{align*}
since $y^{2-3\alpha/2}\ll h^{-1/2}y^{2-\alpha}$ under $y\gg h^{\frac1\alpha}.$  
Similarly,
\begin{align*}
    &\bE\left( \sqrt{n}\sum_{\Delta_i^n N(y) = 0} \left|\left(\Delta_i^n (\sigma\!\cdot\! W)\right)\left( \Delta_i^n b(y)\right)\right|\right) \\
    &\leq \sqrt{n}\max_{1\leq i\leq n}\left| \Delta_i^n b(y)\right| \sum_{i=1}^n \bE\left|\Delta_i^n (\sigma\!\cdot\! W)\right|  \leq K y^{1-\alpha },%
\end{align*}
which is smaller than $|I_0|$ when $y\ll h^{1/2}$, since $y^{1-\alpha}\ll h^{1/2}y^{-\alpha}\iff y \ll h^{1/2}$ and smaller than $I_1$ when $y\gg h^{1/2}$ since $y^{1-\alpha}\ll h^{-1/2}y^{2-\alpha}\iff h^{1/2}\ll y$.  This shows in any of the cases \eqref{e:cases_for_y}, $I_3=o_P(|I_0| + I_1)$, and thus the limit of \eqref{e:fourterms} is determined by $I_0+I_1$. Turning to the first case in \eqref{e:cases_for_y}, under $h^{\frac{1}{2(2-\alpha)}} \vee h^{\frac{1}{\alpha}} \ll y\ll h^{\frac12}$, we have $I_1= o_P(|I_0|)$ (since $y\ll h^{1/2}$ gives $h^{-1/2} y^{2-\alpha}\ll h^{1/2}y^{-\alpha}$), and thus from \eqref{e:I0_bound} we have $I_0\toprob -\infty$, giving  $\sqrt{n}\left(\mathscr C_n(y)- C_T\right) \toprob -\infty$. For the second case in \eqref{e:cases_for_y}, under $y\gg h^{1/2}$ we have instead $|I_0|= o_P(I_1)$, and from \eqref{eq:by_M_clt} we have $I_1\toprob \infty$, giving $\sqrt{n}\left(\mathscr C_n(y)- C_T\right) \toprob \infty$.
\end{proof}

Below, in Propositions \ref{p:prop3_IV}  and \ref{p:prop3}, we establish the convergence in probability of the sequences $\underline C_{n}(X)$ and $\underline c_{n}(i;X)$ that are used in the proof of Proposition \ref{prop:allequal}. Recall  that for notational convenience we set $F( \Delta_i^n Y,\ldots,\Delta_{i+d-1}^n Y)=0$ when $i+d-1>n$ or $i\leq 0$.

\begin{proposition}\label{p:prop3_IV} 
Define the set
\begin{equation}\label{e:def_I^k(y)}
\mathcal I_{n}^{(d)}(y)= \{i=1,\dots,n-d+1:  \Delta_{i+\ell}^n N(y) = 0,\, \Delta_{i+\ell}^n N' = 0,\quad \ell=0,\ldots, d-1\}.
\end{equation}
Let $\widehat C_{n,0}=\widehat C_{n,0}(X)$  belong to the class $\mathcal C$ with $\widehat C_{n,0}(\cdot)$ defined as in \eqref{e:def_C0} and suppose \eqref{e:C0_consistent_cont_case} holds. With $\delta_0$ as in \eqref{e:Lipbound}, let $y=y_n\to 0$ in such a way that, for some $0<\delta<(\frac{1}{\alpha} \wedge 1)-\frac12$,  %
\begin{equation}\label{e:yn_bound_for_minorizing}
h^{(\frac{1}{\alpha} \wedge 1)-\delta}\ll y\ll h^{\frac12}(\log n)^{\frac12 - \frac1{\delta_0}}.
\end{equation} 
For $F$ as in \eqref{e:def_C0}, define,  for a generic process $Y$,
\begin{equation*}%
\underline C_{n}(Y,y) = \sum_{i \in \mathcal I_{n}^{(d)}(y) } F(\Delta_i^n Y,\ldots, \Delta_{i+d-1}^n Y).
\end{equation*}
Then, $\underline C_{n}(X,y) \leq \widehat C_{n,0}$, and
\begin{equation}\label{e:C_lower_consistent}
\underline C_{n}(X,y)\stackrel P \to \int_0^1 \sigma_s^2ds.
\end{equation}
\end{proposition}
\begin{proof}
For simplicity write $\underline C_n(\cdot)=\underline C_n(\cdot,y)$. Observe $\underline C_{n}(X) \leq \widehat C_{n,0}$ is immediate, so we need only prove \eqref{e:C_lower_consistent}. 
Clearly \begin{align*}
& \sum_{i=1}^n F\big( \Delta_i^n (\sigma\!\cdot\! W),\ldots, \Delta_{i+d-1}^n (\sigma\!\cdot\! W)\big) -\underline C_n\big(\sigma\!\cdot\! W\big) \\
 & =  \sum_{i=1}^n F\big( \Delta_i^n (\sigma\!\cdot\! W),\ldots, \Delta_{i+d-1}^n (\sigma\!\cdot\! W)\big)\Big(\sum_{\ell=0}^{d-1}({\bf 1}_{\{\Delta_{i+\ell}^n N(y) \neq 0\}} + {\bf 1}_{\{\Delta_{i+\ell}^n N' \neq 0\}}\big)\Big)\\
 &=:T_1+T_2.
\end{align*}
Note that, using \eqref{e:estimates0}, for each $i$,
\begin{equation*}%
\bP\left( \Delta_i^n N(y) \neq 0\right)  \leq K h y^{-\alpha} =o\Big(\frac{1}{\log n}\Big),
\end{equation*}
since $h(\log n) y^{-\alpha} = o(1)$ by \eqref{e:yn_bound_for_minorizing}. In particular,
$$
\bE \sum_{i=1}^{n}{\bf 1}_{\{\Delta_{i}^n N(y) \neq 0\}}\leq K y^{-\alpha}.
$$
Using \eqref{e:F_maxbound}, this gives%
\begin{align*}
T_1
& \leq K\max ( \Delta_i^n (\sigma\!\cdot\! W))^2\sum_{i=1}^n\sum_{\ell=0}^{d-1}{\bf 1}_{\{\Delta_{i+\ell}^n N(y) \neq 0\}}\\
&\leq K\max ( \Delta_i^n (\sigma\!\cdot\! W))^2\sum_{i=1}^n {\bf 1}_{\{\Delta_{i}^n N(y) \neq 0\}}\\
&= K\frac{\max ( \Delta_i^n (\sigma\!\cdot\! W))^2}{h\log n} O_P\big(hy^{-\alpha} \log n \big)\\
&=O_P(1)\cdot o_P(1).
\end{align*}
Similarly, again using \eqref{e:F_maxbound},
$$
T_2=\sum_{i=1}^n F\big( \Delta_i^n (\sigma\!\cdot\! W),\ldots, \Delta_{i+d-1}^n (\sigma\!\cdot\! W)\big)\sum_{\ell=0}^{d-1}{\bf 1}_{\{\Delta_{i+\ell}^n N' \neq 0\}} \leq K N'_1 \max ( \Delta_i^n (\sigma\!\cdot\! W))^2 =o_P(1).
$$
Thus,
\begin{equation}\label{e:tau2_F-mainF_to_0}
\sum_{i=1}^n F\big( \Delta_i^n (\sigma\!\cdot\! W),\ldots, \Delta_{i+d-1}^n (\sigma\!\cdot\! W)\big) -\underline C_n\big(\sigma\!\cdot\! W\big) =o_P(1).
\end{equation}
Using assumption \eqref{e:C0_consistent_cont_case} together with \eqref{e:tau2_F-mainF_to_0}, we obtain%
\begin{equation} \label{e:tau2_consistent} 
\underline C_n\big(\sigma\!\cdot\! W\big)\stackrel P \to \int_0^1 \sigma_s^2ds. %
\end{equation}
On the other hand, using the notation \eqref{e:delta_i^nX_over_In(y)}, we have $\Delta_i^nX=\Delta_i^n (\sigma \cdot W)+ \Delta_i^n \chi(y)$. Furthermore,  under condition \eqref{e:yn_bound_for_minorizing}, we have, from  Lemma \eqref{l:maxbound_chi}  $\max_{i\in \mathcal I^{(d)}_n(y)} \big|\Delta_i^n \chi(y)\big| = o_P(y)=o_P(1)$. Also, Lemma \ref{l:mod_of_cont}  gives $\max_{i\in  \mathcal I^{(d)}_n} (y) \big|\Delta_i^n (\sigma \!\cdot\! W) \big|\leq \max_{1\leq i \leq n} \big|\Delta_i^n (\sigma \!\cdot\! W) \big| = O_P(\sqrt{h\log n})=o_P(1)$, so that the event
$$
A_n=\left\{\max_{i\in \mathcal I_n (y)} \big|\Delta_i^n (\sigma \!\cdot\! W) \big| \vee \max_{i\in \mathcal I^{(d)}_n (y)} \big|\Delta_i^n \chi(y)\big| \leq 1\right\}
$$
tends to 1 in probability.  Thus, from condition \eqref{e:Lipbound}, we obtain%
\begin{align}
\nonumber
&|\underline C_n(X) -\underline C_n\big(\sigma\!\cdot\! W\big)|\left({\bf 1}_{A_n} + {\bf 1}_{A_n^c}\right)\\
\nonumber
 & = \bigg| \sum_{i\in \mathcal I^{(d)}_n(y)} \Big[F\Big(\Delta_i^n X\ldots,\Delta_{i+d-1}^n X\Big) - F\Big(\Delta_i^n (\sigma \!\cdot\! W),\ldots,\Delta_{i+d-1}^n (\sigma \!\cdot\! W) \Big)\Big]\bigg|{\bf 1}_{A_n} +o_P(1)\\
 \nonumber
& \leq n K\bigg( \max_{i\in \mathcal I_n (y)} \big|\Delta_i^n \chi(y)\big|\bigg)^{\delta_0}\left(\bigg( \max_{i\in \mathcal I_n (y)} \big|\Delta_i^n \chi(y) \big|\bigg)^{2-\delta_0} + \bigg( \max_{i\in \mathcal I_n (y)} \big|\Delta_i^n (\sigma \!\cdot\! W) \big|\bigg)^{2-\delta_0}\right)+o_P(1)\\%\big( |\Delta_i^n (\sigma \!\cdot\! W) | +  | \Delta_i^n \chi(y) |  \big) \Big\} \\%+ N'_1\big(O(h\log n) + o(k_n^2)\big) \\
& = n\,O_P (y^{\delta_0})\cdot \bigg(O_P(y^{2-\delta_0})+O_P\Big((h \log n)^{\frac{2-{\delta_0}}{2}}\Big)\bigg)  +o_P(1)\nonumber\\
& = O_P(ny^2)+ O_P\Big( y^{\delta_0}n^{\frac{\delta_0}{2}}(\log n)^{1-\frac{\delta_0}{2}}\Big) +o_P(1)\nonumber\\
&=o_P(1),\label{e:max_op(1)_lipbound}
 \end{align}
where the last line follows from condition \eqref{e:yn_bound_for_minorizing}.
Thus,
$$
\underline C_n(X) - \underline C_n\big(\sigma\!\cdot\! W\big) = o_P(1),
$$
 which by \eqref{e:tau2_consistent} establishes \eqref{e:C_lower_consistent}. 
\end{proof}

\begin{proposition}\label{p:prop3} Let $\widehat c_{n,0}(i)=\widehat c_{n}(i;X)$, $i=1,\ldots,n$ belong to class $\mathcal C^{\text{spot}}$ with $\widehat c_{n}(i;\cdot)$ of the form \eqref{e:def_little_C0}.  For each $y>0$, with $\mathcal I_{n}^{(k)}(y) $  as in \eqref{e:def_I^k(y)},  %
define, for a generic process $Y$,
\begin{equation*}%
\underline c_{n}(i;Y,y) = \frac{n}{k_n} \sum_{\ell= i-k_n/2+1}^{i+k_n/2} F(\Delta_\ell^n Y,\ldots,\Delta_{\ell+d-1}^n Y){\bf 1}_{\{\Delta_{\ell+m}^n N(y) = 0,\, \Delta_{\ell+m}^n N' = 0,~m=0,\ldots,d-1\}}.
\end{equation*}
Then, $ \underline c_{n}(i;X,y)\leq\widehat c_{n,0}(i) $, and if  $y=y_n\to 0$ as in \eqref{e:yn_bound_for_minorizing}, and $(\log n)^3\ll k_n \ll n$, 
\begin{equation}\label{e:littlec_lower_consistent}
 \max_{1\leq i \leq n}\frac{\sup_{t\in [\frac{i-k_n/2}{n}, \frac{i+k_n/2}n)} \sigma_t^2 }{\underline c_{n,0}(i;X,y)} \stackrel P \longrightarrow 1.
\end{equation}
In particular, if for some $0<a<1$, $n^a \ll k_n\ll n$, then the variables $\mathscr{c}_n(i;y)$ defined in \eqref{AnlgDfnc} satisfy
\begin{equation}\label{e:littlec_oracle_consistent}
 \max_{1\leq i \leq n}\frac{\sup_{t\in [\frac{i-k_n/2}{n}, \frac{i+k_n/2}n)} \sigma_t^2 }{\mathscr c_{n}(i;y)} \stackrel P \longrightarrow 1.
\end{equation}
\end{proposition}
\begin{proof} For simplicity write $ \underline c_{n}(i;\cdot)= \underline c_{n}(i;\cdot,y).$ Since $F$ is nonnegative,  $ \underline c_{n}(i;X)\leq\widehat c_{n,0}(i) $ is immediate, so it remains to show \eqref{e:littlec_lower_consistent}.
Consider
\begin{equation*}%
\underline c_{n,0}\big( \ell;\sigma\!\cdot\! W\big) = \frac{n}{k_n} \sum_{i=\ell-k_n/2+1}^{\ell +k_n/2}F\big(\Delta_i^n (\sigma\!\cdot\! W),\ldots,\Delta_{i+d-1}^n (\sigma\!\cdot\! W)\big){\bf 1}_{\{\Delta_{i+m}^n N(y) = 0,\, \Delta_{i+m}^n N' = 0,~0\leq m\leq d-1\}}
\end{equation*}
We first show $\max_{1\leq \ell \leq n}\big|\underline c_{n,0}\big(\ell; \sigma\!\cdot\! W\big)-  \sup_{t\in [\frac{\ell-k_n/2}{n}, \frac{\ell+k_n/2}n)} \sigma_t^2\big|\toprob 0$. Note
\begin{align}\label{e:zeta_n(y)_convergence}
 0&\leq \max_{1\leq \ell \leq n}\Big(\widehat c_{n,0}(\ell;\sigma\!\cdot\! W)- \underline c_{n,0}\big( \ell;\sigma\!\cdot\! W\big)\Big)\\\nonumber
 &\leq \max_{1 \leq \ell \leq n }\frac{n}{k_n}\sum_{i=\ell-k_n/2+1}^{\ell+k_n/2} F\big( \Delta_i^n (\sigma\!\cdot\! W),\ldots,\Delta_{i+d-1}^n (\sigma\!\cdot\! W)\big)\Big(\sum_{m=0}^{d-1}({\bf 1}_{\{\Delta_{i+m}^n N(y) \neq 0\}} + {\bf 1}_{\{\Delta_{i+m}^n N' \neq 0\}}\Big). %
\end{align}
Using \eqref{e:F_maxbound}, we have
\begin{align}
& \max_{1 \leq \ell \leq n }\frac{n}{k_n}\sum_{i=\ell-k_n/2+1}^{\ell+k_n/2} F\big( \Delta_i^n (\sigma\!\cdot\! W),\ldots,\Delta_{i+d-1}^n (\sigma\!\cdot\! W)\big) \sum_{m=0}^{d-1}{\bf 1}_{\{\Delta_{i+m}^n N(y) \neq 0\}}\notag\\
 & \leq K\max ( \Delta_i^n (\sigma\!\cdot\! W))^2 \max_{1 \leq \ell \leq n }\frac{n}{k_n}\sum_{i=\ell-k_n/2+1}^{\ell+k_n/2}\sum_{m=0}^{d-1}{\bf 1}_{\{\Delta_{i+m}^n N(y) \neq 0\}}\notag\\
&= K\frac{\max ( \Delta_i^n (\sigma\!\cdot\! W))^2}{h\log n} \max_{1 \leq \ell \leq n }\frac{\log n}{k_n}\sum_{i=\ell-k_n/2+1}^{\ell+k_n/2}{\bf 1}_{\{\Delta_{i}^n N(y) \neq 0\}}\notag\\ %
&=O_P(1)\cdot  \max_{1 \leq \ell \leq n }\frac{\log n}{k_n}\sum_{i=\ell-k_n/2+1}^{\ell+k_n/2}{\bf 1}_{\{\Delta_{i+m}^n N(y) \neq 0\}}.\label{e:unif_max_bound1}
\end{align}
Recalling from \eqref{e:estimates0} that $p_n(y): = \bP(\Delta_i^n N(y) \neq 0)\sim K hy^{-\alpha}=o(1/(\log n))$ due to \eqref{e:yn_bound_for_minorizing}, we have
\begin{align*}
&\max_{1\leq \ell \leq n}\frac{ \log n}{k_n} \sum_{i= (\ell-k_n/2+ 1) }^{(\ell+k_n/2)}{\bf 1}_{\{\Delta_i^n N(y) \neq 0 \}} \\
 &\quad = o_P(1)+ \max_{0 \leq \ell \leq n }\frac{ \log n}{k_n} \Big|\sum_{i= (\ell-k_n/2+ 1)\vee 1}^{(\ell+k_n/2)\wedge n}\big({\bf 1}_{\{\Delta_i^n N(y) \neq 0 \}}-p_n(y)\big)\Big|.
\end{align*}
Hoeffding's inequality together with a union bound give
\begin{align*}
\bP\bigg( \max_{1 \leq \ell \leq  n }\frac{\log n}{k_n} & \bigg|\sum_{i= (\ell-k_n/2+ 1)\vee 1}^{(\ell+k_n/2)\wedge n}\big({\bf 1}_{\{\Delta_i^n N(y) \neq 0\}}-p_n(y)\big)\bigg|  \geq \delta \bigg)\\
&\leq n \bP\bigg( \bigg| \sum_{i= (\ell-k_n/2+ 1)\vee 1}^{(\ell+k_n/2)\wedge n}\big({\bf 1}_{\{\Delta_i^n N(y) \neq 0\}}-p_n(y)\big)\bigg|  \geq \frac{\delta k_n }{\log n}\bigg)\\
&\leq 2 n\exp\big(-2\delta^2 k_n/\log^2(n)\big)\\
&\leq 2 n\exp\big(-3\log n \big) = o(1),
\end{align*}
since for large enough $n$, $2\delta^2 k_n/\log^2(n) \gg 3\log n$  due to the assumption $k_n\gg (\log n)^3$.  Hence from \eqref{e:unif_max_bound1}, we obtain
\begin{equation}\label{e:little_c_orc_cont_bound1}
 \max_{1 \leq \ell \leq n }\frac{n}{k_n}\sum_{i=\ell-k_n/2+1}^{\ell+k_n/2} F\big( \Delta_i^n (\sigma\!\cdot\! W),\ldots,\Delta_{i+d-1}^n (\sigma\!\cdot\! W)\big) \sum_{m=0}^{d-1}{\bf 1}_{\{\Delta_{i+m}^n N(y) \neq 0\}}=o_P(1).
\end{equation}
Turning to the second term  in \eqref{e:zeta_n(y)_convergence}, since  $\sum_{i=\ell-k_n/2+1}^{\ell+k_n/2} \sum_{m=0}^{d-1}{\bf 1}_{\{\Delta_{i+m}^n N' \neq 0\}}\leq  d N'_1$,  we have,%
\begin{align}
&\max_{1 \leq \ell \leq n }\frac{n}{k_n}\sum_{i=\ell-k_n/2+1}^{\ell+k_n/2} F\big( \Delta_i^n (\sigma\!\cdot\! W),\ldots,\Delta_{i+d-1}^n (\sigma\!\cdot\! W)\big)\sum_{ m=0}^{d-1}{\bf 1}_{\{\Delta_{i+ m}^n N'\neq 0\}}\notag\\
&  \leq K \max_{1\leq i \leq n} \big( \Delta_i^n (\sigma\!\cdot\! W)\big)^2  \frac{n}{k_n}\cdot  d N_1'\notag\\
& = O_P(h \log n) \cdot \frac{n}{k_n} \cdot O_P(1) = o_P(1).\label{e:little_c_orc_cont_bound2}
\end{align}
Putting together \eqref{e:little_c_orc_cont_bound1} and \eqref{e:little_c_orc_cont_bound2}, by \eqref{e:zeta_n(y)_convergence}, we obtain
\begin{align*}
\max_{1\leq \ell \leq n}\Big(\widehat c_{n,0}(\ell;\sigma\!\cdot\! W)- \underline c_{n,0}\big( \ell;\sigma\!\cdot\! W\big)\Big) = o_P(1).
\end{align*}
Thus, from \eqref{e:little_c0_consistent_contcase}, we obtain
\begin{align}\label{e:c_under_contcase_consistent}
&\max_{1\leq \ell \leq n}\Big|\underline c_{n,0}\big(\ell; \sigma \!\cdot\! W\big)-  \sup_{t\in [\frac{\ell-k_n/2}{n}, \frac{\ell+k_n/2}n)} \sigma_t^2\Big|\\
&\quad =o_P(1) + \max_{1\leq i \leq n}\bigg|  \widehat c_{n,0}(i;\sigma\!\cdot\! W) - \frac{ n }{k_n} \int_{\frac{i-k_n/2}{n}}^{\frac{i+k_n/2}n} \sigma_t^2dt  \bigg|\notag\\
& \qquad \qquad\qquad +\max_{1\leq i \leq n}\bigg| \frac{ n }{k_n} \int_{\frac{i-k_n/2}{n}}^{\frac{i+k_n/2}n} \sigma_t^2dt - \sup_{t\in [\frac{i-k_n/2}{n}, \frac{i+k_n/2}n)} \sigma_t^2  \bigg| \notag\\
&\quad= o_P(1),\notag%
\end{align}
where the last term term on the second line tends to 0 due to \eqref{e:sup_integral_cadlag_}.
On the other hand, 
we have
\begin{align*}
&\max_{1\leq \ell \leq n}\big|\underline c_{n,0}\big(\ell; X\big)-\underline c_{n,0}\big(\ell; \sigma\! \cdot \!W\big) \big|\\
& = \bigg| \max_{1\leq \ell \leq n}\frac{n}{k_n} \sum_{i= \ell-k_n/2+1}^{\ell+k_n/2} \Big[F\Big(\Delta_i^n X\ldots,\Delta_{i+d-1}^n X\Big) - F\Big(\Delta_i^n (\sigma \!\cdot\! W),\ldots,\Delta_{i+d-1}^n (\sigma \!\cdot\! W) \Big)\Big]{\bf 1}_{\{i\in \mathcal I_n^{(d)}(y)\}}\bigg|\\
&\leq n \max_{i\in \mathcal I^{(d)}_n (y)}  \Big|F\Big(\Delta_i^n X\ldots,\Delta_{i+d-1}^n X\Big) - F\Big(\Delta_i^n (\sigma \!\cdot\! W),\ldots,\Delta_{i+d-1}^n (\sigma \!\cdot\! W) \Big)\Big|\\
&=o_P(1),
\end{align*}
where the last line follows by repeating the argument leading to \eqref{e:max_op(1)_lipbound}.
Thus, using \eqref{e:c_under_contcase_consistent}, %
\begin{align*}
\max_{1\leq \ell \leq n}&\Big|\underline c_{n,0}\big(\ell; X\big) -\sup_{t\in [\frac{\ell-k_n/2}{n}, \frac{\ell+k_n/2}n)} \sigma_t^2\Big| \\&\leq\max_{1\leq \ell \leq n}\big|\underline c_{n,0}\big( \ell;X\big)-\underline c_{n,0}\big( \ell;\sigma\! \cdot \!W\big) \big|+ \max_{1\leq \ell \leq n}\Big|\underline c_{n,0}\big(\ell; \sigma\! \cdot \!W\big)- \sup_{t\in [\frac{\ell-k_n/2}{n}, \frac{\ell+k_n/2}n)} \sigma_t^2\Big| =o_P(1),
\end{align*}
 which establishes \eqref{e:littlec_lower_consistent}. For the statement \eqref{e:littlec_oracle_consistent},  since $n^a\ll k_n \ll n$,  Lemma \ref{l:untruncated_strong_consistency} implies the initialization $\widehat c_{n,0}$ as in \eqref{e:def_little_C0} with $F(x)=x^2$ belongs to class $\mathcal C^{\text{spot}}$. Hence, \eqref{e:littlec_oracle_consistent} follows from \eqref{e:littlec_lower_consistent}.
\end{proof}
\begin{lemma}\label{l:untruncated_strong_consistency} Suppose for some $0<a<1$, $n^a\ll k_n \ll n$. Then,%
\begin{equation}\label{e:maxVol_asconverge}
\max_{\frac{k_n}{2}< \ell \leq n-\frac{k_n}{2}}\Bigg| \frac{n}{k_n} \sum_{i=\ell-\frac{k_{n}}{2}}^{\ell +\frac{k_{n}}{2}} \big(\Delta_i^n(\sigma\!\cdot\! W)\big)^2 - \sup_{t\in [\frac{\ell-\frac{k_{n}}{2}}{n}, \frac{\ell+\frac{k_{n}}{2}}n)} \sigma_t^2\Bigg|  =o_P(1).
\end{equation}
\end{lemma}

\begin{proof} 
Throughout we set $\underline{t}_{\ell}=\frac{\ell-\frac{k_{n}}{2}}{n}$ and $\bar{t}_{\ell}=\frac{\ell+\frac{k_{n}}{2}}{n}$. Since $\sigma$ is c\`adl\`ag, we claim
 \begin{equation}\label{e:sup_integral_cadlag_}
 \max_{\frac{k_n}{n}< \ell \leq n-\frac{k_n}{2}}\bigg(\sup_{t\in [\underline{t}_{\ell},\bar{t}_{\ell})} \sigma_t^2-\frac{n}{k_n}\int_{\underline{t}_{\ell}}^{\bar{t}_{\ell}} \sigma^2_s ds\bigg) \to 0\quad \text{a.s.}
 \end{equation}
 Indeed, to show \eqref{e:sup_integral_cadlag_},  for each $\omega$ and any $\delta>0$, we can find a $\nu=\nu(\omega,\delta)$ and $0=s_0<s_1(\omega)<\ldots<s_\nu(\omega)=1$ such that
\begin{equation}\label{e:cadlag_uniform_bound}
\sup_{s,t \in[s_{i-1},s_i)}|\sigma^2_s-\sigma_t^2| <\delta,\qquad i=1,\ldots,\nu
\end{equation}
(see, e.g., p.122 of \cite{billingsley:1999}). 
So, take $ n_0(\omega)$ large enough so that $n\geq n_0(\omega)$ gives
$$
\frac{k_n}{n}<\min \frac{s_i-s_{i-1}}2.%
$$
Then for each $\ell$, there is at most one $s_j$ with $\underline{t}_{\ell}\leq s_j<\bar{t}_{\ell}$.  For any such $\ell$, we have
\begin{align*}
&\sup_{t\in [\underline{t}_{\ell}, \bar{t}_{\ell})} \sigma_t^2-\frac{n}{k_n}\int_{\underline{t}_{\ell}}^{\bar{t}_{\ell}} \sigma^2_s ds \\
&=\max\bigg\{\sup_{t\in [\underline{t}_{\ell}, s_j)}\sigma_t^2-\frac{n}{k_n}\int_{\underline{t}_{\ell}}^{\bar{t}_{\ell}} \sigma^2_s ds,~ \sup_{t\in [s_j,\bar{t}_{\ell})}\sigma^2_t-\frac{n}{k_n}\int_{\underline{t}_{\ell}}^{\bar{t}_{\ell}} \sigma^2_s ds\bigg\}\\
&\leq \max\bigg\{\sup_{t\in [\underline{t}_{\ell}, s_j)}\sigma_t^2-\frac{n}{k_n}\int_{\underline{t}_{\ell}}^{\bar{t}_{\ell}} \sigma^2_s {\bf 1}_{\{s\in [\underline{t}_{\ell}, s_j)\}}ds, ~\sup_{t\in [s_j,\bar{t}_{\ell})}\sigma^2_t-\frac{n}{k_n}\int_{\underline{t}_{\ell}}^{\bar{t}_{\ell}}  \sigma^2_s {\bf 1}_{\{s\in [s_j, \bar{t}_{\ell})\}} ds\bigg\}\\
&= \max\bigg\{\frac{n}{k_n}\int_{\underline{t}_{\ell}}^{\bar{t}_{\ell}} \big(\sup_{t\in [\underline{t}_{\ell}, s_j)}\sigma_t^2-\sigma^2_s{\bf 1}_{\{s\in [\underline{t}_{\ell}, s_j)\}}\big) ds,~\frac{n}{k_n}\int_{\underline{t}_{\ell}}^{\bar{t}_{\ell}}  \big(\sup_{t\in [s_j,\bar{t}_{\ell})}\sigma^2_t-\sigma^2_s {\bf 1}_{\{s\in [s_j, \bar{t}_{\ell})\}}\big)ds\bigg\}\\
& \leq \max\Big\{ \sup_{t,s\in [\underline{t}_{\ell}, s_j)}|\sigma_t^2-\sigma^2_s|, \sup_{t\in [s_j,\bar{t}_{\ell})}|\sigma_t^2-\sigma^2_s|\Big\} < \delta.
\end{align*}
For  each remaining $\ell$, we have $[\underline{t}_{\ell}, \bar{t}_{\ell})\subseteq [s_{j-1},s_j)$ for some $j$, giving, for each such $\ell$,
$$
0\leq\sup_{t\in [\underline{t}_{\ell}, \bar{t}_{\ell})} \sigma_t^2-\frac{n}{k_n}\int_{\underline{t}_{\ell}}^{\bar{t}_{\ell}} \sigma^2_s ds \leq \sup_{t,s\in [s_{j-1},s_j)}|\sigma_t^2-\sigma^2_s| <\delta.
$$
Thus, $\max_{\frac{k_n}{n}< \ell \leq n-\frac{k_n}{2}}\bigg(\sup_{t\in [\underline{t}_{\ell},\bar{t}_{\ell})} \sigma_t^2-\frac{n}{k_n}\int_{\underline{t}_{\ell}}^{\bar{t}_{\ell}} \sigma^2_s ds\bigg)<\delta,$ for all $n\geq n_0(\omega)$, which by arbitrariness of $\delta$, gives \eqref{e:sup_integral_cadlag_}.

So,  we need only to show
\begin{equation}\label{e:needonlytoshow}
\max_{\frac{k_n}{n}< \ell \leq n-\frac{k_n}{2}}\Bigg|\frac{n}{k_n}\sum_{i=\ell-\frac{k_n}{2}}^{\ell + \frac{k_n}{2}} \big(\Delta_i^n(\sigma\!\cdot\! W)\big)^2 -\frac{n}{k_n}\int_{\underline{t}_{\ell}}^{\bar{t}_{\ell}} \sigma^2_s ds \Bigg|= o(1).
\end{equation}
With $Y_t= (\sigma\!\cdot\! W)_t$, It\^o's formula gives
\begin{equation}\label{e:ito_increment}
(\Delta_i^n (\sigma\!\cdot\! W)\big)^2  = 2 \int_{t_{i-1}}^{t_i}  (Y_s-Y_{t_{i-1}}) dY_s +  \int_{t_{i-1}}^{t_i} \sigma^2_s ds.
\end{equation}
Using \eqref{e:ito_increment}, we have
\begin{align*}
\max_{\frac{k_n}{n}< \ell \leq n-\frac{k_n}{2}}\bigg|\frac{n}{k_n}&\sum_{i=\ell-\frac{k_n}{2}}^{\ell + \frac{k_n}{2}} \big(\Delta_i^n(\sigma\!\cdot\! W)\big)^2  -\frac{n}{k_n}\int_{\underline{t}_{\ell}}^{\bar{t}_{\ell}}  \sigma^2_s ds \bigg|\\
 & \leq 2\max_{\frac{k_n}{n}< \ell \leq n-\frac{k_n}{2}}\bigg|\frac{n}{k_n}\sum_{i=\ell-\frac{k_n}{2}}^{\ell + \frac{k_n}{2}}\int_{t_{i-1}}^{t_i}  (Y_s-Y_{t_{i-1}}) dY_s\bigg|.\label{e:uniform_qv_bound} %
\end{align*}
With $g_\ell(s)=\sum_{i=\ell-\frac{k_n}{2}}^{\ell + \frac{k_n}{2}}(Y_s-Y_{t_{i-1}}){\bf 1}_{(t_{i-1},t_i]}(s)$, since $[Y,Y]_t=\int_0^t \sigma_s^2ds$, we obtain, for any $p\geq 2$
\begin{align*}
\bE  \bigg|\frac{n}{k_n}\sum_{i=\ell-\frac{k_n}{2}}^{\ell + \frac{k_n}{2}}\int_{t_{i-1}}^{t_i}  (Y_s-Y_{t_{i-1}}) dY_s \bigg|^p& = \Big(\frac{n}{k_n}\Big)^p\bE  \bigg|\int_0^1  g_\ell(s)dY_s \bigg|^p\\
&\leq K \Big(\frac{n}{k_n}\Big)^p\bE\bigg( \int_0^1 g_\ell^2(s)\sigma_s^2ds\bigg)^{p/2}\\% \bE  \bigg(\frac{1}{k_n}\sum_{i=\ell-\frac{k_n}{2}}^{\ell + \frac{k_n}{2}}\frac{1}{h}\int_{t_{i-1}}^{t_i}  (Y_s-Y_{t_{i-1}}) dY_s \bigg)^4\\
&\leq K \Big(\frac{n}{k_n}\Big)^p\Bigg(\sum_{i=\ell-\frac{k_n}{2}}^{\ell + \frac{k_n}{2}} \int_{t_{i-1}}^{t_i } \bE |Y_s-Y_{t_{i-1}}|^2 ds\Bigg)^{p/2}\\
 &\leq K \Big(\frac{n}{k_n}\Big)^p k_n^{p/2 -1} h^{p/2-1}\sum_{i=\ell-\frac{k_n}{2}}^{\ell + \frac{k_n}{2}} \int_{t_{i-1}}^{t_i } \bE |Y_s-Y_{t_{i-1}}|^p ds\\
 & \leq K k_n^{-p/2},
\end{align*}
which holds uniformly in $\ell$.  Thus,
\begin{align*}
\bP\bigg( 2\max_{\frac{k_n}{n}< \ell \leq n-\frac{k_n}{2}}&\frac{n}{k_n}\bigg|\sum_{i=\ell-\frac{k_n}{2}}^{\ell + \frac{k_n}{2}} \int_{t_{i-1}}^{t_i}  (Y_s-Y_{t_{i-1}}) dY_s \bigg|>\delta\bigg) \\
& \leq \sum_{\ell=\frac{k_n}{2}+1}^{n-\frac{k_n}{2}}\bP\bigg( 2\frac{n}{k_n}\bigg|\sum_{i=\ell-\frac{k_n}{2}}^{\ell+\frac{k_n}{2}} \int_{t_{i-1}}^{t_i}  (Y_s-Y_{t_{i-1}}) dY_s \bigg|>\delta\bigg) \\
& \leq (\delta/2)^{-p} \sum_{\ell=\frac{k_n}{2}+1}^{n-\frac{k_n}{2}} \bE  \bigg|\frac{n}{k_n}\sum_{i=\ell-\frac{k_n}{2}}^{\ell+\frac{k_n}{2}}\int_{t_{i-1}}^{t_i}  (Y_s-Y_{t_{i-1}}) dY_s \bigg|^{p} \\
& \leq K nk_n^{-p/2}.
\end{align*}
By taking $p>2/a$ we get \eqref{e:maxVol_asconverge}.
\end{proof}

\begin{lemma}\label{l:mod_of_cont} %
$$
\limsup_{h\to 0} \sup_{0\leq t\leq 1-h }\frac{| \int_t^{t+h} \sigma_u dW_u |}{\sqrt{2 \big(\sup_{u\in[t,t+h)}\sigma^2_u\big)h \log (1/h)}}\leq 1, \quad \textnormal{a.s.}
$$
In particular, with $\overline \sigma_{i,n}^2 = \sup_{t\in[t_{i-1},t_i)}\sigma^2_t$,
$$
\limsup_{n\to \infty} \max_{i=1\ldots,n}\frac{| \Delta_i^n (\sigma \!\cdot\! W)|}{\sqrt{2\overline \sigma_{i,n}^2 h_n \log (1/h_n)}}\leq 1, \quad \textnormal{a.s.}
$$
\end{lemma}
\begin{proof}
Define %
$$
D(h) =  \sup_{0\leq t\leq 1-h }\frac{| \int_t^{t+h} \sigma_u dW_u |}{\sqrt{ \sup_{u\in[t,t+h)}\sigma^2_u}}.
$$
It suffices to show that for every small $\eta>0$,
$$
\limsup_{h\to 0} D(h) \leq 1+3\eta.
$$
So, for any $m<1/h$, define
\begin{align*}
D(m,h) &= \max_{j=1,\ldots,m} \sup_{t\in\big[\frac{(j-1)}{m},\frac{j}{m} \wedge (1-h)\big)}\frac{| \int_t^{t+h} \sigma_u dW_u |}{\sqrt{\sup_{u\in I_{j,m,h}}\sigma^2_u}},
\end{align*}
where $I_{j,m,h}=\big[\frac{(j-1)}{m},(\frac{j}{m} + h)\wedge 1\big).$
We first establish for appropriate subsequence $h_\ell\to 0$, and $m=m(h_\ell) \to\infty$, %
\begin{equation*}%
\limsup_{\ell\to\infty} \frac{D\big(h_\ell,m(h_\ell)\big)}{\sqrt{2 h_\ell \log (1/h_\ell)}}\leq 1+\eta, \quad \textnormal{a.s.}
\end{equation*}
So, let $0<\eta<1$ be given, and let $\theta = \frac{(1+\eta)^2}{(1+\eta/2)^2}>1$. With $C_t= \int_0^t \sigma^2_sds$, the Dambis-Dubins-Schwarz Theorem (e.g., Theorem 4.6 in \cite{karatzas:shreve:1998}) implies that a.s.,
$$
(\sigma \!\cdot\! W)_t = B_{C_t},\quad t \geq 0,$$
where $B_{t}$ a Brownian motion. Also, since $\sigma$ is bounded above and below, we may take a $J,J_0>0$ such that $\theta^J\geq \overline \sigma^2$ and $\theta^{-J_0}<\inf \sigma_t^2$ a.s.  %
For any $-J_0\leq \ell \leq J$, on the event %
$\big\{\sup_{u\in I_{j,m,h}}\sigma^2_u\in\big(\theta^{\ell -1},\theta^{\ell}\big]\big\}$ observe $C_t\leq u,v\leq C_{t+h}$ implies $|u-v|\leq h\theta^{\ell}$ for each $t\in [\frac{(j-1)}{m},\frac{j}{m})$. This gives  %
\begin{align}
 &\bP\Bigg(\sup_{t\in[\frac{(j-1)}{m},\frac{j}{m})}\frac{| B_{C_{t+h}}-B_{C_t} |}{\sqrt{2\big(\sup_{u\in I_{j,m,h}}\sigma^2_u\big)h\log (1/h)}} >1+\eta,~ \sup_{u\in I_{j,m,h}}\sigma^2_u\in\big(\theta^{\ell -1},\theta^{\ell}\big]\Bigg)\notag\\ %
 & \leq  \bP\Bigg(\sup_{t\in[\frac{(j-1)}{m},\frac{j}{m})}\frac{\sup_{C_{t }\leq u,v \leq C_{t+h}}|B_u-B_v|}{\sqrt{2\big(\sup_{u\in I_{j,m,h}}\sigma^2_u\big)h\log (1/h)}}  >1+\eta, ~\sup_{u\in I_{j,m,h}}\sigma^2_u\in\big(\theta^{\ell -1},\theta^{\ell}\big]\Bigg) \notag\\
& \leq \bP\Bigg(\sup_{|u-v|\leq h\theta^{\ell}}\sup_{ 0\leq u,v\leq \theta^J}\frac{|B_u-B_v|}{\sqrt{2\theta^{\ell-1} h \log (1/h)}} >1+\eta\Bigg) \notag\\
& =\bP\Bigg(\sup_{|u-v|\leq h}\sup_{~0\leq s,t\leq \theta^J/\theta^{\ell}}\frac{\sqrt{\theta^{\ell}}|B_s-B_t|}{\sqrt{2 \theta^{\ell-1} h \log (1/h)}} >1+\eta\Bigg)\notag\\
& \leq \bP\Bigg(\sup_{|s-t|\leq h}\sup_{~0\leq s,t\leq \theta^{J-\ell}}\frac{|B_s-B_t|}{\sqrt{2 h \log (1/h)}} > (1+\eta)/\sqrt{\theta}\Bigg)\notag\\
& \leq \bP\Bigg(\sup_{|s-t|\leq h}\sup_{~0\leq s,t\leq \theta^{J+J_0}}\frac{|B_s-B_t|}{\sqrt{2 h \log (1/h)}} > (1+\eta)/\sqrt{\theta}\Bigg)\notag\\
& = \bP\Bigg(\sup_{|s-t|\leq h}\sup_{~0\leq s,t\leq \theta^{J+J_0}}\frac{ |B_s-B_t|}{\sqrt{2 h \log (1/h)}} >1+\eta/2\Bigg),\label{e:modcontbound1}
\end{align}
where on the 4th line we used self-similarity of $B_t$.  Summing over $\ell$ and applying a union bound we obtain
\begin{align}
\bP\Bigg(\frac{D(h,m)}{\sqrt{2 h \log (1/h)}} >1+\eta\Bigg)%
& \leq K m \,\bP\Bigg(\sup_{|s-t|\leq h}\sup_{~0\leq s,t\leq \theta^{J+J_0}}\frac{ |B_t-B_s|}{\sqrt{2 h \log (1/h)}} >1+\eta/2\bigg).
 \label{e:modcontbound1}
\end{align}
Now, recall (c.f. Lemma 1.1.1 in \cite{csorgo:revesz:1981}), that for every $M,\delta>0$, there exists a $K_0=K_0(\delta,M)$ such that for every positive $v>0$,
$$
\bP\bigg(\sup_{\,0\leq s,t\leq M,\,|s-t|\leq h}|B_t-B_s| \geq v \sqrt h\bigg) \leq \frac{K_0}{h}\exp\bigg\{\frac{-v^2}{2+\delta}\bigg\}.
$$
So, taking $\delta>0$ small enough so that $\frac{2(1+\eta/2)^2}{2+\delta}=1+\kappa $ for some $0<\kappa<1$, and taking $M= \theta^{J+J_0}$, and $v= (1+\eta/2)\sqrt{2 \log(1/h)},$ we get
\begin{align}
\bP\Bigg(\sup_{|s-t|\leq h,~ 0\leq s,t\leq \theta^{J+J_0}}\frac{|B_t-B_s|}{\sqrt{2 h \log (1/h)}} >1+\eta/2\Bigg) &\leq \frac{K_0}{h}\exp\bigg\{\frac{-2(1+\eta/2)^2\log(1/h)}{2+\delta}\bigg\}\notag\\
& \leq K_0 h^{\kappa}\label{e:modcontbound3}
\end{align}
Thus, with $m=m(h)=\lfloor h^{-\kappa/2}\rfloor$, combining \eqref{e:modcontbound1} and \eqref{e:modcontbound3} we get%
\begin{equation}\label{e:D(h,m)_bound1}
\bP\Bigg(\frac{D\big(h,m(h)\big)}{\sqrt{2 h \log (1/h)}} >1+\eta\Bigg) \leq K m h^{\kappa}\leq K  h^{\kappa/2}.
\end{equation}
Now, for some integer $r_0>2/\kappa$ consider the subsequence $h_\ell$ defined by
$$
h_\ell=  \ell^{-r_0}.
$$
Using \eqref{e:D(h,m)_bound1} we obtain
\begin{align*}
\sum_{\ell=2}^\infty\bP\Bigg(\frac{D\big(h_\ell,m(h_\ell)\big)}{\sqrt{2 h_\ell \log (1/h_\ell)}} >1+\eta\Bigg) \leq  \sum_{\ell=2}^\infty \ell^{-(r_0\kappa/2)}<\infty. %
\end{align*}
Thus, the Borel-Cantelli lemma gives
\begin{equation*}%
\limsup_{\ell\to\infty} \frac{D\big(h_\ell,m(h_\ell)\big)}{\sqrt{2 h_\ell \log (1/h_\ell)}}\leq 1+\eta, \quad \textnormal{a.s.}
\end{equation*}
This implies for a.s. $\omega$ we can find a large $\ell_0(\omega)$ such that
\begin{equation}\label{e:ell_0(omega)}
\frac{D\big(h_\ell,m(h_\ell)\big)}{\sqrt{2 h_\ell \log (1/h_\ell)}}< 1+2\eta,\quad \ell\geq \ell_0(\omega).
\end{equation}
We now proceed to establish, for some $\ell_1(\omega)<\infty$, a.s.,
\begin{equation}\label{e:ell_1(omega)}
\frac{D\big(h_\ell\big)}{\sqrt{2 h_\ell \log (1/h_\ell)}}\leq\bigg(\frac{1 + 3\eta}{1+2\eta}\bigg)\frac{D\big(h_\ell,m(h_\ell)\big)}{\sqrt{2 h_\ell \log (1/h_\ell)}},\quad \ell\geq \ell_1(\omega).
\end{equation}
Recall (c.f. \eqref{e:cadlag_uniform_bound}) we may take $\nu=\nu(\omega)$ large enough so that so that for all $u_1\in[s_0,s_1)$,\ldots, $u_\nu\in[s_{\nu-1},s_\nu)$,
$$
\max_{i=1,\ldots,\nu} \sup_{t \in[s_{i-1},s_i)} \sigma^2_t -  \sigma^2_{u_i}  \leq \delta' \theta^{-J_0},
$$
where $\delta'>0$ is chosen small enough so that  $\sqrt{1-\delta'}=\frac{1+2\eta}{1+3\eta}$. This gives, for any $u \in [s_{i-1},s_i)$, %
\begin{align*}
\sigma^2_{u}&= \sup_{t \in [s_{i-1},s_i)}\sigma^2_t -\Big(\sup_{t \in [s_{i-1},s_i)} \sigma^2_t - \sigma^2_{u}\Big)\\
& \geq  \sup_{t \in [s_{i-1},s_i)} \sigma^2_t - \delta' \theta^{-J_0}\\
& \geq (1-\delta')\sup_{t \in [s_{i-1},s_i)} \sigma^2_t,%
\end{align*}
since $\sup_{t \in I_i} \sigma^2_t\geq \inf_t \sigma^2_t \geq \theta^{-J_0}$.  So, taking $m_0(\omega)$ large enough so that
$$
\frac{2}{m_0(\omega)} \leq \min_{i=1\ldots \nu} s_{i}-s_{i-1},
$$
for every $m\geq m_0(\omega)$ there is at most one $s_i$ in each interval $ I_{j,m,h}\subseteq[\frac{j-1}{m},\frac{j+1}{m})$. In the case that for some $i$, $I_{j,m,h}\subseteq [s_{i-1},s_i)$, we have
$$
\sup_{t\in\big [\frac{j-1}{m},\frac{j}{m} \wedge(1-h)\big) } \frac{| \int_t^{t+h} \sigma_u dW_u |}{\sqrt{ \sup_{u\in[t,t+h)}\sigma^2_u}} \leq  \sup_{t\in\big [\frac{j-1}{m},\frac{j}{m} \wedge(1-h)\big) }\frac{| \int_t^{t+h} \sigma_u dW_u |}{\sqrt{(1-\delta')\sup_{u\in I_{j,m,h}}\sigma^2_u}}.
$$
In the other case, i.e. for some $i$, $s_i \in I_{j,m,h}$ (i.e., $s_{i-1}< \frac{j-1}{m}< s_i < (\frac{j}{m}+ h) \wedge 1 <s_{i+1}$), we have:
\begin{align*}
&\sup_{t\in \big[\frac{j-1}{m},\frac{j}{m} \wedge (1-h)\big) } \frac{| \int_t^{t+h} \sigma_u dW_u |}{\sqrt{ \sup_{u\in[t,t+h)}\sigma^2_u}} \\
&=\sup_{t\in \big[\frac{j-1}{m},\frac{j}{m} \wedge (1-h)\big)  } \frac{| \int_t^{t+h} \sigma_u dW_u |}{\sqrt{ \max\{\sup_{u\in[t,(t+h) \wedge s_i)}\sigma^2_u,\sup_{u\in[(t+h) \wedge s_i, t+h)}\sigma^2_u\big\}}} \\
&\leq \sup_{t\in\big[\frac{j-1}{m},\frac{j}{m} \wedge (1-h)\big) } \frac{| \int_t^{t+h} \sigma_u dW_u |}{\sqrt{ \max\{(1-\delta')\sup_{u\in[\frac{j-1}{m},   s_i)}\sigma^2_u,(1-\delta')\sup_{u\in[  s_i,(\frac{j}{m} + h )\wedge 1)}\sigma^2_u,\big\}}} \\
& =   \sup_{t\in\big[\frac{j-1}{m},\frac{j}{m} \wedge (1-h)\big)}\frac{| \int_t^{t+h} \sigma_u dW_u |}{\sqrt{(1-\delta')\sup_{u\in I_{j,m,h}}\sigma^2_u}}
\end{align*}
Hence, for all $m\geq m_0(\omega)$, and all $h<\frac{1}{m}$,
\begin{align*}
D(h) &=  \sup_{0\leq t\leq 1-h }\frac{| \int_t^{t+h} \sigma_u dW_u |}{\sqrt{ \sup_{u\in[t,t+h)}\sigma^2_u}}\\
&=\max_{{j=1,\ldots,m}}\sup_{t\in \big[\frac{j-1}{m},\frac{j}{m} \wedge (1-h)\big) } \frac{| \int_t^{t+h} \sigma_u dW_u |}{\sqrt{ \sup_{u\in[t,t+h)}\sigma^2_u}} \\
&\leq  \frac{1}{\sqrt{1-\delta'}}\max_{j=1,\ldots,m}\sup_{t\in \big[\frac{j-1}{m},\frac{j}{m} \wedge (1-h)\big) }\frac{| \int_t^{t+h} \sigma_u dW_u |}{\sqrt{\sup_{u\in[\frac{(j-1)}{m},\frac{j}{m})}\sigma^2_u}}\\
& = \frac{1 + 3\eta}{1+2\eta} D(h,m).
\end{align*}
So, if we choose $\ell_1(\omega)\geq \ell_0(\omega)$ large enough so that $m(h_\ell)> m_0(\omega)$ for all $\ell \geq \ell_1(\omega)$, expression \eqref{e:ell_1(omega)} holds.  From \eqref{e:ell_0(omega)} we then have %
$$
\frac{D\big(h_\ell\big)}{\sqrt{2 h_\ell \log (1/h_\ell)}}\leq\bigg(\frac{1 + 3\eta}{1+2\eta}\bigg)  \frac{D\big(h_\ell,m(h_\ell)\big)}{\sqrt{2 h_\ell \log (1/h_\ell)}}< 1+3\eta,\quad \ell\geq \ell_1(\omega),
$$
i.e.,
$$
\limsup_{\ell\to\infty }\frac{D\big(h_\ell\big)}{\sqrt{2 h_\ell \log (1/h_\ell)}} \leq 1+ 3\eta,\quad \text{a.s.}
$$
Now, for any $h$ with $h_{\ell+1}\leq h <h_\ell$, we have
$$
\frac{D(h)}{2h\log (1/h)} \leq \frac{D\big(h_\ell\big)}{\sqrt{2 h_{\ell+1} \log (1/h_{\ell+1})}}= \frac{D\big(h_\ell\big)}{\sqrt{2 h_{\ell} \log (1/h_{\ell})}} \frac{\sqrt{2 h_{\ell} \log (1/h_{\ell})}}{\sqrt{2 h_{\ell+1} \log (1/h_{\ell+1})}}%
$$
which gives
$$
\limsup_{h\to 0} \frac{D(h)}{2h\log (1/h)} \leq \limsup_{\ell\to\infty }\frac{D\big(h_\ell\big)}{\sqrt{2 h_\ell \log (1/h_\ell)}} \limsup_{\ell\to\infty}\frac{\sqrt{2 h_{\ell} \log (1/h_{\ell})}}{\sqrt{2 h_{\ell+1} \log (1/h_{\ell+1})}}  \leq 1 + 3\eta,\quad \text{a.s.,}
$$
which completes the proof.
\end{proof}

\begin{acks}[Acknowledgments]
 The authors would like to thank an anonymous referee whose comments and suggestions helped to substantially improve the paper.
\end{acks}

\bibliographystyle{imsart-nameyear} %
\bibliography{fixpoint_Bib.bib}       %

\end{document}